\documentclass[11pt,reqno,fleqn]{amsart}

\usepackage[T1]{fontenc}
\usepackage{cmap}          
\usepackage[utf8]{inputenc}
\usepackage{amsmath,amssymb,amsthm,mathtools}
\usepackage{bm}
\usepackage{enumitem}
\usepackage[numbers,sort&compress]{natbib}

\usepackage{xcolor}
\definecolor{DarkSlateGrey}{rgb}{0.18,0.31,0.31}
\usepackage{hyperref}
\hypersetup{
  colorlinks=true,
  linkcolor=DarkSlateGrey,
  citecolor=DarkSlateGrey,
  urlcolor=DarkSlateGrey,
  pdftitle={On the local well-posedness of the Benjamin-Ono-Zakharov-Kuznetsov equation},
  pdfauthor={Ailton C. Nascimento},
  pdfsubject={Local well-posedness of the Benjamin-Ono-Zakharov-Kuznetsov equation},
  pdfkeywords={BO-ZK equation, local well-posedness, maximal-function estimate, local smoothing, normal form}
}

\usepackage{aliascnt}
\usepackage[nameinlink,capitalise,noabbrev]{cleveref}

\newtheorem{theorem}{Theorem}[section]
\newaliascnt{proposition}{theorem}
\newtheorem{proposition}[proposition]{Proposition}
\aliascntresetthe{proposition}
\newaliascnt{lemma}{theorem}
\newtheorem{lemma}[lemma]{Lemma}
\aliascntresetthe{lemma}
\newaliascnt{corollary}{theorem}
\newtheorem{corollary}[corollary]{Corollary}
\aliascntresetthe{corollary}
\newaliascnt{claim}{theorem}

\aliascntresetthe{claim}
\theoremstyle{definition}
\newaliascnt{definition}{theorem}
\newtheorem{definition}[definition]{Definition}
\aliascntresetthe{definition}
\theoremstyle{remark}
\newaliascnt{remark}{theorem}
\newtheorem{remark}[remark]{Remark}
\aliascntresetthe{remark}
\numberwithin{equation}{section}

\newcommand{\R}{\mathbb R}
\newcommand{\Z}{\mathbb Z}
\newcommand{\Sscr}{\mathcal S}
\newcommand{\supp}{\operatorname{supp}}
\newcommand{\sgn}{\operatorname{sgn}}
\newcommand{\Op}{\operatorname{Op}}
\newcommand{\dd}{\,d}
\newcommand{\eps}{\varepsilon}
\newcommand{\ip}[2]{\left\langle #1,#2\right\rangle}
\newcommand{\norm}[1]{\left\lVert #1\right\rVert}
\newcommand{\abs}[1]{\left\lvert #1\right\rvert}
\newcommand{\B}{\mathcal B}
\newcommand{\Ncal}{\mathcal N}

\begin{document}

\title[Local well-posedness for BO--ZK]
      {On the local well-posedness of the\\
       Benjamin--Ono--Zakharov--Kuznetsov equation}

\author{Ailton C. Nascimento}
\address{Departamento de Matem\'atica, Universidade Federal do Piau\'i,
  Campus Universit\'ario Ministro Petr\^onio Portella, Ininga,
  64049-550 Teresina, PI, Brazil}
\email{ailton.nascimento@ufpi.edu.br}



\begin{abstract}
	We study the Cauchy problem for the Benjamin--Ono--Zakharov--Kuznetsov equation on $\mathbb R^2$. Following the strategy of Kenig and Ziesler, we establish new maximal-function estimates adapted to the BO--ZK equation and use them to implement the Kenig--Koenig method. As a result, we improve the best previously known isotropic result of Nascimento (2020), lowering the local well-posedness threshold from $s>5/4$ to $s>19/16$. At the BO--ZK endpoint, the resulting isotropic data class also contains the anisotropic $E^{5/4+}$ class of the preceding theory. On bounded subsets of $H^s(\mathbb R^2)$, the lifespan may be chosen so that
$
	T\gtrsim_s \bigl(1+\norm{u_0}_{H^s}\bigr)^{-8}.
$
	The proof combines a sharp dyadic mixed maximal-function estimate with an anisotropic local-smoothing mechanism that exploits the complementary behavior of the longitudinal and transverse group velocities. In particular, transverse dispersion compensates for the degeneration of longitudinal smoothing near the characteristic region. Together with refined short-time Strichartz estimates and a modified energy argument, these ingredients close the nonlinear estimates at the stated regularity. Existence, uniqueness, and continuous dependence on the initial data are then established in the corresponding solution class. The resulting threshold reflects the present optimization of the method and is not claimed to be sharp.
\end{abstract}

\maketitle

\section{Introduction}

We study the Cauchy problem
\begin{equation}\label{IVP}
\begin{cases}
 u_t+\mathcal H_xu_{xx}+u_{xyy}+uu_x=0,
   &(x,y)\in\R^2,\ t\in\R,\\
 u(x,y,0)=u_0(x,y),
\end{cases}
\end{equation}
where $\mathcal H_x$ denotes the Hilbert transform in the $x$ variable.  With the Fourier convention
\[
 \widehat f(\xi,\eta)=\int_{\R^2}e^{-i(x\xi+y\eta)}f(x,y)\dd x\dd y,
\]
the linear group is
\begin{equation}\label{eq:group}
 \widehat{U(t)f}(\xi,\eta)
 =e^{it\omega(\xi,\eta)}\widehat f(\xi,\eta),
 \qquad
 \omega(\xi,\eta)=\xi(\eta^2-|\xi|).
\end{equation}
The equation is a two-dimensional extension of the Benjamin--Ono equation in which the weak transverse dispersion is of Zakharov--Kuznetsov type.  It arises in models of two-dimensional internal waves and in the evolution of BO-type solitary waves under weak lateral dispersion; see, for instance, \cite{Benjamin,Latorre,Ono}.

\subsection{Background and prior results}\label{subsec:background}

The solitary-wave theory for generalized BO--ZK equations was
developed in \cite{EsfahaniPastorInstability,EsfahaniPastorBona}.
The first of these works establishes instability in the corresponding
supercritical regime, whereas the second classifies the parameter
range in which solitary waves exist and studies their regularity,
decay, and orbital stability in the stable regime.  The variational
analysis in \cite{EsfahaniPastorBona} also provides an anisotropic
Gagliardo--Nirenberg inequality adapted to the BO--ZK energy.  Its
optimal constant was subsequently characterized in
\cite{EsfahaniPastorSharp} in terms of the associated ground states,
leading in particular to uniform a priori bounds for smooth solutions
in the energy space.

Unique continuation was investigated in
\cite{EsfahaniPastorUCP,CunhaPastorWeighted}.  In
\cite{EsfahaniPastorUCP}, a sufficiently regular solution whose
support remains in a fixed rectangle throughout its lifespan is shown
to vanish identically.  Cunha and Pastor
\cite{CunhaPastorWeighted} strengthened this rigidity mechanism: a
sufficiently smooth local solution with the prescribed algebraic
decay at three distinct times must also be trivial.

On periodic domains, related recent work concerns control and
stabilization.  In \cite{NascimentoControl2026}, a localized,
mean-preserving damping acting in the Benjamin--Ono direction is used
to obtain observability, semi-global exponential stabilization, and
local exact controllability for BO--ZK on $\mathbb T^2$.  The
dispersion-generalized periodic model is considered in
\cite{NascimentoDGBOZK2026}; in the parameter range specified there,
a dissipation-normalized Bourgain-space argument yields global
$L^2$ well-posedness and small-data exponential stabilization.  These
periodic results are analytically distinct from the Euclidean Cauchy
problem studied here, but they further illustrate the role played by
the anisotropic longitudinal smoothing.

We now summarize the well-posedness theory on $\R^2$.  Cunha and
Pastor \cite{CunhaPastorWeighted} first proved local well-posedness in
$H^s(\R^2)$ for $s>2$ by parabolic regularization, together with the
conservation of mass and energy.  They also obtained an anisotropic
theory in $H^{s_1,s_2}(\R^2)$ for $s_2>2$ and $s_1\ge s_2$.  Esfahani
and Pastor \cite{EsfahaniPastor}, following the strategy of
Molinet--Saut--Tzvetkov \cite{MolinetSautTzvetkov}, proved that the
data-to-solution map fails to be $C^2$ at the origin in every
$L^2$-based anisotropic Sobolev space.  Thus a direct Picard iteration
on the Duhamel formula cannot provide the well-posedness theory.

Cunha and Pastor \cite{CunhaPastorLow} subsequently lowered the
isotropic threshold to $s>11/8$ by adapting the Koch--Tzvetkov
refinement of the energy method \cite{KochTzvetkov}.  For the
dispersion-generalized equation
\begin{equation}\label{eq:DGBOZK}
 \partial_tu-D_x^\alpha\partial_xu+\partial_x\partial_y^2u
 =u\partial_xu,
 \qquad 1\le\alpha\le2,
\end{equation}
Ribaud and Vento \cite{RibaudVento} introduced the anisotropic spaces
\[
 E_\alpha^r
 =\left\{f:\ \langle |\xi|^\alpha+\eta^2\rangle^r
 \widehat f(\xi,\eta)\in L^2(\R^2)\right\}
 =H^{\alpha r,2r}(\R^2)
\]
and proved local well-posedness for $r>2/\alpha-3/4$, as well as
global well-posedness in the energy space $E_\alpha^{1/2}$ when
$\alpha>8/5$.  At the BO--ZK endpoint $\alpha=1$, this yields the
anisotropic threshold $r>5/4$; below we abbreviate $E^r:=E_1^r$.

The best previously known isotropic result is due to Nascimento
\cite{Nascimento2020}.  For every $s>5/4$ and
$\phi\in H^s(\R^2)$, it provides a unique solution
\[
 u\in C([0,T];H^s(\R^2)),
 \qquad
 u,\ \partial_xu\in L^1([0,T];L^\infty(\R^2)),
\]
with continuous dependence on the initial data.  Its proof combines a
refined Strichartz estimate \cite{KenigKPI} with the energy method and also establishes
propagation of regularity.  The present paper lowers this isotropic
threshold from $5/4$ to $19/16$.

Our purpose is to implement, in two dimensions, the refinement introduced by Kenig and Koenig for the Benjamin--Ono equation \cite{KenigKoenig}.  The two essential ingredients in that argument are a short-time Strichartz estimate and a maximal-function/local-smoothing pairing.  The first component gives a natural candidate threshold.  The second is genuinely anisotropic because the group velocities are
\begin{equation}\label{eq:velocities}
 v_x(\xi,\eta)=\partial_\xi\omega=\eta^2-2|\xi|,
 \qquad
 v_y(\xi,\eta)=\partial_\eta\omega=2\xi\eta.
\end{equation}
The longitudinal velocity vanishes on the characteristic parabola
\begin{equation}\label{eq:Gamma}
 \Gamma=\{(\xi,\eta):\eta^2=2|\xi|\},
\end{equation}
while on that set $|v_y|\simeq |\xi|^{3/2}$.  Thus the transverse smoothing becomes stronger precisely where the longitudinal smoothing degenerates.

The main result is the following.

\begin{theorem}\label{thm:main}
Let $s>19/16$ and $u_0\in H^s(\R^2)$.  There exists
\[
 T=T(\norm{u_0}_{H^s})>0
\]
and a solution of \eqref{IVP} such that
\begin{equation}\label{eq:solutionclass}
 u\in C([0,T];H^s(\R^2)),
 \qquad
 \nabla u\in L^2([0,T];L^\infty(\R^2)).
\end{equation}
The solution is unique in the class \eqref{eq:solutionclass}.
For every $R_0>0$ there exists $T_{R_0}>0$ such that the data-to-solution map is continuous from the ball
\[
 \{u_0\in H^s(\R^2):\norm{u_0}_{H^s}\le R_0\}
\]
into $C([0,T_{R_0}];H^s(\R^2))$.  Moreover, one may take
\begin{equation}\label{eq:lifespan}
 T_{R_0}\ge c_s(1+R_0)^{-8}.
\end{equation}
\end{theorem}

The exponent $19/16$ is produced by a refined Strichartz argument based on a frequency-dependent partition of the time interval, in the spirit of Kenig and of Linares--Pilod--Saut \cite{LinaresPilodSaut}.  The BO--ZK Strichartz pairs satisfy
\begin{equation}\label{eq:admissible-intro}
 p>\frac83,
 \qquad
 \frac1q+\frac{4}{3p}=\frac12.
\end{equation}
After localizing a frequency $N$ solution to time intervals of length $N^{-1/2}$, one obtains the exponents
\begin{equation}\label{eq:ab-intro}
 a(p)=2-\frac{13}{6p},
 \qquad
 b(p)=\frac32-\frac{13}{6p},
 \qquad a(p)-b(p)=\frac12.
\end{equation}
Letting $p\downarrow8/3$ gives
\[
 a(p)\downarrow\frac{19}{16},
 \qquad
 b(p)\downarrow\frac{11}{16}.
\]

\subsection*{Scaling, criticality, and the status of the threshold}

Next, we discuss some issues regarding the exponent $19/16$, since for \eqref{IVP} the scaling heuristic and the
genuine obstruction to low regularity are two different things. In fact, equation \eqref{IVP} admits the one-parameter symmetry
\begin{equation}\label{eq:scaling-intro}
 u_\lambda(x,y,t)=\lambda\,u\bigl(\lambda x,\lambda^{1/2}y,\lambda^2t\bigr),
 \qquad\lambda>0,
\end{equation}
and this is the only such scaling: writing
$u_\lambda=\lambda^{a}u(\lambda x,\lambda^{b}y,\lambda^{c}t)$ and using
that $\mathcal H_x$ is homogeneous of degree zero, the four terms of
\eqref{IVP} acquire the factors $\lambda^{a+c}$, $\lambda^{a+2}$,
$\lambda^{a+1+2b}$ and $\lambda^{2a+1}$, which are equal only for
$a=1$, $b=1/2$, $c=2$.  Equivalently, the phase in \eqref{eq:group} is
quasi-homogeneous of degree two,
$\omega(\lambda\xi,\lambda^{1/2}\eta)=\lambda^2\omega(\xi,\eta)$.

If $u_0\in\Sscr(\R^2)$ has $\widehat{u_0}$ supported away from
$\{\xi=0\}$, then $\widehat{u_{0,\lambda}}$ lives where
$|\xi|\simeq\lambda$ and $|\eta|\lesssim\lambda^{1/2}$, so that
$(\xi^2+\eta^2)^{s}\simeq\lambda^{2s}$ on its support; taking into
account the Jacobian $\lambda^{3/2}$ one obtains
\begin{equation}\label{eq:Hs-scaling-intro}
 \norm{u_{0,\lambda}}_{\dot H^{s}(\R^2)}
 \simeq\lambda^{\,s+\frac14}\norm{u_0}_{L^2(\R^2)},
 \qquad\lambda\to\infty .
\end{equation}
The scaling-critical index is therefore
\begin{equation}\label{eq:sc-intro}
 s_c=-\tfrac14 .
\end{equation}
On the family \eqref{eq:scaling-intro} one has $|\xi|\simeq\eta^2\simeq\lambda$,
so the weight $\langle|\xi|+\eta^2\rangle^{r}$ of the anisotropic space
$E^{r}$ yields the same exponent, and $E^{r}$ has the
same critical index $-1/4$.  The estimate
\eqref{eq:Hs-scaling-intro} is also the source of the lifespan
\eqref{eq:lifespan}: it is the homogeneous counterpart of the
inhomogeneous scaling bound proved in \cref{sec:proof-main}, where the
choice $\lambda^{1/4}R_0\simeq\delta_s$ gives
$\lambda\simeq R_0^{-4}$ and hence $T\simeq\lambda^2\simeq R_0^{-8}$.

Two consequences should be stated explicitly. First, the relation between the isotropic and anisotropic data classes
is one-sided.  For $r\ge0$ one has $E^{r}\hookrightarrow H^{r}$,
whereas in the converse direction only
$H^{s}\hookrightarrow E^{s/2}$ holds, because the weight
$\langle|\xi|+\eta^2\rangle^{r}$ requires $2r\le s$ along
$\{\xi=0\}$.  Consequently,
\[
 E^{5/4+}\hookrightarrow H^{5/4+}\subsetneq H^{19/16+},
\]
while $H^{19/16}$ embeds only into $E^{19/32}$ and
$19/32<5/4$.  Thus, at the level of admissible initial data,
\cref{thm:main} strictly enlarges the class covered by the
$E^{5/4+}$ theory of \cite{RibaudVento} at $\alpha=1$.  The two
theories nevertheless use different solution and uniqueness classes.

Second, and more importantly, scaling is not the operative obstruction
here.  Writing $\Omega(\theta,\zeta)=\omega(\theta)+\omega(\zeta)-\omega(\zeta+\theta)$
for $\theta=(\alpha,\beta)$ and $\zeta=(\xi,\eta)$, direct expansion
gives on $\{\alpha,\xi>0\}$
\begin{equation}\label{eq:Omega-intro}
 \Omega
 =2\alpha\xi-\xi(2\beta\eta+\beta^2)-\alpha(2\beta\eta+\eta^2),
\end{equation}
so that in the low--high regime $\alpha\ll\xi$ the leading behaviour is
$\Omega=\xi(2\alpha-2\beta\eta-\beta^2)+O(\alpha)$.  The bracket
vanishes along the curve $2\alpha=2\beta\eta+\beta^2$, which survives as
$\alpha\to0$: a nontrivial resonant set persists at arbitrarily small
longitudinal frequency.  This is the mechanism behind
the theorem of Esfahani and Pastor \cite{EsfahaniPastor}.  More
precisely, their counterexample rules out a contraction argument on
the unmodified Duhamel formula whenever that argument would produce a
$C^2$ flow map.  A well-posedness proof must therefore bypass this
direct Picard scheme, for example through compactness and energy
estimates, as here and in
\cite{CunhaPastorLow,KochTzvetkov,Nascimento2020}, short-time Fourier
restriction spaces as in \cite{IonescuKenigTataru,RibaudVento}, or a
gauge transformation.  In particular $19/16$ cannot be compared with a
``bilinear-estimate optimal'' index in the way that Kinoshita's exponent
$-1/4$ for the Zakharov--Kuznetsov equation \cite{Kinoshita2D} can.

Since the counterexample of \cite{EsfahaniPastor} obstructs smoothness
of the flow but not well-posedness in Kato's sense, it does not provide
a rigorous lower bound for the admissible Sobolev exponents.  Scaling
suggests $s_c=-1/4$ as the natural benchmark, but scaling alone does
not imply ill-posedness below that index.  If $s_{\mathrm{opt}}$
denotes the infimum of the admissible isotropic indices, the presently
justified summary is therefore
\begin{equation}\label{eq:bracket-intro}
 s_c=-\tfrac14\quad\text{(formal scaling benchmark)},
 \qquad
 s_{\mathrm{opt}}\le\tfrac{19}{16}.
\end{equation}
No norm-inflation or failure-of-continuity result is currently known
that supplies a matching lower bound.  The exponent $19/16$ is the
value produced by the particular optimization carried out in
\cref{sec:concluding}, where the two-parameter problem behind
\eqref{eq:ab-intro} is recorded; it is not asserted there, or anywhere
below, to be the limit of what arguments of this type can achieve.

\smallskip
The following maximal-function estimate is a principal new linear
ingredient of the paper.  Here $P_N$ denotes a smooth isotropic
Littlewood--Paley projection, so that its symbol is
$\varphi_1(\zeta/N)$ for a fixed $\varphi_1\in C_c^\infty(\R^2)$
supported in $\{1/2\le|\zeta|\le2\}$ and equal to one on
$\{3/4\le|\zeta|\le3/2\}$.

\begin{theorem}\label{thm:maximal}
There exist constants $C,C_0>0$, depending only on the fixed cutoff
functions, such that, for every dyadic
$N\ge2$, every $0<T\le1$, and every $f\in L^2(\R^2)$,
\begin{equation}\label{eq:dyadic-maximal}
 \norm{U(t)P_Nf}_{L_x^2L_{y,T}^\infty}
 +\norm{U(t)P_Nf}_{L_y^2L_{x,T}^\infty}
 \le C N^{1/2}(1+\log N)^{C_0}\norm{P_Nf}_{L^2}.
\end{equation}
One may take $C_0=1$.  Consequently, for every $\eps>0$, the factor on
the right-hand side may be replaced by $C_\eps N^{1/2+\eps}$.  The
power $1/2$ is optimal: no estimate with $N^{1/2-\delta}$, $\delta>0$,
can hold uniformly in $N$.  The theorem does not assert a
logarithm-free endpoint.
\end{theorem}

Two comments on the content of \cref{thm:maximal} are in order. First, the optimality assertion should be read in the strong sense
familiar from the maximal function estimates of Kenig and Ziesler
\cite{KenigZiesler}.  The power $N^{1/2}$ is exactly the cost of the
one-variable Sobolev embedding $L^2\to L_x^2L_y^\infty$ on a frequency
block of size $N$, and it is already forced at $t=0$; see
\cref{rem:sharpness}.  Thus \eqref{eq:dyadic-maximal}
asserts that taking the supremum over the \emph{whole} time interval,
in addition to the supremum over one spatial variable, costs no power
of $N$ beyond that static embedding.  It is in this sense, and not in
the sense of a gain over Sobolev, that the estimate is sharp.  The
mechanism is the product-frequency dispersive decay proved below,
whose strength is governed by the nonuniform curvature of the phase:
the two mixed norms are controlled by
$L^{1/4}$ and $\lambda^{1/2}$, respectively, and both of these are
$O(N^{1/2})$ on an isotropic block, with equality attained on
different product blocks.  Both estimates are used because they pair
with different signed smoothing norms: the $L_x^2L_{y,T}^\infty$
estimate pairs with smoothing in $L_x^\infty L_{y,T}^2$, whereas the
$L_y^2L_{x,T}^\infty$ estimate pairs with smoothing in
$L_y^\infty L_{x,T}^2$.

Second, the significance of \cref{thm:maximal} is not merely linear.
Its $N^{1/2+}$ cost is paired with the half-derivative gain of the
signed microlocal smoothing estimate.  In the nonlinear estimate this
permits the derivative falling on $u\partial_xu$ to be distributed
between the mixed maximal and smoothing factors.  The refined
Strichartz estimate then requires $11/16+$ derivatives on the forcing,
and the additional half derivative in the product estimate gives
precisely
\[
 \frac{11}{16}+\frac12=\frac{19}{16}.
\]
This is the mechanism by which the new maximal estimate lowers the
previous isotropic threshold.

The proof follows the $TT^*$ architecture of Kenig and Ziesler
\cite{KenigZiesler}.  On each open half-plane $\xi>0$ and $\xi<0$,
\begin{equation}\label{eq:hessian-intro}
 \abs{\det D^2\omega(\xi,\eta)}=4(|\xi|+\eta^2).
\end{equation}
The determinant does not vanish away from the origin on either open
half-plane, but its size is not uniform on an isotropic annulus: it is
of order $N^2$ in the strongly transverse region and only of order
$N$ in the parabolic region $|\xi|\simeq N$, $|\eta|\lesssim
N^{1/2}$.  This loss of one curvature power is reflected in the
quantity $L=\lambda+\mu^2$ in the product-frequency kernel estimates
and ultimately in the exponent $1/2$ of the maximal estimate.

The second new component is a continuous microlocal smoothing estimate.  We construct four signed charts.  The $x$ charts cover the region in which $|v_x|$ is elliptic.  The $y$ charts are confined to a band around \eqref{eq:Gamma}, where
\[
 |\xi|\simeq N,
 \qquad |\eta|\simeq N^{1/2},
 \qquad |v_y|\simeq N^{3/2}.
\]
Frequency-adapted weights convert the positive commutator into the norms
\[
 L_x^\infty L^2_{y,t}
 \quad\text{and}\quad
 L_y^\infty L^2_{x,t},
\]
which pair exactly with the two mixed maximal norms in \eqref{eq:dyadic-maximal}.  The only nonstandard chart commutator is supported in the transition band.  Its low frequency $\theta=(\alpha,\beta)$ is split into the tangential region
\[
 |\beta|\lesssim N^{-1/2}|\alpha|
\]
and the normal region.  The tangential contribution is controlled by $\partial_xu$.  On the normal region, the resonance
\[
 \Omega(\theta,\zeta)
 =\omega(\theta)+\omega(\zeta)-\omega(\zeta+\theta)
\]
satisfies
\begin{equation}\label{eq:res-lb-intro}
 |\Omega(\theta,\zeta)|\gtrsim N^{3/2}|\beta|,
\end{equation}
whereas the commutator costs only $N^{1/2}|\beta|$.  Division by the resonance therefore gains one full power of $N$.  The resulting cubic modified local energy has a quartic remainder.  A Fourier-series expansion in the smooth high variable proves the quartic estimate without invoking a flag-paraproduct theorem.

The paper is organized as follows.  In \cref{sec:prelim} we set the
notation and record the energy and product estimates.  The
product-frequency dispersive, Strichartz, and maximal-function
arguments are given in \cref{sec:maximal}.  The refined Strichartz estimate is proved in \cref{sec:refined}.  The frequency-adapted microlocal smoothing estimate and the transition-band normal form occupy \cref{sec:smoothing}.  The coupled nonlinear maximal, smoothing, and energy inequalities are closed in \cref{sec:apriori}.  Finally, \cref{sec:proof-main} contains existence, uniqueness, and continuous dependence.

\section{Preliminaries}\label{sec:prelim}

\subsection{Basic notation}

We write $D_x=-i\partial_x$, $D_y=-i\partial_y$, $D=(-\Delta)^{1/2}$, and $J=(1-\Delta)^{1/2}$.  Mixed norms are ordered from left to right; for example,
\[
 \norm f_{L_x^2L_{y,T}^\infty}
 =\left(\int_\R\sup_{(y,t)\in\R\times[0,T]}|f(x,y,t)|^2\dd x\right)^{1/2}.
\]
Let $P_N$ be a smooth isotropic projection to $|\zeta|\simeq N$, $\zeta=(\xi,\eta)$, with $P_1$ denoting the low-frequency projection.  We also use product projections $P^x_\lambda P^y_\mu$, where $\lambda,\mu\ge1$ are dyadic and the value $1$ includes the corresponding low coordinate.  The notation $A\lesssim B$ allows a constant independent of the dyadic parameters and of $T\le1$.

The linear equation is
\begin{equation}\label{eq:linear}
 w_t-i\omega(D)w=F,
\end{equation}
where $U(t)=e^{it\omega(D)}$.

\subsection{Pseudodifferential operators}\label{subsec:op-notation}

We use the Kohn--Nirenberg quantization and the associated symbolic
calculus; standard references are
\cite{KohnNirenberg,HormanderIII,TaylorPDO}.  If
$a=a(x,y,\xi,\eta)$ is a smooth symbol and $z=(x,y)$,
$\zeta=(\xi,\eta)$, then
\begin{equation}\label{eq:Op-definition}
 \Op(a)f(z)
 =\frac1{(2\pi)^2}
 \int_{\R^2}e^{iz\cdot\zeta}
 a(z,\zeta)\widehat f(\zeta)\dd\zeta.
\end{equation}
When $a$ depends only on $\zeta$, we also write $a(D)$; this is the
Fourier multiplier with symbol $a$.  Multiplication by a function
$h=h(z)$ is denoted by $M_h$.  Thus $D_x$, $D_y$, $J^s$, the
Littlewood--Paley projections, and all chart localizations are special
cases of \eqref{eq:Op-definition}.  A real frequency multiplier is
self-adjoint on $L^2$.

Only finite symbolic expansions are needed.  All symbols are localized
to a specified frequency block, and their orders are understood after
rescaling that block to unit size.  If $a$ and $b$ have orders $m$ and
$m'$, respectively, and satisfy the derivative bounds displayed at the
point of use, then
\begin{align}
 \Op(a)\Op(b)
 &=\Op(ab)+\frac1i\Op(\nabla_\zeta a\cdot\nabla_zb)
   +\Op(r_{m+m'-2}),\label{eq:Op-composition}\\
 \Op(a)^*&=\Op(\overline a)+\Op(r_{m-1}),
 \label{eq:Op-adjoint}\\
 i[\Op(a),\Op(b)]
 &=\Op(\{a,b\})+\Op(r_{m+m'-2}),
 \label{eq:Op-commutator}
\end{align}
where
\[
 \{a,b\}
 =\nabla_\zeta a\cdot\nabla_zb
  -\nabla_za\cdot\nabla_\zeta b
\]
is the Poisson bracket.  The remainders in
\eqref{eq:Op-composition}-\eqref{eq:Op-commutator} are localized to
the same enlarged block and are controlled directly by finitely many
rescaled symbol seminorms.  The $L^2$ boundedness of the order-zero
operators used below is a blockwise form of the
Calder\'on--Vaillancourt theorem \cite{CalderonVaillancourt}.  On a
parabolic BO--ZK block, one $\xi$ derivative gains $N^{-1}$ and one
$\eta$ derivative gains $N^{-1/2}$.  These are precisely the scales
used in the commutator expansions of \cref{sec:smoothing}.  Moreover,
a symbol compactly supported in frequency and satisfying the stated
rescaled derivative bounds has a kernel dominated by an integrable
anisotropic majorant at the reciprocal frequency scales.  The
corresponding operator is therefore bounded on every mixed norm used
in the paper.  This blockwise formulation avoids introducing global
symbol-class notation that is not otherwise needed.

\subsection{Phase geometry}

\begin{lemma}\label{lem:no-trapping}
There exists $c>0$ such that, for $|\zeta|\ge2$,
\begin{equation}\label{eq:no-trapping}
 \max\{|v_x(\zeta)|,|v_y(\zeta)|\}\ge c|\zeta|.
\end{equation}
\end{lemma}

\begin{proof}
If $|\eta|\ge|\zeta|/2$ and $|\xi|\le\eta^2/4$, then $|v_x|\ge\eta^2/2$.  If $|\eta|\ge|\zeta|/2$ and $|\xi|>\eta^2/4$, then $|v_y|=2|\xi\eta|\gtrsim|\eta|^3$.  If $|\eta|<|\zeta|/2$, then $|\xi|\simeq|\zeta|$.  When $\eta^2\le|\xi|$, one has $|v_x|\ge|\xi|$; otherwise $|v_y|\gtrsim|\xi|^{3/2}$.  Each lower bound dominates $c|\zeta|$ at high frequency.
\end{proof}

Fix $0<c_0\ll1$.  On $|\zeta|\simeq N$, choose a partition
\begin{equation}\label{eq:chart-partition}
 1=\chi_{x+}^2+\chi_{x-}^2+\chi_{y+}^2+\chi_{y-}^2
\end{equation}
with
\begin{align}
 \supp\chi_{x\pm}&\subset\{\pm v_x\ge c_0N\},\label{eq:xcharts}\\
 \supp\chi_{y\sigma}&\subset
 \{ |\eta^2-2|\xi||\le4c_0N,\ \sigma\xi\eta>0\}.
 \label{eq:ycharts}
\end{align}
On the $y$ charts,
\begin{equation}\label{eq:ygeometry}
 |\xi|\simeq N,
 \qquad |\eta|\simeq N^{1/2},
 \qquad |v_y|\simeq N^{3/2}.
\end{equation}
The transition cutoffs obey
\begin{equation}\label{eq:chart-symbol}
 |\partial_\xi^a\partial_\eta^b\chi_\nu(\xi,\eta)|
 \le C_{a,b}N^{-a-b/2}
\end{equation}
whenever their derivatives are supported in the parabolic band.  Away from the band, only the ordinary isotropic Littlewood--Paley derivatives occur.

\subsection{Energy and fractional product estimates}

\begin{lemma}\label{lem:energy}
Let $s>1$ and let $u$ be a smooth solution of \eqref{IVP} on $[0,T]$.  Then
\begin{equation}\label{eq:energy}
 \norm{u}_{L_T^\infty H^s}
 \le C\norm{u_0}_{H^s}
 \exp\left(C\int_0^T\norm{\nabla u(t)}_{L^\infty}\dd t\right).
\end{equation}
\end{lemma}

\begin{proof}
Apply $J^s$ to \eqref{IVP}, pair with $J^su$, and use the skew-adjointness of the linear operator.  The Kato--Ponce commutator estimate
\cite{KatoPonce,GrafakosOh} gives
\[
 \frac{\dd}{\dd t}\norm{J^su}_2^2
 \lesssim \norm{\nabla u}_\infty\norm{J^su}_2^2.
\]
Gronwall's inequality proves \eqref{eq:energy}.
\end{proof}

We use repeatedly the following consequence of the fractional Leibniz rule and Bony's decomposition.

\begin{lemma}\label{lem:paraproduct}
Let $0<\gamma<1$ and $s>\gamma$.  Then
\begin{equation}\label{eq:paraproduct}
 \norm{J^\gamma(uu_x)-uJ^\gamma u_x}_{L^2}
 \lesssim \norm{\nabla u}_{L^\infty}\norm{u}_{H^s}.
\end{equation}
Moreover, after extracting the low--high term $u_{<N/8}P_NJ^\gamma u_x$, all comparable-frequency terms satisfy the same bound after square summation in $N$.
\end{lemma}

\begin{proof}
The first statement follows from the Kato--Ponce commutator estimate
\cite{KatoPonce,GrafakosOh}.  For the balanced and high--high
paraproducts, Coifman--Meyer bounds and inverse Bernstein place one
full spatial derivative on the low or comparable-frequency factor;
this requires $\norm{\nabla u}_{L^\infty}$, rather than only
$\norm{u_x}_{L^\infty}$.  Littlewood--Paley square summation then
proves the second statement.
\end{proof}

\section{Oscillatory kernels and maximal functions}\label{sec:maximal}

\subsection{Product-frequency kernels}

Let $\varphi\in C_c^\infty((1/2,2))$ and let $a_{\lambda,\mu}$ be a smooth symbol supported where $|\xi|\simeq\lambda$ and $|\eta|\simeq\mu$, with the standard modifications when $\lambda=1$ or $\mu=1$.  Set
\begin{equation}\label{eq:kernel}
 K_{\lambda,\mu}(x,y,t)
 =\iint e^{i(x\xi+y\eta+t\omega(\xi,\eta))}
 a_{\lambda,\mu}(\xi,\eta)\dd\xi\dd\eta
\end{equation}
and
\begin{equation}\label{eq:L}
 L=\lambda+\mu^2.
\end{equation}

\begin{proposition}\label{prop:kernel}
Uniformly for $|t|\le1$,
\begin{align}
 \int_\R\sup_{y,|t|\le1}|K_{\lambda,\mu}(x,y,t)|\dd x
 &\lesssim
 L^{1/2}\bigl[1+\log(2+\lambda\mu L^{1/2})\bigr],
 \label{eq:kernel-x}\\
 \int_\R\sup_{x,|t|\le1}|K_{\lambda,\mu}(x,y,t)|\dd y
 &\lesssim
 \lambda\bigl[1+\log(2+\lambda\mu^2)\bigr].
 \label{eq:kernel-y}
\end{align}
\end{proposition}

\begin{proof}
The size estimate is
\begin{equation}\label{eq:kernel-size}
 |K_{\lambda,\mu}|\lesssim\lambda\mu.
\end{equation}
Throughout the proof we write
\begin{equation}\label{eq:full-phase}
 \Phi(\xi,\eta)=x\xi+y\eta+t\omega(\xi,\eta)
\end{equation}
for the full phase in \eqref{eq:kernel}, the variables $x,y,t$ being
regarded as parameters.  We next record the uniform dispersive estimate
\begin{equation}\label{eq:kernel-disp}
 |K_{\lambda,\mu}(x,y,t)|
 \lesssim \min\left\{\lambda\mu,\frac1{|t|L^{1/2}}\right\}.
\end{equation}
For $\lambda\ge2$, split into $\xi>0$ and $\xi<0$.
We spell out the standard iterated stationary-phase bookkeeping.  For
fixed $\xi$, insert a smooth partition between a neighbourhood of the
critical point of
\[
 \eta\longmapsto y\eta+t\xi\eta^2
\]
and the region on which its derivative is bounded away from zero.  On
the latter region, repeated integration by parts in $\eta$ gives an
arbitrarily decaying error.  On the critical neighbourhood, the
one-dimensional stationary-phase lemma with parameters
\cite{Hormander,Stein} produces the factor
$(|t||\xi|)^{-1/2}$ and incorporates its remainder into an amplitude
whose total variation in $\xi$ is bounded by a constant times its
supremum on $|\xi|\simeq\lambda$.  This follows after rescaling
$\eta=\mu\widetilde\eta$: derivatives of the rescaled cutoff are
uniform, while differentiation of $(|t||\xi|)^{-1/2}$ costs only
$O(\lambda^{-1})$.  Thus the subsequent van der Corput estimate
applies to the principal stationary-phase term and its remainder.  If
the relevant oscillatory parameter is smaller than one, the same
normalized-amplitude statement follows directly from the size bound.

On $\xi>0$, the resulting phase is
\[
 \Psi_+(\xi)=x\xi-t\xi^2-\frac{y^2}{4t\xi},
\]
and
\[
 \Psi_+''(\xi)=-2t-\frac{y^2}{2t\xi^3}.
\]
The two terms have the same sign.  On the stationary support, $|y|\simeq|t|\lambda\mu$, and hence
\[
 |\Psi_+''(\xi)|\gtrsim |t|\frac{L}{\lambda}.
\]
Van der Corput in $\xi$ gives the second factor
$(|t|L/\lambda)^{-1/2}$, proving \eqref{eq:kernel-disp}.  The
half-plane $\xi<0$ is identical after the substitution $\xi=-\rho$.
The nonstationary $\eta$ contribution satisfies the same estimate,
with additional decay, by the integrations by parts just described.

It remains to prove \eqref{eq:kernel-disp} for the block $\lambda=1$,
which may meet the kink $\xi=0$.  There the iterated argument above
breaks down, because $\partial_\eta^2\omega=2\xi$ degenerates as
$\xi\to0$, and one must use the full two-dimensional curvature.  When
$\mu=1$ the estimate is immediate from \eqref{eq:kernel-size}, since
$|t|\le1$ and $L\simeq1$.  Assume therefore $\mu\ge2$ and rescale
$\eta=\mu\widetilde\eta$, so that
\begin{equation}\label{eq:lowblock-rescaled}
 K_{1,\mu}(x,y,t)
 =\mu\iint
 e^{i(x\xi+y\mu\widetilde\eta+S\xi\widetilde\eta^2)}
 b(\xi,\widetilde\eta)\dd\xi\dd\widetilde\eta,
 \qquad S=t\mu^2,
\end{equation}
where
\[
 b(\xi,\widetilde\eta)
 =a_{1,\mu}(\xi,\mu\widetilde\eta)\,e^{-it\xi|\xi|}.
\]
Thus the kink has been absorbed into the amplitude.  This is legitimate
because $|\xi|\lesssim1$ on the block and $|t|\le1$, so that
\[
 \partial_\xi\bigl(t\xi|\xi|\bigr)=2t|\xi|=O(1),
 \qquad
 \partial_\xi^2\bigl(t\xi|\xi|\bigr)=2t\sgn\xi=O(1)
\]
almost everywhere.  Hence $b$ is supported in a set of unit size and
has derivatives of order at most two bounded uniformly in $\mu$ and
$t$, together with bounded derivatives of every order in
$\widetilde\eta$; that is, $b\in C^{1,1}$ with uniform bounds.  The
remaining phase
\[
 \widetilde\Phi(\xi,\widetilde\eta)
 =x\xi+y\mu\widetilde\eta+S\xi\widetilde\eta^2
\]
is a polynomial, and on the support $|\widetilde\eta|\simeq1$ its
Hessian satisfies
\[
 \det D^2\widetilde\Phi=-4S^2\widetilde\eta^2,
 \qquad
 \left|\det D^2\widetilde\Phi\right|\simeq S^2 .
\]
Moreover $\widetilde\Phi$ has at most two critical points there: the
equation $\partial_\xi\widetilde\Phi=0$ determines
$\widetilde\eta^{\,2}$, and $\partial_{\widetilde\eta}\widetilde\Phi=0$
then determines $\xi$.  Two-dimensional
stationary phase with a $C^{1,1}$ amplitude on a support of unit size
therefore gives
\[
 |K_{1,\mu}(x,y,t)|
 \lesssim\mu\,|S|^{-1}
 =\frac1{|t|\mu}
 \simeq\frac1{|t|L^{1/2}},
\]
which is the second alternative in \eqref{eq:kernel-disp}; when
$|t|\mu^2\lesssim1$ the size estimate \eqref{eq:kernel-size} is the
stronger of the two.  In invariant terms this exponent records
\[
 |\det D^2\omega(\xi,\eta)|=4(|\xi|+\eta^2)\simeq\mu^2
 \qquad\text{on the block }\lambda=1,\ \mu\ge2 .
\]

We turn to the nonstationary estimates used for the two majorants.
When $\lambda\ge2$ the $\xi$ support is separated from the origin and
all integrations by parts below are classical.  For the low
$x$-frequency block $\lambda=1$, split the $\xi$ integral at zero.
The phase and its first $\xi$ derivative are continuous at zero,
\[
 \Phi(0^+,\eta)=\Phi(0^-,\eta),\qquad
 \partial_\xi\Phi(0^+,\eta)=\partial_\xi\Phi(0^-,\eta)
 =x+t\eta^2,
\]
so the first boundary terms cancel.  The second derivative has the
jump $\partial_\xi^2\Phi(0^\pm,\eta)=\mp2t$; after two integrations by
parts its boundary contribution is bounded by
$C|t|\mu |x|^{-3}$.  Hence
\[
 \int_{|x|>CL}|t|\mu |x|^{-3}\,\dd x
 \lesssim \frac{|t|\mu}{L^2}\lesssim1\lesssim L^{1/2}.
\]
Thus two integrations suffice for the $\lambda=1$ tail, while the
higher blocks admit arbitrary repetitions.  This justifies the
nonstationary estimates across the kink of $\xi|\xi|$.

For the $x$ majorant define
\[
 x_0=\frac{L^{1/2}}{\lambda\mu}.
\]
The size estimate applies for $|x|\le x_0$.  If $x_0<|x|\le CL$ and $|t|L\le c|x|$, then $|\partial_\xi\Phi|\ge|x|/2$, and one integration by parts in $\xi$ gives
\[
 |K_{\lambda,\mu}|
 \lesssim \frac{\mu}{|x|}
 +\frac{|t|\lambda\mu}{|x|^2}
 \lesssim \frac{L^{1/2}}{|x|}.
\]
If $|t|L>c|x|$, use \eqref{eq:kernel-disp} to obtain the same bound.
Repeated integration by parts gives rapid decay for $|x|>CL$; each
step gains a factor $(\lambda|x|)^{-1}$.  Indeed, symbol derivatives
are of size $\lambda^{-1}$ and
$|\partial_\xi\Phi|\ge|x|/2$, while the term in which the derivative
falls on the reciprocal phase derivative is bounded by
$|t|/|x|^2$.  Since $|x|>CL\ge C\lambda$ and $|t|\le1$, this term is
also bounded by $C/(\lambda|x|)$.  Therefore
\[
 \sup_{y,t}|K_{\lambda,\mu}(x,y,t)|
 \lesssim_M
 \begin{cases}
 \lambda\mu,& |x|\le x_0,\\
 L^{1/2}|x|^{-1},&x_0<|x|\le CL,\\
\lambda\mu\,(1+\lambda|x|)^{-M},&|x|>CL.
 \end{cases}
\]
Here $M$ may be taken arbitrarily large when $\lambda\ge2$, whereas for
$\lambda=1$ the discussion above supplies $M=2$; since $L=1+\mu^2$ on
that block, the corresponding tail integral is
$\displaystyle\int_{|x|>CL}\mu|x|^{-2}\dd x\lesssim\mu/L\lesssim1\lesssim L^{1/2}$,
so $M=2$ is all that is needed.
Integration proves \eqref{eq:kernel-x}.

For the $y$ majorant let $y_0=\mu^{-1}$.  If $|t|\lambda\mu\le c|y|$, integration by parts in $\eta$ gives $|K|\lesssim\lambda/|y|$.  In the complementary region, \eqref{eq:kernel-disp} gives
\[
 |K|\lesssim\frac{\lambda\mu}{|y|L^{1/2}}
 \le\frac{\lambda}{|y|}.
\]
Repeated integration by parts applies for $|y|>C\lambda\mu$.  Thus
\[
 \sup_{x,t}|K_{\lambda,\mu}(x,y,t)|
 \lesssim_M
 \begin{cases}
 \lambda\mu,& |y|\le y_0,\\
 \lambda|y|^{-1},&y_0<|y|\le C\lambda\mu,\\
\lambda\mu\,(1+\mu|y|)^{-M},&|y|>C\lambda\mu.
 \end{cases}
\]
This proves \eqref{eq:kernel-y}.
\end{proof}

\subsection{Strichartz estimates}

A pair $(p,q)$ will be called BO--ZK admissible if
\begin{equation}\label{eq:admissible}
 \frac83<p<\infty,
 \qquad
 \frac1q+\frac{4}{3p}=\frac12.
\end{equation}
The next estimate is the BO--ZK analogue of the linear estimate used
in the refined Strichartz argument of Linares, Pilod, and Saut.  We
include the proof because the product-frequency geometry is essential:
an isotropic $TT^*$ argument would leave a loss $N^{1/(6p)}$ after
Bernstein and would not lead to the exponent $19/16$.

\begin{proposition}\label{prop:strichartz}
For every admissible pair $(p,q)$ and every interval $I$ with
$|I|\le1$,
\begin{equation}\label{eq:strichartz}
 \norm{U(t)f}_{L_I^pL_{x,y}^q}\lesssim_{p,q}\norm f_2.
\end{equation}
The constant is independent of the position and length of $I$.
\end{proposition}

\begin{proof}
Let $U_{\lambda,\mu}(t)=U(t)P^x_\lambda P^y_\mu$ and set
$L=\lambda+\mu^2$.  By \eqref{eq:kernel-disp},
\begin{equation}\label{eq:block-dispersive-strichartz}
 \norm{U_{\lambda,\mu}(t)U_{\lambda,\mu}(s)^*g}_{L^\infty}
 \lesssim |t-s|^{-1}L^{-1/2}\norm g_{L^1}.
\end{equation}
The same operator is bounded on $L^2$ by unitarity.  Let $r$ be
defined by
\begin{equation}\label{eq:r-strichartz}
 \frac1r+\frac1p=\frac12.
\end{equation}
Interpolation between \eqref{eq:block-dispersive-strichartz} and the
$L^2$ bound gives
\begin{equation}\label{eq:block-interpolated}
 \norm{U_{\lambda,\mu}(t)U_{\lambda,\mu}(s)^*g}_{L^r}
 \lesssim
 |t-s|^{-2/p}L^{-1/p}\norm g_{L^{r'}}.
\end{equation}
Since $p>2$, the one-dimensional Hardy--Littlewood--Sobolev inequality
applied to the time convolution in \eqref{eq:block-interpolated}
yields
\begin{equation}\label{eq:block-TTstar}
 \norm{TT^*G}_{L_I^pL^r}
 \lesssim L^{-1/p}\norm G_{L_I^{p'}L^{r'}},
 \qquad T f=U_{\lambda,\mu}(t)f\big|_I.
\end{equation}
Taking the square root of the $TT^*$ norm, we obtain
\begin{equation}\label{eq:block-strichartz-r}
 \norm{U_{\lambda,\mu}(t)f}_{L_I^pL^r}
 \lesssim L^{-1/(2p)}\norm f_2.
\end{equation}

The admissibility relation and \eqref{eq:r-strichartz} give
\begin{equation}\label{eq:bernstein-gap}
 \frac1r-\frac1q=\frac1{3p}.
\end{equation}
Bernstein's inequality on the product block therefore implies
\begin{align}
 \norm{U_{\lambda,\mu}(t)f}_{L_I^pL^q}
 &\lesssim
 (\lambda\mu)^{1/(3p)}L^{-1/(2p)}\norm f_2.\label{eq:block-strichartz-q}
\end{align}
Since $\lambda\le L$ and $\mu\le L^{1/2}$,
\begin{equation}\label{eq:block-cancellation}
 \lambda\mu\le L^{3/2},
 \qquad
 (\lambda\mu)^{1/(3p)}L^{-1/(2p)}\le1.
\end{equation}
Thus every product block satisfies \eqref{eq:strichartz} with a
uniform constant.

It remains to recombine the blocks.  The product Littlewood--Paley
square function is bounded on $L^q(\R^2)$, and $p,q\ge2$.  Hence
Minkowski's inequality for the $\ell^2$ sum gives
\begin{align*}
 \norm{U(t)f}_{L_I^pL^q}
 &\lesssim
 \norm{\left(\sum_{\lambda,\mu}
 |U(t)P^x_\lambda P^y_\mu f|^2\right)^{1/2}}
 _{L_I^pL^q}\\
 &\le
 \left(\sum_{\lambda,\mu}
 \norm{U(t)P^x_\lambda P^y_\mu f}_{L_I^pL^q}^2
 \right)^{1/2}\\
 &\lesssim
 \left(\sum_{\lambda,\mu}
 \norm{P^x_\lambda P^y_\mu f}_2^2\right)^{1/2}
 \lesssim \norm f_2.
\end{align*}
The Hardy--Littlewood--Sobolev estimate is applied after extending functions on $I$ by zero, so its constant is independent of $I$.
\end{proof}

\subsection{\texorpdfstring{$TT^*$}{TT-star} and the dyadic maximal estimate}

\begin{lemma}\label{lem:TTstar}
Let $U_{\lambda,\mu}(t)$ be the block propagator with kernel \eqref{eq:kernel}.  If $A_x$ and $A_y$ denote the left-hand sides of \eqref{eq:kernel-x} and \eqref{eq:kernel-y}, then
\begin{align}
 \norm{U_{\lambda,\mu}(t)f}_{L_x^2L_{y,T}^\infty}
 &\lesssim A_x^{1/2}\norm f_2,\label{eq:TTx}\\
 \norm{U_{\lambda,\mu}(t)f}_{L_y^2L_{x,T}^\infty}
 &\lesssim A_y^{1/2}\norm f_2.\label{eq:TTy}
\end{align}
\end{lemma}

\begin{proof}
We prove \eqref{eq:TTx}; the proof of \eqref{eq:TTy} is obtained by
interchanging the roles of $x$ and $y$ and using \eqref{eq:kernel-y}
in place of \eqref{eq:kernel-x}.

Write $Tf=U_{\lambda,\mu}(t)f$, regarded as a map
$L^2(\R^2)\to L_x^2L_{y,T}^\infty$.  Frequency localization makes
$Tf$ continuous in $(y,t)$, so the pointwise supremum in the mixed
norm is measurable.  The K\"othe-dual norm identity for mixed
Lebesgue spaces gives
\[
 \norm{Tf}_{L_x^2L_{y,T}^\infty}
 =\sup_{\norm g_{L_x^2L_{y,T}^1}\le1}|\ip{Tf}{g}|.
\]
This identity can also be obtained first for bounded,
compactly supported simple functions and then by monotone
approximation.  Therefore \eqref{eq:TTx} follows from
\begin{equation}\label{eq:Tstar}
 \norm{T^*g}_{L^2(\R^2)}
 \lesssim A_x^{1/2}\norm g_{L_x^2L_{y,T}^1},
 \qquad
 T^*g=\int_0^TU_{\lambda,\mu}(t)^*g(\cdot,t)\dd t .
\end{equation}
Indeed, $\ip{Tf}{g}=\ip f{T^*g}$ for such simple functions, and this
identity extends by density in the norming space
$L_x^2L_{y,T}^1$.

To prove \eqref{eq:Tstar}, expand the square:
\begin{equation}\label{eq:TTstar-expand}
 \norm{T^*g}_{L^2}^2
 =\int_0^T\!\!\int_0^T
 \ip{U_{\lambda,\mu}(t')U_{\lambda,\mu}(t)^*g(t)}{g(t')}
 \dd t\dd t' .
\end{equation}
The operator $U_{\lambda,\mu}(t')U_{\lambda,\mu}(t)^*$ is convolution
in $(x,y)$ with $\widetilde K_{\lambda,\mu}(\cdot,\cdot,t'-t)$, where
$\widetilde K_{\lambda,\mu}$ is the kernel \eqref{eq:kernel} with
$a_{\lambda,\mu}$ replaced by $|a_{\lambda,\mu}|^2$.  The latter symbol
satisfies the same support and derivative hypotheses as
$a_{\lambda,\mu}$, so \cref{prop:kernel} applies to
$\widetilde K_{\lambda,\mu}$ with the same majorants.  Moreover
$|t'-t|\le T\le1$, so only the range $|t|\le1$ of \eqref{eq:kernel-x}
is used.  Set
\[
 k_{\lambda,\mu}^{x}(x)=\sup_{y\in\R,\ |t|\le1}
 |\widetilde K_{\lambda,\mu}(x,y,t)|,
 \qquad
 G(x)=\int_0^T\!\!\int_\R|g(x,y,t)|\dd y\dd t .
\]
Estimating the integrand of \eqref{eq:TTstar-expand} pointwise and
taking the supremum in the $y$ and $t$ differences gives
\[
 \norm{T^*g}_{L^2}^2
 \le\int_\R\int_\R k_{\lambda,\mu}^{x}(x-x')G(x')G(x)\dd x'\dd x .
\]
By Young's inequality in $x$ followed by Cauchy--Schwarz,
\[
 \norm{T^*g}_{L^2}^2
 \le\norm{k_{\lambda,\mu}^{x}}_{L^1(\R)}\norm G_{L^2(\R)}^2
 \lesssim A_x\norm g_{L_x^2L_{y,T}^1}^2,
\]
since $\norm{k_{\lambda,\mu}^{x}}_{L^1}\lesssim A_x$ by
\eqref{eq:kernel-x} and
$\norm G_{L^2(\R)}=\norm g_{L_x^2L_{y,T}^1}$.  This is
\eqref{eq:Tstar}.
\end{proof}
\subsection{Proof of the maximal-function estimate}

We now prove \cref{thm:maximal}, the principal new linear ingredient in the proof of the main local well-posedness result.

\begin{proof}[Proof of \cref{thm:maximal}]
Combining \cref{prop:kernel} with \cref{lem:TTstar},
\begin{align}
 \norm{U_{\lambda,\mu}(t)f}_{L_x^2L_{y,T}^\infty}
 &\lesssim L^{1/4}\bigl(1+\log(2+\lambda\mu L^{1/2})\bigr)^{1/2}
 \norm f_2,\label{eq:block-max-x}\\
 \norm{U_{\lambda,\mu}(t)f}_{L_y^2L_{x,T}^\infty}
 &\lesssim \lambda^{1/2}\bigl(1+\log(2+\lambda\mu^2)\bigr)^{1/2}
 \norm f_2.\label{eq:block-max-y}
\end{align}

Let $\mathcal B_N$ be the set of dyadic pairs $(\lambda,\mu)$ for
which $P^x_\lambda P^y_\mu P_N\ne0$.  If $(\lambda,\mu)\in\mathcal B_N$
then $\lambda\le2N$ and $\mu\le2N$; moreover, since
$\max\{|\xi|,|\eta|\}\ge|\zeta|/\sqrt2\ge N/(2\sqrt2)$ on the support
of $P_N$, at least one of $\lambda,\mu$ is comparable to $N$.  Hence
\begin{equation}\label{eq:block-count}
 \#\mathcal B_N\lesssim\log N .
\end{equation}
For $(\lambda,\mu)\in\mathcal B_N$ one has
$L=\lambda+\mu^2\lesssim N^2$, so $L^{1/4}\lesssim N^{1/2}$ and
$\lambda^{1/2}\lesssim N^{1/2}$; and $\lambda\mu L^{1/2}\lesssim N^3$,
$\lambda\mu^2\lesssim N^3$, so both logarithms in
\eqref{eq:block-max-x}-\eqref{eq:block-max-y} are $O(1+\log N)$.
Consequently, uniformly over $\mathcal B_N$,
\begin{equation}\label{eq:block-max-uniform}
 \norm{U_{\lambda,\mu}(t)P_Nf}_{L_x^2L_{y,T}^\infty}
 +\norm{U_{\lambda,\mu}(t)P_Nf}_{L_y^2L_{x,T}^\infty}
 \lesssim N^{1/2}(1+\log N)^{1/2}
 \norm{P^x_\lambda P^y_\mu P_Nf}_2 .
\end{equation}

The two target norms contain an $L^\infty$ component, so no
Banach-valued Littlewood--Paley orthogonality is available and we sum
by hand.  Writing
$P_Nf=\sum_{(\lambda,\mu)\in\mathcal B_N}P^x_\lambda P^y_\mu P_Nf$,
the triangle inequality, \eqref{eq:block-max-uniform}, Cauchy--Schwarz
over $\mathcal B_N$, \eqref{eq:block-count}, and the almost
orthogonality of the product projections give
\begin{align*}
 &\norm{U(t)P_Nf}_{L_x^2L_{y,T}^\infty}
 +\norm{U(t)P_Nf}_{L_y^2L_{x,T}^\infty}\\
 &\qquad\le
 \sum_{(\lambda,\mu)\in\mathcal B_N}
 \Bigl(
 \norm{U_{\lambda,\mu}(t)P_Nf}_{L_x^2L_{y,T}^\infty}
 +\norm{U_{\lambda,\mu}(t)P_Nf}_{L_y^2L_{x,T}^\infty}
 \Bigr)\\
 &\qquad\lesssim
 N^{1/2}(1+\log N)^{1/2}
 \sum_{(\lambda,\mu)\in\mathcal B_N}
 \norm{P^x_\lambda P^y_\mu P_Nf}_2\\
 &\qquad\le
 N^{1/2}(1+\log N)^{1/2}
 (\#\mathcal B_N)^{1/2}
 \Bigl(
 \sum_{(\lambda,\mu)\in\mathcal B_N}
 \norm{P^x_\lambda P^y_\mu P_Nf}_2^2
 \Bigr)^{1/2}\\
 &\qquad\lesssim
 N^{1/2}(1+\log N)\norm{P_Nf}_2 .
\end{align*}
This proves \eqref{eq:dyadic-maximal} with $C_0=1$.  The
$N^{1/2+\eps}$ consequence follows by absorbing the logarithm into
$N^\eps$.
\end{proof}

\begin{remark}\label{rem:sharpness}
Fix $\psi\in\Sscr(\R^2)$, $\psi\not\equiv0$, with $\widehat\psi$
supported in the annulus $\{3/4\le|\zeta|\le3/2\}$, and set
$f_N(x,y)=N\psi(Nx,Ny)$.  Then
$\widehat{f_N}(\zeta)=N^{-1}\widehat\psi(\zeta/N)$ is supported where
the symbol $\varphi_1(\cdot/N)$ of $P_N$ equals one, so
$P_Nf_N=f_N$, and $\norm{f_N}_2=\norm\psi_2\simeq1$.  Putting
$\Psi(a)=\sup_{b\in\R}|\psi(a,b)|$, which lies in $L^2(\R)$, one has
$\sup_y|f_N(x,y)|=N\Psi(Nx)$ and therefore, already at $t=0$,
\[
 \norm{f_N}_{L_x^2L_y^\infty}
 =N^{1/2}\norm\Psi_{L^2(\R)}\simeq N^{1/2},
\]
and symmetrically for the other mixed norm.  Since the left-hand side
of \eqref{eq:dyadic-maximal} dominates its value at $t=0$, the power
$N^{1/2}$ cannot be lowered.  Note that this example uses no
dispersion whatsoever: the power $N^{1/2}$ is precisely the cost of
the embedding $L^2\to L_x^2L_y^\infty$ at frequency $N$.  The content
of \cref{thm:maximal} is that the additional supremum over
$t\in[0,T]$ is free up to logarithms.
\end{remark}

Define the continuous mixed maximal norm
\begin{equation}\label{eq:M-def}
 M_T(w)=
 \norm w_{L_x^2L_{y,T}^\infty}
 +\norm w_{L_y^2L_{x,T}^\infty}.
\end{equation}

\begin{corollary}\label{cor:retarded-max}
For every $\eps>0$ and every solution of \eqref{eq:linear},
\begin{equation}\label{eq:retarded-max}
 M_T(w)
 \lesssim_\eps
 \norm{J^{1/2+\eps}w(0)}_2
 +\norm{J^{1/2+\eps}F}_{L_T^1L^2}.
\end{equation}
\end{corollary}

\begin{proof}
Apply \cref{thm:maximal} to every dyadic component $P_N$, $N\ge2$,
of the homogeneous term.  For $P_1$, the same $TT^*$ argument with
$\lambda=\mu=1$ gives $A_x+A_y\lesssim1$, and hence the corresponding
estimate without a frequency loss.  For the Duhamel term, Minkowski's
inequality and time translation give
\[
 M_T\left(\int_0^tU(t-t')P_NF(t')\dd t'\right)
 \lesssim_\eps N^{1/2+\eps}\norm{P_NF}_{L_T^1L^2}.
\]
No Christ--Kiselev argument is needed: the forcing is measured in
$L_T^1L^2$, and the time-truncated supremum is dominated pointwise
before Minkowski is applied.
Summing in $N$ with an arbitrarily small additional Sobolev loss proves \eqref{eq:retarded-max}.
\end{proof}

\section{Refined Strichartz estimates}\label{sec:refined}

The estimate in this section is obtained by a frequency-dependent
partition of the time interval.  This device goes back to the refined
Strichartz argument of Kenig \cite{KenigKPI} and was used systematically for
fractional KP equations by Linares, Pilod, and Saut
\cite[Lemma 4.11]{LinaresPilodSaut}.  Their proof is based on a
Littlewood--Paley decomposition in the dispersive variable and an
$L_t^1L^\infty$ estimate.  Here we adapt the same organization to
isotropic BO--ZK frequency blocks and to the norm
$L_T^2L^\infty_{x,y}$.  We keep the short-time scale free in the next
lemma and optimize it afterward.

A minor point in the summation deserves attention.  A bound involving
$\norm{P_Nw}_{L_T^\infty L^2}$ cannot simply be square-summed in $N$,
because the maximizing time may depend on the frequency.  We avoid
this issue as follows.  On intervals whose length is the prescribed
short-time scale, the base time in Duhamel's formula is selected by
averaging the $L^2$ norm.  Frequencies for which the whole interval
$[0,T]$ is shorter than that scale are treated directly from the
initial time.  This gives the two-case estimate below and permits a
standard Sobolev square summation.

\begin{lemma}
\label{lem:refined-dyadic}
Let $N\ge2$ be dyadic, let $(p,q)$ be BO--ZK admissible in the sense
of \eqref{eq:admissible}, and let $0\le\vartheta\le1$.  Set
\begin{equation}\label{eq:ab-vartheta}
 a_\vartheta(p)=1+\frac2q+\frac\vartheta p,
 \qquad
 b_\vartheta(p)=1+\frac2q-\vartheta+\frac\vartheta p.
\end{equation}
If $w_N=P_Nw$ and $F_N=P_NF$, where $w$ solves \eqref{eq:linear} on
$[0,T]$, $0<T\le1$, then the following estimates hold.

If $T\ge N^{-\vartheta}$, then
\begin{equation}\label{eq:refined-dyadic-long}
 \norm{\nabla w_N}_{L_T^2L^\infty}
 \lesssim
 N^{a_\vartheta(p)}\norm{w_N}_{L_T^2L^2}
 +N^{b_\vartheta(p)}\norm{F_N}_{L_T^2L^2}.
\end{equation}
If $T<N^{-\vartheta}$, then
\begin{equation}\label{eq:refined-dyadic-short}
 \norm{\nabla w_N}_{L_T^2L^\infty}
 \lesssim
 N^{a_\vartheta(p)-\vartheta/2}\norm{w_N(0)}_2
 +N^{b_\vartheta(p)}\norm{F_N}_{L_T^2L^2}.
\end{equation}
In particular, in both cases,
\begin{equation}\label{eq:refined-dyadic-simple}
 \norm{\nabla w_N}_{L_T^2L^\infty}
 \lesssim
 N^{a_\vartheta(p)}\norm{w_N}_{L_T^\infty L^2}
 +N^{b_\vartheta(p)}\norm{F_N}_{L_T^2L^2}.
\end{equation}
The implicit constants are independent of $N$ and $T$.
\end{lemma}

\begin{proof}
Write $\delta_N=N^{-\vartheta}$.  Since $w_N$ is supported where
$|\zeta|\simeq N$, Bernstein's inequality in two space dimensions
gives
\begin{equation}\label{eq:bernstein-short-interval}
 \norm{\nabla w_N(t)}_{L^\infty}
 \lesssim N^{1+2/q}\norm{w_N(t)}_{L^q}.
\end{equation}
Also, $p>8/3>2$, and therefore, for any time interval $I$,
\begin{equation}\label{eq:holder-short-interval}
 \norm{\nabla w_N}_{L_I^2L^\infty}
 \lesssim
 N^{1+2/q}|I|^{\frac12-\frac1p}
 \norm{w_N}_{L_I^pL^q}.
\end{equation}

We first suppose that $T\ge\delta_N$.  Partition $[0,T]$ into
consecutive intervals $I_j$ whose lengths satisfy
\begin{equation}\label{eq:comparable-short-intervals}
 \frac12\delta_N\le |I_j|\le\delta_N.
\end{equation}
By averaging, one may choose $t_j\in I_j$ such that
\begin{equation}\label{eq:average-base-time}
 \norm{w_N(t_j)}_2^2
 \le \frac{2}{\delta_N}
 \int_{I_j}\norm{w_N(t)}_2^2\dd t.
\end{equation}
For $t\in I_j$, the group property gives the two-sided Duhamel formula
\begin{equation}\label{eq:duhamel-average-base}
 w_N(t)=U(t-t_j)w_N(t_j)
 +\int_{t_j}^{t}U(t-t')F_N(t')\dd t'.
\end{equation}
The orientation of the last integral is immaterial in the estimates.
By time translation, restriction to $I_j$, and
\cref{prop:strichartz},
\begin{equation}\label{eq:homogeneous-short-interval}
 \norm{U(t-t_j)w_N(t_j)}_{L_{I_j}^pL^q}
 \lesssim \norm{w_N(t_j)}_2.
\end{equation}
Minkowski's inequality and the same homogeneous estimate imply
\begin{align}
 &\norm{\int_{t_j}^{t}U(t-t')F_N(t')\dd t'}_{L_{I_j}^pL^q}
 \notag\\
 &\qquad\lesssim \int_{I_j}\norm{F_N(t')}_2\dd t'.
 \label{eq:inhomogeneous-short-interval}
\end{align}
No Christ--Kiselev argument is required: for each fixed $t'$, the
relevant forward or backward time interval is a restriction of a
time translate of the homogeneous Strichartz estimate.

Combining
\eqref{eq:holder-short-interval}-\eqref{eq:inhomogeneous-short-interval}
and using \eqref{eq:comparable-short-intervals}, we obtain
\begin{equation}\label{eq:one-short-interval}
 \norm{\nabla w_N}_{L_{I_j}^2L^\infty}
 \lesssim
 N^{1+\frac2q}\delta_N^{\frac12-\frac1p}
 \left(
   \norm{w_N(t_j)}_2
   +\int_{I_j}\norm{F_N(t')}_2\dd t'
 \right).
\end{equation}
Square and sum over the disjoint intervals.  For the homogeneous
part, \eqref{eq:average-base-time} gives
\begin{align}
 &N^{1+\frac2q}\delta_N^{\frac12-\frac1p}
 \left(\sum_j\norm{w_N(t_j)}_2^2\right)^{1/2}
 \notag\\
 &\qquad\lesssim
 N^{1+\frac2q}\delta_N^{-1/p}
 \norm{w_N}_{L_T^2L^2}
 =N^{a_\vartheta(p)}\norm{w_N}_{L_T^2L^2}.
 \label{eq:homogeneous-summed}
\end{align}
For the forcing part, Cauchy--Schwarz on each $I_j$ yields
\[
 \left(\int_{I_j}\norm{F_N(t')}_2\dd t'\right)^2
 \le |I_j|\int_{I_j}\norm{F_N(t')}_2^2\dd t'.
\]
Consequently,
\begin{align}
 &N^{1+\frac2q}\delta_N^{\frac12-\frac1p}
 \left[
   \sum_j
   \left(\int_{I_j}\norm{F_N(t')}_2\dd t'\right)^2
 \right]^{1/2}
 \notag\\
 &\qquad\lesssim
 N^{1+\frac2q}\delta_N^{1-\frac1p}
 \norm{F_N}_{L_T^2L^2}
 =N^{b_\vartheta(p)}\norm{F_N}_{L_T^2L^2}.
 \label{eq:forcing-summed}
\end{align}
This proves \eqref{eq:refined-dyadic-long}.

Suppose now that $T<\delta_N$.  We use the entire interval as one
piece and base Duhamel's formula at $0$.  Equations
\eqref{eq:holder-short-interval},
\eqref{eq:homogeneous-short-interval}, and
\eqref{eq:inhomogeneous-short-interval} give
\begin{align}
 \norm{\nabla w_N}_{L_T^2L^\infty}
 &\lesssim
 N^{1+\frac2q}T^{\frac12-\frac1p}\norm{w_N(0)}_2
 \notag\\
 &\quad+
 N^{1+\frac2q}T^{\frac12-\frac1p}
 \int_0^T\norm{F_N(t')}_2\dd t'.
 \label{eq:short-total-time}
\end{align}
Since $T<N^{-\vartheta}$ and $1/2-1/p>0$,
\[
 T^{\frac12-\frac1p}
 \le N^{-\vartheta(\frac12-\frac1p)}.
\]
After Cauchy--Schwarz in the forcing integral, we also have
\[
 T^{1-\frac1p}
 \le N^{-\vartheta(1-\frac1p)}.
\]
Substitution in \eqref{eq:short-total-time} proves
\eqref{eq:refined-dyadic-short}.  Finally,
\eqref{eq:refined-dyadic-simple} follows from
$\norm{w_N}_{L_T^2L^2}\le T^{1/2}\norm{w_N}_{L_T^\infty L^2}$ in
the first case and from
$\norm{w_N(0)}_2\le\norm{w_N}_{L_T^\infty L^2}$ in the second.
\end{proof}

\begin{proposition}\label{prop:refined}
Let $w$ solve \eqref{eq:linear} on $[0,T]$, $0<T\le1$.  For every
$\eps>0$,
\begin{equation}\label{eq:refined}
 \norm{\nabla w}_{L_T^2L^\infty}
 \lesssim_\eps
 \norm{J^{19/16+\eps}w}_{L_T^\infty L^2}
 +\norm{J^{11/16+\eps}F}_{L_T^2L^2}
 +\norm w_{L_T^\infty L^2}
 +\norm F_{L_T^2L^2}.
\end{equation}
\end{proposition}

\begin{proof}
Apply \cref{lem:refined-dyadic} with $\vartheta=1/2$.  The two
frequency exponents in \eqref{eq:ab-vartheta} become
\begin{equation}\label{eq:refined-exponents-general}
 a(p)=1+\frac2q+\frac1{2p},
 \qquad
 b(p)=1+\frac2q-\frac12+\frac1{2p},
 \qquad a(p)-b(p)=\frac12.
\end{equation}
This is the choice compatible with the half-derivative gain required
by the subsequent microlocal smoothing argument.
By the admissibility relation \eqref{eq:admissible},
\[
 \frac2q=1-\frac8{3p},
\]
and hence
\begin{equation}\label{eq:refined-exponents-optimized}
 a(p)=2-\frac{13}{6p},
 \qquad
 b(p)=\frac32-\frac{13}{6p}.
\end{equation}
As $p\downarrow8/3$,
\begin{equation}\label{eq:refined-endpoint-limits}
 a(p)\downarrow\frac{19}{16},
 \qquad
 b(p)\downarrow\frac{11}{16}.
\end{equation}

Fix $\eps>0$ and choose $p>8/3$ sufficiently close to $8/3$ that
\begin{equation}\label{eq:p-close-endpoint}
 a(p)<\frac{19}{16}+\frac\eps4,
 \qquad
 b(p)<\frac{11}{16}+\frac\eps4.
\end{equation}
Split the high frequencies into
\[
 \mathcal N_{\mathrm{long}}(T)
 =\{N\ge2:T\ge N^{-1/2}\},
 \qquad
 \mathcal N_{\mathrm{short}}(T)
 =\{N\ge2:T<N^{-1/2}\}.
\]
For $N\in\mathcal N_{\mathrm{long}}(T)$, use
\eqref{eq:refined-dyadic-long}; for
$N\in\mathcal N_{\mathrm{short}}(T)$, use
\eqref{eq:refined-dyadic-short}.  The triangle inequality in
$L_T^2L^\infty$ yields
\begin{align}
 \norm{\nabla P_{\ge2}w}_{L_T^2L^\infty}
 \lesssim{}&
 \sum_{N\in\mathcal N_{\mathrm{long}}(T)}
 N^{a(p)}\norm{w_N}_{L_T^2L^2}
 \notag\\
 &+
 \sum_{N\in\mathcal N_{\mathrm{short}}(T)}
 N^{a(p)-1/4}\norm{w_N(0)}_2
 \notag\\
 &+
 \sum_{N\ge2}N^{b(p)}\norm{F_N}_{L_T^2L^2}.
 \label{eq:refined-three-sums}
\end{align}

Insert the summable dyadic weight $N^{-\eps/4}$ in each sum and use
Cauchy--Schwarz.  The first sum is bounded by
\begin{align*}
 \left(\sum_{N\ge2}N^{-\eps/2}\right)^{1/2}
 \left(\sum_{N\ge2}
 N^{2a(p)+\eps/2}\norm{w_N}_{L_T^2L^2}^2\right)^{1/2}
 \lesssim_\eps
 \norm{J^{19/16+\eps}w}_{L_T^2L^2}.
\end{align*}
Since $T\le1$, this is at most
$\norm{J^{19/16+\eps}w}_{L_T^\infty L^2}$.  The second sum has an
additional gain $N^{-1/4}$ and therefore satisfies
\[
 \sum_{N\in\mathcal N_{\mathrm{short}}(T)}
 N^{a(p)-1/4}\norm{w_N(0)}_2
 \lesssim_\eps \norm{J^{19/16+\eps}w(0)}_2
 \le \norm{J^{19/16+\eps}w}_{L_T^\infty L^2}.
\]
Similarly,
\[
 \sum_{N\ge2}N^{b(p)}\norm{F_N}_{L_T^2L^2}
 \lesssim_\eps
 \norm{J^{11/16+\eps}F}_{L_T^2L^2}.
\]
Thus the small Sobolev loss in \eqref{eq:refined} is used only for
the dyadic summation of the $L^\infty$ bounds; no vector-valued
Littlewood--Paley estimate with an $L^\infty$ target is invoked.

Finally, Bernstein and $T\le1$ give
\[
 \norm{\nabla P_1w}_{L_T^2L^\infty}
 \lesssim \norm{P_1w}_{L_T^\infty L^2}.
\]
This is controlled by the third term on the right of
\eqref{eq:refined}; the final low-frequency forcing term is retained
for the inhomogeneous formulation.  Combining the low- and
high-frequency estimates proves \eqref{eq:refined}.
\end{proof}

\section{Frequency-adapted microlocal smoothing}\label{sec:smoothing}

This section proves the nonlinear smoothing estimate in the continuous norms that pair with \eqref{eq:M-def}.  The use of frequency-adapted weights is important: unit-width strip estimates would require a discrete spatial maximal function, whereas the weights below directly recover $L_x^\infty$ or $L_y^\infty$ after a vector-valued Bernstein argument.

\subsection{Frequency-adapted weights}

Choose an even function $\widehat\psi_0\in C_c^\infty(\R)$, supported in
$[-\kappa_0,\kappa_0]$, such that its inverse Fourier transform
$\psi_0$ is real, $\psi_0(0)=1$, and
\begin{equation}\label{eq:w-lower}
 \psi_0(y)\ge c_\psi>0\qquad (|y|\le2).
\end{equation}
This is achieved by taking $\kappa_0>0$ sufficiently small.  Then
$\psi_0$ is
Schwartz and band limited.  Set
\begin{equation}\label{eq:aprime-square}
 a(y)=\int_{-\infty}^{y}\psi_0(r)^2\dd r.
\end{equation}
Thus $a$ is bounded and increasing, $a'=\psi_0^2\ge0$, $a'$ is Schwartz,
and $\widehat{a'}$ is compactly supported.  The factorization
$a'=\psi_0^2$ is used in \cref{prop:positive,prop:crossterm}; compact
Fourier support is used in \cref{lem:hilbert-weight}.  Fix a large
constant $R\gg1$.  For an $x$ frequency $\lambda\ge1$ and $x_0\in\R$, set
\begin{equation}\label{eq:scaled-x-weight}
 a_{\lambda,x_0}(x)=a\left(\frac{\lambda(x-x_0)}R\right).
\end{equation}
For a $y$ frequency $\mu\ge1$, define $a_{\mu,y_0}$ analogously.

\begin{lemma}\label{lem:local-mean}
Let $H$ be a Hilbert space.  If $f:\R\to H$ has Fourier support in $|\xi|\le C\lambda$, then
\begin{equation}\label{eq:local-mean-x}
 \sup_{x_0\in\R}\norm{f(x_0)}_H^2
 \lesssim_R
 \lambda\sup_{x_0\in\R}
 \int_\R a'\left(\frac{\lambda(x-x_0)}R\right)
 \norm{f(x)}_H^2\dd x.
\end{equation}
The analogous statement holds in the $y$ variable.
\end{lemma}

\begin{proof}
It suffices first to consider $H$-valued Schwartz functions; the
general case follows by frequency-preserving approximation.  Choose
$\rho\in\mathcal S(\R)$ such that
\[
 \widehat\rho(\xi)=1\qquad (|\xi|\le C),
\]
and set $\rho_\lambda(x)=\lambda\rho(\lambda x)$.  The Fourier-support
hypothesis gives the reproducing formula
\begin{equation}\label{eq:local-mean-reproduction}
 f(x_0)=\int_\R \rho_\lambda(x_0-x)f(x)\dd x .
\end{equation}
Let
\[
 \mathfrak m=\sup_{x\in\R}\norm{f(x)}_H
 \quad\hbox{and}\quad
 I_{x_0}=\left\{x\in\R:|x-x_0|\le\frac R\lambda\right\}.
\]
Splitting \eqref{eq:local-mean-reproduction} over $I_{x_0}$ and its
complement, and applying the triangle inequality for the Bochner
integral, yields
\begin{align}
 \norm{f(x_0)}_H
 &\le
 \int_{I_{x_0}}|\rho_\lambda(x_0-x)|\norm{f(x)}_H\dd x
 \notag\\
 &\quad+
 \int_{\R\setminus I_{x_0}}
 |\rho_\lambda(x_0-x)|\norm{f(x)}_H\dd x .
 \label{eq:local-mean-split}
\end{align}
By Cauchy--Schwarz and the scaling of $\rho_\lambda$, the first term
on the right is bounded by
\begin{equation}\label{eq:local-mean-near}
 \norm{\rho_\lambda}_{L^2}
 \left(\int_{I_{x_0}}\norm{f(x)}_H^2\dd x\right)^{1/2}
 \le
 \lambda^{1/2}\norm\rho_{L^2}
 \left(\int_{I_{x_0}}\norm{f(x)}_H^2\dd x\right)^{1/2}.
\end{equation}
For the complementary term, the change of variables
$r=\lambda(x-x_0)$ gives
\begin{equation}\label{eq:local-mean-tail}
 \int_{\R\setminus I_{x_0}}
 |\rho_\lambda(x_0-x)|\norm{f(x)}_H\dd x
 \le \mathfrak m\int_{|r|>R}|\rho(r)|\dd r.
\end{equation}
Since $\rho$ is Schwartz, the fixed constant $R$ may be chosen large
enough that the last integral is at most $1/2$.  Taking the supremum in
$x_0$ in \eqref{eq:local-mean-split}, and then absorbing the resulting
$\mathfrak m/2$ term, gives
\begin{equation}\label{eq:local-mean-unweighted}
 \mathfrak m^2\lesssim_R
 \lambda\sup_{x_0\in\R}
 \int_{I_{x_0}}\norm{f(x)}_H^2\dd x.
\end{equation}
If $x\in I_{x_0}$, then
$|\lambda(x-x_0)/R|\le1$, and hence \eqref{eq:w-lower} and
\eqref{eq:aprime-square} imply
\[
 a'\left(\frac{\lambda(x-x_0)}R\right)
 =\psi_0\left(\frac{\lambda(x-x_0)}R\right)^2
 \ge c_\psi^2.
\]
Substituting this lower bound into \eqref{eq:local-mean-unweighted}
proves \eqref{eq:local-mean-x}.  The argument in the $y$ variable is
identical.  Only the triangle inequality for Bochner integrals and
Cauchy--Schwarz are used, so the proof applies without change to
Hilbert-valued functions.
\end{proof}

For a smooth solution and an exponent $\sigma>1$ define
\begin{align}
 S_\sigma(u;T)^2
 ={}&\sum_{N\ge2}\sum_{\tau=\pm}
 \norm{J^\sigma|v_x(D)|^{1/2}\chi_{x\tau,N}(D)P_Nu}_{L_x^\infty L_{y,T}^2}^2
 \notag\\
 &+\sum_{N\ge2}\sum_{\tau=\pm}
 \norm{J^\sigma|v_y(D)|^{1/2}\chi_{y\tau,N}(D)P_Nu}_{L_y^\infty L_{x,T}^2}^2.
 \label{eq:S-def}
\end{align}
The nonlinear smoothing estimate (\cref{prop:smoothing}) is proved for
every $1<\sigma<s$; the strict inequality $\sigma<s$ absorbs the
logarithmic losses of \cref{lem:tangential}.  In the sequel we write
$z_N=P_NJ^\sigma u$; the exponent $\sigma$ is fixed in
\cref{prop:coupled} as $\sigma=s-\eps_0$.  For the $x$ charts, the
weight scale is $\lambda\simeq N$ (every $P_N$ function has $x$
frequency $O(N)$, so \cref{lem:local-mean} applies at that single
scale); on the $y$ charts, \eqref{eq:ygeometry} fixes the scale
$\mu\simeq N^{1/2}$.  The low-frequency component $P_1u$ is
not included in \eqref{eq:S-def}; throughout this section it is
estimated directly by Bernstein's inequality and the energy norm.

We now fix the geometric constants in the order used below.  Choose
$K\gg1$.  The proof of the resonance bound uses only that, on a band
of width $cN$, the error $v_x\alpha$ is bounded by
$C(c/K)N^{3/2}|\beta|$.  We therefore fix a number
$c_{\mathrm{res}}(K)>0$ such that this error, together with the fixed
enlargement errors, is at most one quarter of the lower bound for
$|v_y\beta|$ whenever $c\le c_{\mathrm{res}}(K)$.  Next choose the
constant $c_0$ in \eqref{eq:xcharts}-\eqref{eq:ycharts} so small
that
\begin{equation}\label{eq:c0-resonance}
 32c_0\le c_{\mathrm{res}}(K),
\end{equation}
and finally choose
\begin{equation}\label{eq:constant-order}
 0<\delta\ll\min\{c_0^2,(4K^2)^{-1}\}.
\end{equation}

\begin{definition}\label{def:chart-ladder}
For every chart $\nu$ and every $N\ge2$ choose
symbols
\[
 \Theta_{\nu,N,0},\Theta_{\nu,N,1},
 \Theta_{\nu,N,2},\Theta_{\nu,N,3}
\]
with $\Theta_{\nu,N,0}=\chi_{\nu,N}$ and the following properties.
Each symbol obeys \eqref{eq:chart-symbol} and is supported where the
relevant velocity has the fixed sign of the chart.  For
$\ell=0,1,2$, $\Theta_{\nu,N,\ell+1}=1$ on an
$O(\delta N)$ anisotropic neighbourhood of the transition support of
$\Theta_{\nu,N,\ell}$.  On the $y$ charts the successive supports
are contained in the bands with parameters
$4c_0,8c_0,16c_0,32c_0$.  On the $x$ charts, the auxiliary symbols
$\Theta_{\nu,N,\ell}$, $\ell\ge1$, are confined to the corresponding
enlarged transition bands, where $|v_x|\simeq N$.  All constants are
uniform in $\ell$.
\end{definition}

The sequence contains four fixed levels.  It is used only for the
commutators generated by the transition of the microlocal cutoffs: the
estimate at level $\ell$ is expressed in terms of the positive quantity at
level $\ell+1$, while the last level is bounded directly.

\subsection{Positive commutators}

We first record the only point at which the kink of $\xi|\xi|$ at
$\xi=0$ enters.

\begin{lemma}\label{lem:hilbert-weight}
Let $b\in C^\infty(\R)$ be bounded with $b'\in\Sscr(\R)$.  On
Schwartz functions,
\begin{equation}\label{eq:hilbert-weight}
 i[-D_x|D_x|,b]
 =-\bigl(b'|D_x|+|D_x|b'\bigr)+\mathcal R_b,
\end{equation}
where the Fourier kernel of the remainder is
\begin{equation}\label{eq:hilbert-kernel}
 \widehat{\mathcal R_b f}(\xi)
 =i\int_\R
 \bigl(\xi|\xi'|-\xi'|\xi|\bigr)
 \widehat b(\xi-\xi')\widehat f(\xi')\dd\xi'.
\end{equation}
The kernel vanishes when $\xi$ and $\xi'$ have the same sign, and
\begin{equation}\label{eq:hilbert-L2}
 \norm{\mathcal R_b}_{L^2\to L^2}
 \le \frac12\norm{\widehat{b''}}_{L^1}.
\end{equation}
For $b(x)=a(\lambda(x-x_0)/R)$, with $a$ defined by
\eqref{eq:aprime-square}, the remainder is supported in
\begin{equation}\label{eq:hilbert-lowxi}
 |\xi-\xi'|\lesssim \lambda/R,
 \qquad
 \sgn\xi\ne\sgn\xi',
 \qquad
 \max\{|\xi|,|\xi'|\}\lesssim\lambda/R.
\end{equation}
Moreover, for every Hilbert-valued $f$,
\begin{align}\label{eq:hilbert-strip}
 |\ip{\mathcal R_bf}{f}|
 \lesssim{}&\left(\frac\lambda R\right)^2
 \sum_{d\in\Z}\langle d\rangle^{-2}
 \int_\R
 a'\left(\frac{\lambda(x-x_0-dR/\lambda)}R\right)
 \norm{P^x_{\lesssim\lambda/R}f(x)}_H^2\dd x.
\end{align}
\end{lemma}

\begin{proof}
A direct Fourier computation gives \eqref{eq:hilbert-kernel}.  If
$\xi\xi'\ge0$, its multiplier is zero.  If the signs are opposite,
\[
 |\xi|\xi'|-\xi'|\xi||=2|\xi\xi'|
 \le\frac12|\xi-\xi'|^2,
\]
which proves \eqref{eq:hilbert-L2} by Schur's test.  For the scaled
weight, $b'=(\lambda/R)a'(\lambda(x-x_0)/R)$ and
$\widehat{a'}=\widehat\psi_0*\widehat\psi_0$ is compactly supported.
Away from the harmless point $\xi=\xi'$, the distribution
$\widehat b(\xi-\xi')$ therefore has the same support as
$\widehat{b'}$, giving the first condition in
\eqref{eq:hilbert-lowxi}.  Opposite signs then imply the last
condition.  Put $h=R/\lambda$ and rescale
$x=x_0+hX$, $y=x_0+hY$.  After the cancellation at $\xi=\xi'$, the rescaled
amplitude in \eqref{eq:hilbert-kernel} is compactly supported and is
piecewise smooth across the coordinate axes.  Its distributional
frequency derivatives through the order needed below are finite
measures: the factor
$\xi|\xi'|-\xi'|\xi|$ vanishes quadratically at the intersection of
the sign regions.  Two integrations by parts in each frequency
variable therefore give
\[
 |K_b(x,y)|\le Ch^{-3}
 \left\langle\frac{x-x_0}{h}\right\rangle^{-2}
 \left\langle\frac{y-x_0}{h}\right\rangle^{-2}.
\]
Schur's test gives the expected $h^{-2}$ operator size.  Decomposing
both variables into intervals centred at $x_0+dh$ and using
\eqref{eq:w-lower} to dominate their characteristic functions by
translated copies of $a'=\psi_0^2$ proves \eqref{eq:hilbert-strip}.
The argument is unchanged for Hilbert-valued functions.
\end{proof}

The following elementary identity fixes the sign convention used in
all weighted estimates.

\begin{lemma}\label{lem:weighted-energy}
Let $z$ solve
\begin{equation}\label{eq:weighted-linear-equation}
 z_t-i\omega(D)z=G
\end{equation}
on a time interval, and let $A$ be multiplication by a real bounded
function.  Then
\begin{equation}\label{eq:weighted-energy-identity}
 \frac{\dd}{\dd t}\ip{Az}{z}
 +\ip{i[\omega(D),A]z}{z}
 =2\operatorname{Re}\ip{AG}{z}.
\end{equation}
If $A=\tau a_*$ on a microlocal chart and
$\tau=\sgn v_\nu$ there, the principal symbol of the commutator term
on the left is
\begin{equation}\label{eq:weighted-positive-sign}
 \tau a_*'v_\nu=a_*'|v_\nu|\ge0.
\end{equation}
\end{lemma}

\begin{proof}
Differentiate $\ip{Az}{z}$, substitute
$z_t=i\omega(D)z+G$, and use that $A$ and $\omega(D)$ are
self-adjoint.  This gives \eqref{eq:weighted-energy-identity};
\eqref{eq:weighted-positive-sign} follows from the definition of
$\tau$.
\end{proof}

The following proposition includes the frequency-adapted form of the signed microlocal smoothing estimate.

\begin{proposition}\label{prop:positive}
Let $1<\sigma<s$, $z_N=P_NJ^\sigma u$, and
$z_{\nu,N}=\chi_{\nu,N}(D)z_N$.  On an $x$ chart take the single
weight scale $\lambda\simeq N$.  If $\tau\in\{\pm1\}$ is the sign of
$v_x$ on the chart, the commutator with
$\tau a_{\lambda,x_0}$ has the form
\begin{equation}\label{eq:positive-x}
 \ip{i[\omega(D),\tau a_{\lambda,x_0}]z_{\nu,N}}{z_{\nu,N}}
 =\norm{Q_{x,\nu,x_0}z_N}_2^2
 +\ip{R_{x,\nu,x_0}z_N}{z_N},
\end{equation}
where
\[
 Q_{x,\nu,x_0}
 =\Op\left(
 \left(\frac\lambda R\right)^{1/2}
 \psi_0\left(\frac{\lambda(x-x_0)}R\right)
 |v_x|^{1/2}\chi_{\nu,N}\right).
\]
This uses the smooth factorization $a'=\psi_0^2$ from
\eqref{eq:aprime-square}; in particular, no nonsmooth choice of the
square root of $a'$ is involved.  On a $y$ chart the analogous
identity holds with $\mu\simeq N^{1/2}$, $y_0$, and $v_y$.  For $R$
sufficiently large,
\begin{equation}\label{eq:remainder-absorb}
 |\ip{R_{r,\nu,r_0}z_N}{z_N}|
 \le \frac14\sup_{r_0'\in\R}\norm{Q_{r,\nu,r_0'}z_N}_2^2
 +C_R\norm{z_N}_2^2,
\end{equation}
where $r=x$ or $y$ and $r_0=x_0$ or $y_0$.  The constants are
uniform in $N$ and the spatial centres.  For every smooth
frequency-truncated function,
\begin{equation}\label{eq:positive-finite}
 \sup_{r_0\in\R}\norm{Q_{r,\nu,r_0}z_N}_2^2
 \le C_RN^3\norm{z_N}_2^2<\infty.
\end{equation}
Thus the supremum in \eqref{eq:remainder-absorb} is finite before any
absorption is performed.

Finally, the statement holds verbatim with $\chi_{\nu,N}$ replaced by
any of the auxiliary symbols $\Theta_{\nu,N,\ell}$,
$\ell\in\{0,1,2,3\}$, of \cref{def:chart-ladder}, with constants
uniform in $\ell$.  Indeed, the proof uses only the symbol bounds
\eqref{eq:chart-symbol} and the fact that $v_\nu$ has a fixed sign on
the support, both of which are imposed on every level in
\cref{def:chart-ladder}.
\end{proposition}

\begin{proof}
	We write the argument in a form that treats the two kinds of charts
	simultaneously.  Set
	\[
	\kappa_x=\lambda\simeq N,
	\qquad
	\kappa_y=\mu\simeq N^{1/2},
	\qquad
	h_r=\frac{\kappa_r}{R},
	\]
	where $r=x$ or $r=y$, and define
	\[
	w_{r_0}(r)
	=
	\psi_0\left(\frac{\kappa_r(r-r_0)}{R}\right).
	\]
	Thus, if $b=a_{\lambda,x_0}$ on an $x$ chart or
	$b=a_{\mu,y_0}$ on a $y$ chart, then
	\[
	\partial_r b=h_r w_{r_0}^2.
	\]
	Let
	\[
	f_{\nu,N}=\chi_{\nu,N}(D)z_N.
	\]
	
	We first identify the commutators exactly.  Since
	\[
	\omega(D)=D_xD_y^2-D_x|D_x|,
	\]
	the identity in \cref{lem:hilbert-weight} gives, for an $x$-dependent
	weight $b$,
	\[
	i[\omega(D),b]
	=
	\frac12\left(
	M_{b'}v_x(D)+v_x(D)M_{b'}
	\right)
	+\mathcal R_b.
	\]
	Indeed, $D_y$ commutes with $b$ and
	\[
	\frac12\left(
	M_{b'}v_x(D)+v_x(D)M_{b'}
	\right)
	=
	b'D_y^2-\bigl(b'|D_x|+|D_x|b'\bigr).
	\]
	Thus the only nonpolynomial contribution is the remainder
	$\mathcal R_b$ from \cref{lem:hilbert-weight}.
	
	For a $y$-dependent weight, the multiplier $D_x|D_x|$ commutes with
	$b$, while a direct computation gives
	\begin{equation}\label{eq:exact-y-comm}
		i[\omega(D),b]
		=
		2b'D_xD_y-i b''D_x
		=
		\frac12\left(
		M_{b'}v_y(D)+v_y(D)M_{b'}
		\right).
	\end{equation}
	Consequently, the term involving $b''D_x$ is not an additional
	heuristic error: it is precisely the correction that makes the
	principal $y$-commutator symmetric.
	
	Fix a chart and let $\tau$ be the sign of $v_r$ on its support.
	Choose a real smooth symbol $p_{\nu,N}^{(r)}$ that agrees with
	$\tau v_r$ on a neighbourhood of $\supp\chi_{\nu,N}$ and is strictly
	positive on a slightly larger neighbourhood.  This is possible
	because
	\[
	\tau v_r\ge c_0N
	\]
	on the chart.  Put
	\[
	q_{\nu,N}^{(r)}
	=
	\bigl(p_{\nu,N}^{(r)}\bigr)^{1/2}.
	\]
	The square root is therefore smooth on the relevant frequency
	region, and
	\[
	q_{\nu,N}^{(r)}(D)f_{\nu,N}
	=
	|v_r(D)|^{1/2}f_{\nu,N}.
	\]
	In particular,
	\[
	Q_{r,\nu,r_0}z_N
	=
	h_r^{1/2}
	M_{w_{r_0}}
	q_{\nu,N}^{(r)}(D)f_{\nu,N},
	\]
	which agrees with the operator stated in the proposition.
	
	The symmetrized principal quadratic form can now be factorized
	without using sharp G{\aa}rding.  Writing
	$q=q_{\nu,N}^{(r)}(D)$ and $w=w_{r_0}$, we have
	\[
	\begin{aligned}
		&\frac{h_r}{2}
		\ip{
			\bigl(M_{w^2}q^2+q^2M_{w^2}\bigr)f_{\nu,N}
		}{f_{\nu,N}}
		\\
		&\qquad
		=
		h_r\ip{qM_{w^2}qf_{\nu,N}}{f_{\nu,N}}
		+
		\frac{h_r}{2}
		\ip{
			\bigl[\,[M_{w^2},q],q\,\bigr]f_{\nu,N}
		}{f_{\nu,N}}
		\\
		&\qquad
		=
		\norm{Q_{r,\nu,r_0}z_N}_2^2
		+
		\frac{h_r}{2}
		\ip{
			\bigl[\,[M_{w^2},q],q\,\bigr]f_{\nu,N}
		}{f_{\nu,N}}.
	\end{aligned}
	\]
	This is an exact algebraic identity.
	
	It remains to estimate the double commutator.  On an enlarged chart,
	the symbol $q_{\nu,N}^{(r)}$ satisfies the relative derivative bounds
	\[
	\left|
	\partial_\xi^\alpha\partial_\eta^\beta
	q_{\nu,N}^{(r)}
	\right|
	\lesssim_{\alpha,\beta}
	q_{\nu,N}^{(r)}
	N^{-\alpha-\beta/2}.
	\]
	On an $x$ chart these bounds use, besides $\tau v_x\ge c_0N$, the
	elementary fact that $|v_x|\simeq\max\{|\xi|,\eta^2\}\gtrsim|\eta|N^{1/2}$
	there, verified in the proof of \cref{prop:product}.
	For an $x$ weight only $\xi$ derivatives occur in the composition
	with $M_{w^2}$, while for a $y$ weight only $\eta$ derivatives occur.
	Accordingly, each commutator contributes respectively the factor
	\[
	\frac{\lambda}{R}N^{-1}
	\lesssim R^{-1}
	\]
	or
	\[
	\frac{\mu}{R}N^{-1/2}
	\lesssim R^{-1}.
	\]
	The resulting symbol carries the second derivative
	$\partial_r^2(w_{r_0}^2)=h_r^2(\psi_0^2)''(\kappa_r(r-r_0)/R)$.
	Since $(\psi_0^2)''$ is Schwartz and $\psi_0^2\ge c_\psi^2$ on
	$[-2,2]$ by \eqref{eq:w-lower}, it is dominated by a rapidly
	convergent sum of translated copies of the weight,
	\[
	\bigl|(\psi_0^2)''(v)\bigr|
	\le
	C\sum_{d\in\Z}\langle d\rangle^{-2}\psi_0^2(v-d),
	\]
	exactly as in \eqref{eq:hilbert-strip}.  Applying the blockwise
	symbolic calculus
	\eqref{eq:Op-composition}--\eqref{eq:Op-commutator} and summing the
	translated strips therefore gives
	\[
	\left|
	\frac{h_r}{2}
	\ip{
		\bigl[\,[M_{w^2},q],q\,\bigr]f_{\nu,N}
	}{f_{\nu,N}}
	\right|
	\le
	\frac{C}{R}
	\sup_{r_0'\in\R}
	\norm{Q_{r,\nu,r_0'}z_N}_2^2
	+
	C_R\norm{z_N}_2^2.
	\]
	The constants are uniform in $N$ and in the spatial centre.  As for
	the Hilbert remainder treated next, the supremum over translated
	centres appears here because the derivatives of the weight profile
	are not pointwise dominated by the profile itself; this is the form
	in which the bound is used in \eqref{eq:remainder-absorb}.
	
	Suppose now that $r=x$.  By \eqref{eq:hilbert-lowxi}, the quadratic
	form associated with $\mathcal R_b$ is supported where
	\[
	|\xi|+|\xi'|\lesssim\frac{N}{R},
	\qquad
	\sgn\xi\ne\sgn\xi'.
	\]
	Since both frequency variables remain on the isotropic annulus
	$|\zeta|\simeq N$, this implies
	\[
	|\eta|\simeq N,
	\qquad
	v_x(\xi,\eta)
	=
	\eta^2-2|\xi|
	\simeq N^2.
	\]
	Hence the Hilbert remainder vanishes on the $x_-$ chart.  On the
	$x_+$ chart it is confined to a region on which
	$|v_x|^{1/2}\simeq N$.
	
	Let $\Pi_N^x$ be a smooth multiplier equal to one on the
	$x$-frequency support described in \eqref{eq:hilbert-lowxi}.  The
	strip estimate \eqref{eq:hilbert-strip}, applied with
	$H=L_y^2$, yields
	\[
	\begin{aligned}
		\left|\ip{\mathcal R_bf_{\nu,N}}{f_{\nu,N}}\right|
		\lesssim{}&
		\left(\frac{\lambda}{R}\right)^2
		\sum_{d\in\mathbb Z}\langle d\rangle^{-2}
		\\
		&\times
		\int_{\mathbb R^2}
		a'\left(
		\frac{\lambda(x-x_0-dR/\lambda)}{R}
		\right)
		|\Pi_N^xf_{\nu,N}(x,y)|^2
		\,\dd x\,\dd y.
	\end{aligned}
	\]
	On this support,
	\[
	\Pi_N^xf_{\nu,N}
	=
	m_N(D)\,
	|v_x(D)|^{1/2}f_{\nu,N},
	\]
	where $m_N$ is a smooth multiplier satisfying
	\[
	\norm{m_N(D)}_{L^2\to L^2}\lesssim N^{-1}.
	\]
	The kernel of $m_N(D)$ is integrable at the spatial scales
	$R/N$ in the $x$ variable and $N^{-1}$ in the $y$ variable.
	Decomposing its $x$ kernel into translates of the strips appearing
	in \eqref{eq:hilbert-strip} gives
	\[
	\begin{aligned}
		&\int_{\mathbb R^2}
		a'\left(
		\frac{\lambda(x-x_0-dR/\lambda)}{R}
		\right)
		|\Pi_N^xf_{\nu,N}|^2
		\,\dd x\,\dd y
		\\
		&\qquad\lesssim
		N^{-2}
		\sum_{e\in\mathbb Z}\langle e\rangle^{-2}
		\norm{
			w_{x_0+(d+e)R/\lambda}
			|v_x(D)|^{1/2}f_{\nu,N}
		}_2^2.
	\end{aligned}
	\]
	Because $\lambda\simeq N$ and
	\[
	\norm{
		w_{x_0'}
		|v_x(D)|^{1/2}f_{\nu,N}
	}_2^2
	=
	\frac{R}{\lambda}
	\norm{Q_{x,\nu,x_0'}z_N}_2^2,
	\]
	the summability of the translated-strip coefficients gives
	\[
	\left|\ip{\mathcal R_bf_{\nu,N}}{f_{\nu,N}}\right|
	\le
	\frac{C}{RN}
	\sup_{x_0'\in\mathbb R}
	\norm{Q_{x,\nu,x_0'}z_N}_2^2
	+
	C_R\norm{z_N}_2^2.
	\]
	Since $N\ge2$, the first coefficient is bounded by $C/R$.  This proves
	that the kink contribution is controlled by the same positive
	quantity, with the supremum over translated centres required in
	\eqref{eq:remainder-absorb}.
	
	Combining the preceding estimates, we obtain
	\[
	\left|
	\ip{R_{r,\nu,r_0}z_N}{z_N}
	\right|
	\le
	\frac{C}{R}
	\sup_{r_0'\in\mathbb R}
	\norm{Q_{r,\nu,r_0'}z_N}_2^2
	+
	C_R\norm{z_N}_2^2.
	\]
	Choosing $R$ sufficiently large proves
	\eqref{eq:remainder-absorb}.
	
	For completeness, the supremum used above is finite before the
	absorption argument.  Indeed, on an $x$ chart,
	\[
	h_x|v_x|\lesssim_R N^3,
	\]
	whereas on a $y$ chart,
	\[
	h_y|v_y|\lesssim_R N^2.
	\]
	Since $w_{r_0}$ is uniformly bounded, these estimates imply
	\[
	\sup_{r_0\in\mathbb R}
	\norm{Q_{r,\nu,r_0}z_N}_2^2
	\le
	C_RN^3\norm{z_N}_2^2,
	\]
	which is \eqref{eq:positive-finite}.
	
	The same proof applies to each auxiliary cutoff
	$\Theta_{\nu,N,\ell}$.  The required sign condition and rescaled
	symbol estimates hold uniformly in $\ell$ by
	\cref{def:chart-ladder}; for the auxiliary $x$-chart levels, the
	Hilbert remainder is either absent by support separation or is
	estimated by the same translated-strip argument.  Finally, one first
	applies the identities to Schwartz functions with the stated
	frequency localization.  Standard frequency-preserving
	approximation, together with Fatou's lemma for the nonnegative
	$Q^*Q$ term, yields the result for general smooth
	frequency-truncated functions.
\end{proof}

Combining \cref{lem:local-mean,prop:positive}, the positive term controls precisely the continuous smoothing norms in \eqref{eq:S-def}.

\subsection{The commutator near the characteristic curve}

After localizing the equation by $\chi_{\nu,N}(D)$ and pairing it with
the correspondingly localized weighted factor, the relevant
symmetrized chart commutator is
\begin{equation}\label{eq:chart-commutator}
 [a_{\nu,N}(D),u_{\ll N}]\partial_xz_N,
 \qquad a_{\nu,N}=\chi_{\nu,N}^2.
\end{equation}
The terms generated by commuting one copy of $\chi_{\nu,N}(D)$
through the spatial weight are lower-order translated-strip errors and
are included in the remainder estimate of \cref{prop:positive}.
The low frequency is measured by
\begin{equation}\label{eq:rho}
 \rho(\alpha,\beta)=|\alpha|+\beta^2,
 \qquad
 \rho(\alpha,\beta)\le\delta N.
\end{equation}

We first record the high-variable expansion that converts anisotropic
symbol bounds into the bilinear estimates used throughout the
commutator and weighted-energy arguments.

For a bilinear symbol $m_N$, write
\begin{equation}\label{eq:Bmn}
	\B_{m_N}(f,g)(x)
	=
	\iint e^{ix\cdot(\theta+\zeta)}
	m_N(\theta,\zeta)
	\widehat f(\theta)\widehat g(\zeta)
	\dd\theta\dd\zeta.
\end{equation}

\begin{lemma}\label{lem:high-expansion}
Suppose that $m_N$ is supported where
$\rho(\theta)\le\delta N$ and in a fixed parabolic high-frequency
block, and assume
\begin{equation}\label{eq:high-derivatives}
	|\partial_\xi^a\partial_\eta^bm_N(\theta,\zeta)|
	\le C_{a,b}N^{-1-a-b/2}.
\end{equation}
No regularity in $\theta$ is assumed. If $s>1$, then
\begin{equation}\label{eq:B-basic}
	\norm{\B_{m_N}(f,g)}_2
	\lesssim_s
	N^{-1}\norm f_{H^s}\norm g_2.
\end{equation}
If
\[
d_N(\theta,\zeta)
=
i\alpha q_N(\theta)m_N(\theta,\zeta),
\]
where $q_N$ is uniformly bounded and supported in
$\rho(\theta)\le\delta N$, then
\begin{equation}\label{eq:B-low-derivative}
	\norm{\B_{d_N}(F,g)}_2
	\lesssim_s
	\norm F_{H^s}\norm g_2.
\end{equation}
\end{lemma}

\begin{proof}
Rescale only the high-frequency variable:
\[
\xi=N\widetilde\xi,
\qquad
\eta=N^{1/2}\widetilde\eta.
\]
Choose $\psi_N$ equal to one on the support of $m_N$ in the
$\zeta$ variable and supported strictly inside the corresponding
periodization cell. A Fourier series in the rescaled high variable
gives
\begin{equation}\label{eq:high-series}
	m_N(\theta,\zeta)
	=
	N^{-1}\sum_{k\in\mathbb Z^2}
	c_{N,k}(\theta)
	e^{2\pi i(k_1\xi/N+k_2\eta/N^{1/2})}
	\psi_N(\zeta).
\end{equation}
Integration by parts in the rescaled high variable, using
\eqref{eq:high-derivatives}, yields, for every $M$,
\[
\sup_\theta|c_{N,k}(\theta)|
\le C_M\langle k\rangle^{-M}.
\]
No derivative of $c_{N,k}$ in $\theta$ is used.

For fixed $k$, the multiplier $c_{N,k}(D)$ is bounded on $H^s$ by
Plancherel's theorem. Since $s>1$,
\[
\norm{c_{N,k}(D)f}_{L^\infty}
\lesssim_s
\norm{c_{N,k}(D)f}_{H^s}
\lesssim_s
\langle k\rangle^{-M}\norm f_{H^s}.
\]
The exponential factor in \eqref{eq:high-series} translates the
high-frequency factor by
\[
\left(
\frac{2\pi k_1}{N},
\frac{2\pi k_2}{N^{1/2}}
\right)
\]
and therefore preserves its $L^2$ norm. Summing in $k$ proves
\eqref{eq:B-basic}.

For \eqref{eq:B-low-derivative}, apply the same expansion to
$d_N$. On the low-frequency support, $|\alpha|\lesssim N$, while
$q_N$ is uniformly bounded. Hence the factor $|\alpha|$ cancels the
prefactor $N^{-1}$ in \eqref{eq:high-series}. The preceding argument
then gives \eqref{eq:B-low-derivative}.
\end{proof}

Write $a_N=a_{\nu,N}=\chi_{\nu,N}^2$.  The bilinear symbol of
\eqref{eq:chart-commutator} is
\begin{equation}\label{eq:mN}
 m_N(\theta,\zeta)
 =i\xi\bigl(a_N(\zeta+\theta)-a_N(\zeta)\bigr),
 \qquad \theta=(\alpha,\beta),\quad \zeta=(\xi,\eta).
\end{equation}
By the fundamental theorem of calculus,
\begin{align}
 m_{N,x}
 &=i\xi\alpha\int_0^1\partial_\xi a_N(\zeta+\tau\theta)\dd\tau,
 \label{eq:mNx}\\
 m_{N,y}
 &=i\xi\beta\int_0^1\partial_\eta a_N(\zeta+\tau\theta)\dd\tau.
 \label{eq:mNy}
\end{align}
The term \eqref{eq:mNx} is controlled by $\partial_xu$.  Choose a
smooth symbol $\chi_{N,\Gamma}^{\mathrm{tr}}(\theta,\zeta)$, equal to one
when $\zeta+\tau\theta$ meets a chart transition for some
$\tau\in[0,1]$, supported in a fixed enlargement of
\begin{equation}\label{eq:band}
 |\xi|\simeq N,
 \qquad |\eta|\simeq N^{1/2},
 \qquad |\eta^2-2|\xi||\lesssim c_0N,
\end{equation}
and satisfying the mixed symbol bounds of \eqref{eq:chart-symbol}.
Define
\begin{equation}\label{eq:mNy-band-def}
 m_{N,\Gamma}
 :=m_{N,y}\chi_{N,\Gamma}^{\mathrm{tr}}.
\end{equation}
Away from this band, $\partial_\eta a_N=O(N^{-1})$ and
\eqref{eq:mNy} is controlled by $\partial_yu$.  The dangerous part
has size
\begin{equation}\label{eq:danger-size}
 |m_{N,\Gamma}(\theta,\zeta)|
 \lesssim N^{1/2}|\beta|.
\end{equation}

We distinguish two regions in the low-frequency plane.  Let
\begin{equation}\label{eq:tan-normal}
 \mathcal S_N^{\mathrm{tan}}
 =\{|\beta|\le K N^{-1/2}|\alpha|\},
 \qquad
 \mathcal S_N^{\mathrm{nor}}
 =\{|\beta|>K N^{-1/2}|\alpha|\}.
\end{equation}
Choose an even function $\phi\in C^\infty(\R)$ satisfying
$\phi(r)=0$ for $|r|\le K$ and $\phi(r)=1$ for $|r|\ge2K$, and set
\begin{equation}\label{eq:sector-partition}
 \psi_{\mathrm{nor}}(\theta)=
 \phi\!\left(\frac{N^{1/2}\beta}{\alpha}\right),
 \qquad
 \psi_{\mathrm{tan}}(\theta)=1-\psi_{\mathrm{nor}}(\theta).
\end{equation}
The quotient in \eqref{eq:sector-partition} is interpreted by
homogeneity: $\psi_{\mathrm{nor}}(0,\beta)=1$ for $\beta\ne0$; the
value at $\theta=0$ is immaterial.  Thus $\psi_{\mathrm{tan}}$ is
supported in $\{|\beta|\le2KN^{-1/2}|\alpha|\}$, whereas
$\psi_{\mathrm{nor}}$ is supported in $\mathcal S_N^{\mathrm{nor}}$.
The functions $\psi_{\mathrm{tan}}$ and $\psi_{\mathrm{nor}}$ are
homogeneous of degree zero in $\theta$ and are not
smooth at the origin.  Consequently the estimate in the first region
cannot be deduced from a pointwise bound for the symbol.  We use an
anisotropic dyadic decomposition of the low frequency.

\begin{lemma}\label{lem:tangential}
Let $\nu$ be any chart and let $m_{N,\Gamma}$ be the symbol defined in
\eqref{eq:mNy-band-def}, whose high-frequency support is contained in
\eqref{eq:band}.  Set
$m_{N,1}=m_{N,\Gamma}\psi_{\mathrm{tan}}$.
Assume $\delta\le\min\{c_0^2,(4K^2)^{-1}\}$.  Then, for every $f$ with
Fourier support in $\{\rho(\theta)\le\delta N\}$ and every
$g\in L^2(\R^2)$,
\begin{equation}\label{eq:tangent-log}
 \norm{\B_{m_{N,1}}(f,g)}_{L^2}
 \le C_K\bigl[(1+\log N)\norm{\partial_xf}_{L^\infty}
 +\norm{f}_{L^\infty}\bigr]\norm g_{L^2}.
\end{equation}
In particular, uniformly in the weight centre,
\[
 \abs{\ip{\B_{m_{N,1}}(u_{\ll N},z_N)}{a_*z_N}}
 \le C_K\bigl[(1+\log N)\norm{u_x(t)}_{L^\infty}
 +\norm{u(t)}_{L^\infty}\bigr]\norm{z_N(t)}_2^2.
\]
\end{lemma}

\begin{proof}
Throughout, split into $\pm\alpha>0$; we treat $\alpha>0$.  Decompose
the low frequency into anisotropic shells: choose
$\varphi_j\in C_c^\infty$ with
$\sum_{j\ge0}\varphi_j(\theta)=1$ on $\{0<\rho(\theta)\le\delta N\}$
and $\supp\varphi_j\subset\{\rho(\theta)\simeq\rho_j\}$, where
$\rho_j=2^{-j}\delta N$.  Write
$m_j=m_{N,1}\varphi_j$.

\smallskip
\emph{Step 1: geometry of the dyadic shells.}
On $\supp m_j$ one has $\beta^2\le4K^2N^{-1}\alpha^2$ and
$\alpha+\beta^2\simeq\rho_j$.  If $\alpha\le\rho_j/2$ then
$\beta^2\ge\rho_j/2$, whence
$\rho_j\le8K^2N^{-1}\alpha^2\le2K^2N^{-1}\rho_j^2$, i.e.\
$\rho_j\ge N/(2K^2)$, contradicting $\rho_j\le\delta N\le N/(4K^2)$.
Therefore
\begin{equation}\label{eq:alpha-comparable}
 \alpha\simeq\rho_j
 \qquad\text{on }\supp m_j,\ \text{for every }j\ge0.
\end{equation}
Consequently $\supp_\theta m_j$ is contained in an anisotropic box of
sides
\begin{equation}\label{eq:shell-scales}
 r_\alpha=C\rho_j,
 \qquad
 r_\beta=CKN^{-1/2}\rho_j,
\end{equation}
and direct differentiation of the four $\theta$-dependent factors of
$m_j$, namely $\beta$, $\psi_{\mathrm{tan}}$, $\varphi_j$, and
the translate $\int_0^1\partial_\eta a_N(\zeta+\tau\theta)\dd\tau$,
gives, on $\supp m_j$ and using \eqref{eq:alpha-comparable} together
with \eqref{eq:chart-symbol},
\begin{equation}\label{eq:shell-class}
 \abs{\partial_\alpha^{c}\partial_\beta^{d}
 \partial_\xi^{a}\partial_\eta^{b}m_j(\theta,\zeta)}
 \le C_{a,b,c,d}\,A_j\,
 r_\alpha^{-c}\,r_\beta^{-d}\,N^{-a}\,N^{-b/2},
 \qquad
 A_j:=\min\{K\rho_j,\;CK\delta^{1/2}N\}.
\end{equation}
(The size $A_j$ comes from
$\abs{m_{N,1}}\lesssim N^{1/2}\abs\beta
\le2K\alpha\simeq K\rho_j$; the second entry of the minimum is the
crude bound $N^{1/2}\abs\beta\le\delta^{1/2}N$ and is not used below.)

\smallskip
\emph{Step 2: separated expansion with tame coefficients.}
Fix $j$.  On the fixed rescaled high-frequency box containing
$\supp_\zeta m_j$ (the doubled band; recall $\delta\le c_0^2$), expand
$m_j$ in a Fourier series in $\zeta$ as in the proof of
\cref{lem:high-expansion}:
\begin{equation}\label{eq:tame-series}
 m_j(\theta,\zeta)
 =A_j\sum_{k\in\Z^2}c_{j,k}(\theta)\,
 e^{2\pi i(k_1\xi/N+k_2\eta/N^{1/2})}\,\psi_N(\zeta),
\end{equation}
with $\psi_N$ a fixed cutoff equal to $1$ on the box.  Because
\eqref{eq:shell-class} controls \emph{mixed} derivatives, integration
by parts in $\zeta$ yields, for every $M$ and every multi-index
$(c,d)$,
\[
 \abs{\partial_\alpha^{c}\partial_\beta^{d}c_{j,k}(\theta)}
 \le C_{M,c,d}\,\langle k\rangle^{-M}\,
 r_\alpha^{-c}\,r_\beta^{-d},
\]
and $c_{j,k}$ is supported in the box \eqref{eq:shell-scales}.  Hence
the inverse Fourier transform satisfies
\[
 \abs{\check c_{j,k}(x,y)}
 \le C_M\langle k\rangle^{-M}\,
 r_\alpha r_\beta\,
 \langle r_\alpha x\rangle^{-2}\langle r_\beta y\rangle^{-2},
 \qquad\text{so}\qquad
 \norm{\check c_{j,k}}_{L^1(\R^2)}\le C_M\langle k\rangle^{-M},
\]
uniformly in $j$ and $N$: the support widths and the derivative scales
in \eqref{eq:shell-scales} are exactly reciprocal.  Therefore the
multipliers $c_{j,k}(D)$ are bounded on $L^\infty$:
\begin{equation}\label{eq:tame-Linfty}
 \norm{c_{j,k}(D)h}_{L^\infty}
 \le C_M\langle k\rangle^{-M}\norm h_{L^\infty}.
\end{equation}
Since the exponential in \eqref{eq:tame-series} translates the high
factor and preserves its $L^2$ norm,
\begin{equation}\label{eq:per-shell}
 \norm{\B_{m_j}(f,g)}_{L^2}
 \le C\,A_j\,\norm{c_{j,\cdot}(D)f}_{\ell^1_kL^\infty}\norm g_2
 \le C\,A_j\,\norm f_{L^\infty}\norm g_2.
\end{equation}

\smallskip
\emph{Step 3: summation.}
For the shells with $\rho_j\ge1$ (there are at most $C\log N$ of
them), factor out one $x$ derivative: by
\eqref{eq:alpha-comparable}, $n_j:=m_j/(i\alpha)$ obeys
\eqref{eq:shell-class} with $A_j$ replaced by $C_K$, so Step 2 applied
to $n_j$ gives
\[
 \norm{\B_{m_j}(f,g)}_{2}
 =\norm{\B_{n_j}(\partial_xf,g)}_{2}
 \le C_K\norm{\partial_xf}_{L^\infty}\norm g_2.
\]
Summing these $O(\log N)$ shells produces the first term of
\eqref{eq:tangent-log}.  For the shells with $\rho_j<1$, use
\eqref{eq:per-shell} directly: since $A_j\le K\rho_j$ and
$\sum_{\rho_j<1}\rho_j\le2$, their total contribution is
$C_K\norm f_{L^\infty}\norm g_2$, the second term of
\eqref{eq:tangent-log}.  The weighted pairing bound follows from
Cauchy--Schwarz and $\abs{a_*}\le C$, uniformly in the centre.
\end{proof}

The absorption of the factor $1+\log N$ is carried out in the proof
of \cref{prop:smoothing}, where the strict inequality $\sigma<s$
supplies the summable weight $N^{-2(s-\sigma)}(1+\log N)\le C_{s-\sigma}$.

\subsection{The resonance function and a cubic correction}

Define
\begin{equation}\label{eq:Omega}
 \Omega(\theta,\zeta)
 =\omega(\theta)+\omega(\zeta)-\omega(\zeta+\theta).
\end{equation}

\begin{lemma}\label{lem:res-identity}
Assume $\xi>0$ and $\xi+\alpha>0$.  Then
\begin{equation}\label{eq:res-identity}
 \Omega(\theta,\zeta)
 =-v_x(\zeta)\alpha-v_y(\zeta)\beta
 +(\alpha^2-\alpha|\alpha|)
 -2\eta\alpha\beta-\xi\beta^2.
\end{equation}
The identity on $\xi<0$ follows by oddness of $\omega$.
\end{lemma}

\begin{proof}
On $\xi>0$, $\omega(\xi,\eta)=\xi\eta^2-\xi^2$.  Expand $\omega(\zeta+\theta)$ and cancel the two $\alpha\beta^2$ terms.  The remaining terms are exactly \eqref{eq:res-identity}.
\end{proof}

\begin{proposition}\label{prop:res-lower}
Let the constants be chosen as in
\eqref{eq:c0-resonance} and \eqref{eq:constant-order}.  On the support of
$m_{N,\Gamma}\psi_{\mathrm{nor}}$,
\begin{equation}\label{eq:res-lower}
 |\Omega(\theta,\zeta)|
 \ge cN^{3/2}|\beta|.
\end{equation}
\end{proposition}

\begin{proof}
On the band, $|v_y|\ge c_1N^{3/2}$ and $|v_x|\le Cc_0N$.  The
definition of $\mathcal S_N^{\mathrm{nor}}$ gives
\[
 N^{3/2}|\beta|\ge KN|\alpha|.
\]
Thus $|v_x\alpha|$ is a small fraction of $N^{3/2}|\beta|$.  The remaining terms in \eqref{eq:res-identity} satisfy
\begin{align*}
 |2\eta\alpha\beta|&\lesssim\delta N^{3/2}|\beta|,\\
 |\xi\beta^2|&\lesssim\delta^{1/2}N^{3/2}|\beta|,\\
 |\alpha^2-\alpha|\alpha||&\lesssim(\delta/K)N^{3/2}|\beta|.
\end{align*}
The stated order of choices proves the result.
\end{proof}

With $\psi_{\mathrm{nor}}=1-\psi_{\mathrm{tan}}$ the cutoff
fixed in \eqref{eq:sector-partition}, define
\begin{equation}\label{eq:bN}
 b_N(\theta,\zeta)
 =\frac{m_{N,\Gamma}(\theta,\zeta)
 \psi_{\mathrm{nor}}(\theta)}{i\Omega(\theta,\zeta)}.
\end{equation}
By \cref{prop:res-lower},
\begin{equation}\label{eq:bN-size}
 |b_N(\theta,\zeta)|\lesssim N^{-1}.
\end{equation}

\begin{lemma}
\label{lem:bN-class}
On the support of $b_N$ (the doubled band in $\zeta$, the normal
sector with $\rho(\theta)\le\delta N$ in $\theta$, and a fixed
half-plane in each of $\xi$ and $\alpha$), the resonance function
satisfies
\begin{equation}\label{eq:Omega-derivatives}
\begin{gathered}
 \partial_\xi\Omega=2\alpha-2\eta\beta-\beta^2,
 \qquad
 \partial_\eta\Omega=-2\eta\alpha-2\xi\beta-2\alpha\beta,\\
 \partial_\eta^2\Omega=-2\alpha,
 \qquad
 \partial_\xi\partial_\eta\Omega=-2\beta,
 \qquad
 \partial_\xi^2\Omega=0,
\end{gathered}
\end{equation}
and consequently
\begin{equation}\label{eq:Omega-ratios}
 \abs{\partial_\xi\Omega}\le CN^{-1}\abs\Omega,
 \quad
 \abs{\partial_\eta\Omega}\le CN^{-1/2}\abs\Omega,
 \quad
 \abs{\partial_\eta^2\Omega}\le CN^{-1}\abs\Omega,
 \quad
 \abs{\partial_\xi\partial_\eta\Omega}\le CN^{-3/2}\abs\Omega.
\end{equation}
Therefore $b_N$ satisfies \eqref{eq:high-derivatives}:
\begin{equation}\label{eq:bN-highclass}
 \abs{\partial_\xi^a\partial_\eta^b\,b_N(\theta,\zeta)}
 \le C_{a,b}\,N^{-1-a-b/2},
 \qquad
 \text{uniformly in }\theta.
\end{equation}
No regularity of $b_N$ in $\theta$ is asserted.
\end{lemma}

\begin{proof}
The identities \eqref{eq:Omega-derivatives} follow by differentiating
\eqref{eq:res-identity} in $\zeta$ (the term
$\alpha^2-\alpha\abs\alpha$ is $\zeta$-independent).  For
\eqref{eq:Omega-ratios}, use \cref{prop:res-lower},
$\abs\Omega\ge cN^{3/2}\abs\beta$, together with
$\abs\alpha\le K^{-1}N^{1/2}\abs\beta$ and
$\abs\beta\le\delta^{1/2}N^{1/2}$: for instance
$\abs{2\alpha}\le2K^{-1}N^{1/2}\abs\beta\le CN^{-1}\abs\Omega$,
$\abs{2\eta\beta}\le CN^{1/2}\abs\beta\le CN^{-1}\abs\Omega$,
$\abs{\beta^2}\le\delta^{1/2}N^{1/2}\abs\beta\le CN^{-1}\abs\Omega$,
and similarly for the $\eta$ derivatives; the mixed second derivative
gives $2\abs\beta\le CN^{-3/2}\abs\Omega$, which is sharp.  Since
$\Omega$ is a polynomial of degree two in $\zeta$ on each half-plane,
Fa\`a di Bruno applied to $1/\Omega$ yields
$\abs{\partial_\xi^a\partial_\eta^b(1/\Omega)}
\le C_{a,b}N^{-a-b/2}\abs\Omega^{-1}$.  The factor
$m_{N,\Gamma}=i\xi\beta\int_0^1\partial_\eta
a_N(\zeta+\tau\theta)\dd\tau$ obeys
$\abs{\partial_\xi^a\partial_\eta^b m_{N,\Gamma}}
\le C_{a,b}N^{1/2}\abs\beta\,N^{-a-b/2}$ by
\eqref{eq:chart-symbol}, and the Leibniz rule combined with
\eqref{eq:bN-size} gives \eqref{eq:bN-highclass}.
\end{proof}
The angular cutoff is of Mikhlin type, rather than a classical
Coifman--Meyer symbol, in the low-frequency variable.  In the
arguments below it is never differentiated with respect to that
variable.

\subsection{Weighted cubic correction and adjacent localizations}

Let $a_*$ denote the appropriate scaled $x$- or $y$-weight.  For a
bilinear symbol $b$, introduce the trilinear form
\[
\Lambda_b(f,g,h)
=
\int_{\theta+\zeta+\zeta_3=0}
b(\theta,\zeta)
\widehat f(\theta)\widehat g(\zeta)\widehat h(\zeta_3)
\dd\theta\dd\zeta.
\]
The cubic correction is
\begin{equation}\label{eq:CN}
	\mathcal C_N(t)
	=
	\operatorname{Re}
	\Lambda_{b_N}(u_{\ll N},z_N,a_*z_N).
\end{equation}
Recall that
\[
\Ncal(u)=\frac12\partial_x(u^2),
\qquad
z_N=P_NJ^\sigma u,
\qquad
F_N=P_NJ^\sigma\Ncal(u).
\]
Since $P_{\ll N}$, $P_N$, and $J^\sigma$ commute with the linear
group, the three factors in \eqref{eq:CN} satisfy
\[
(\partial_t-i\omega(D))u_{\ll N}
=-P_{\ll N}\Ncal(u),
\]
\[
(\partial_t-i\omega(D))z_N=-F_N,
\]
and, because $a_*$ is independent of time,
\[
(\partial_t-i\omega(D))(a_*z_N)
=-a_*F_N-[i\omega(D),a_*]z_N.
\]
If $\theta+\zeta+\zeta_3=0$, the oddness of $\omega$ gives
\[
\omega(\theta)+\omega(\zeta)+\omega(\zeta_3)
=
\omega(\theta)+\omega(\zeta)-\omega(\theta+\zeta)
=
\Omega(\theta,\zeta).
\]
Consequently,
\begin{align*}
	\frac{\dd}{\dd t}\mathcal C_N
	={}&
	\operatorname{Re}
	\Lambda_{i\Omega b_N}(u_{\ll N},z_N,a_*z_N)\\
	&-\operatorname{Re}
	\Lambda_{b_N}(P_{\ll N}\Ncal(u),z_N,a_*z_N)\\
	&-\operatorname{Re}
	\Lambda_{b_N}(u_{\ll N},F_N,a_*z_N)\\
	&-\operatorname{Re}
	\Lambda_{b_N}(u_{\ll N},z_N,a_*F_N)\\
	&-\operatorname{Re}
	\Lambda_{b_N}
	\bigl(u_{\ll N},z_N,[i\omega(D),a_*]z_N\bigr).
\end{align*}
With the sign convention
\[
\mathcal T_{N,2}
=
-\operatorname{Re}
\Lambda_{i\Omega b_N}(u_{\ll N},z_N,a_*z_N)
\]
and
\[
\mathcal X_N
=
-\operatorname{Re}
\Lambda_{b_N}
\bigl(u_{\ll N},z_N,[i\omega(D),a_*]z_N\bigr),
\]
we obtain the exact identity
\begin{equation}\label{eq:C-identity}
	\frac{\dd}{\dd t}\mathcal C_N
	=-\mathcal T_{N,2}
	+\mathcal X_N+\mathcal R_N,
\end{equation}
where
\[
\mathcal R_N
=
\mathcal R_{N,0}
+\mathcal R_{N,1}
+\mathcal R_{N,2}
\]
and
\begin{align*}
	\mathcal R_{N,0}
	&=
	-\operatorname{Re}
	\Lambda_{b_N}
	\bigl(P_{\ll N}\Ncal(u),z_N,a_*z_N\bigr)\\
	&=
	-\operatorname{Re}
	\ip{\B_{b_N}(P_{\ll N}\Ncal(u),z_N)}{a_*z_N},\\[2mm]
	\mathcal R_{N,1}
	&=
	-\operatorname{Re}
	\Lambda_{b_N}(u_{\ll N},F_N,a_*z_N)\\
	&=
	-\operatorname{Re}
	\ip{\B_{b_N}(u_{\ll N},F_N)}{a_*z_N},\\[2mm]
	\mathcal R_{N,2}
	&=
	-\operatorname{Re}
	\Lambda_{b_N}(u_{\ll N},z_N,a_*F_N)\\
	&=
	-\operatorname{Re}
	\ip{\B_{b_N}(u_{\ll N},z_N)}{a_*F_N}.
\end{align*}
Thus $\mathcal R_{N,0}$ is generated when the nonlinearity falls on
the low-frequency factor, while $\mathcal R_{N,1}$ and
$\mathcal R_{N,2}$ arise from the two high-frequency factors.
There are no additional quartic terms, since the symbol $b_N$, the
frequency projections, and the spatial weight are independent of
time.

We next estimate the cross term $\mathcal X_N$.  Both high-frequency
factors must be placed in the local-smoothing norm; estimating either
one only in $L^2$ would lose a factor $N^{1/2}$.  To preserve the
frequency localization throughout the argument, we use the finite
sequence of enlarged cutoffs introduced above.

\begin{definition}
	\label{def:positive-ladder}
	For $\ell=0,1,2,3$, set
\[
	V_{\nu,N,\ell}
	=\Op\bigl(|v_\nu|^{1/2}\Theta_{\nu,N,\ell}\bigr).
\]
	For a $y$ chart, define
	\begin{equation}\label{eq:P-ladder}
		\mathcal P_{\nu,N,\ell}(T)
		=\sup_{y_0\in\R}\frac\mu R
		\int_0^T\!\!\int_{\R^2}
		a'\left(\frac{\mu(y-y_0)}R\right)
		|V_{\nu,N,\ell}z_N|^2\dd x\dd y\dd t,
	\end{equation}
	and use the analogous definition on an $x$ chart, with
	$\lambda\simeq N$ and the supremum taken over $x_0\in\R$.
	For $\ell=0,1,2$, the construction
	\eqref{eq:mN}-\eqref{eq:bN}, applied with
	$\Theta_{\nu,N,\ell}$ in place of the original chart cutoff, produces
	symbols $m_{N,\ell}$ and $b_{N,\ell}$, cubic corrections
	$\mathcal C_{N,\ell}$, normal commutator terms
	$\mathcal T_{N,\ell,2}$, cross terms $\mathcal X_{N,\ell}$, and
	quartic remainders $\mathcal R_{N,\ell}$.  The identity
	\eqref{eq:C-identity} holds with these level indices.  The resonance estimate,
	the high-variable symbol bounds, the boundary-correction estimate,
	and the quartic estimate below hold uniformly for
	$\ell=0,1,2$.  At the last level,
	\begin{equation}\label{eq:P3-crude}
		\mathcal P_{\nu,N,3}(T)
		\lesssim_R N^2T\norm{z_N}_{L_T^\infty L^2}^2.
	\end{equation}
	Indeed, on a $y$ chart,
	\[
	\frac\mu R\norm{V_{\nu,N,3}z_N}_2^2
	\lesssim_R
	N^{1/2}N^{3/2}\norm{z_N}_2^2
	=
	N^2\norm{z_N}_2^2,
\]
	while on an $x$ chart the corresponding bound is
	\[
	\frac\lambda R\norm{V_{\nu,N,3}z_N}_2^2
	\lesssim_R
	NN\norm{z_N}_2^2
	=
	N^2\norm{z_N}_2^2.
\]
	Here the $x$-chart bound uses the defining property of
	\cref{def:chart-ladder}: every auxiliary level $\ell\ge1$ is
	confined to the enlarged transition band, where $|v_x|\simeq N$.
	This is why the present $N^2$ bound is stronger than the crude
	level-zero estimate \eqref{eq:positive-finite}, for which
	$|v_x|$ may reach size $N^2$ on the full chart.
	Integration in time proves \eqref{eq:P3-crude}.
\end{definition}

The cubic corrections are perturbative at the energy level. By
\cref{def:positive-ladder}, the support and high-variable symbol
bounds established in \cref{lem:bN-class} hold uniformly for
$b_{N,\ell}$, with $\ell\in\{0,1,2\}$. Therefore,
\cref{lem:high-expansion}, applied with $m_N=b_{N,\ell}$, gives
\[
\norm{
\B_{b_{N,\ell}}(u_{\ll N}(t),z_N(t))
}_2
\lesssim_s
N^{-1}
\norm{u_{\ll N}(t)}_{H^s}
\norm{z_N(t)}_2.
\]
Since $P_{\ll N}$ is uniformly bounded on $H^s$ and the scaled weight
satisfies
\[
\norm{a_*}_{L^\infty}
\le
\norm{a}_{L^\infty},
\]
the definition of $\mathcal C_{N,\ell}$ and the Cauchy--Schwarz
inequality yield
\[
\begin{aligned}
|\mathcal C_{N,\ell}(t)|
&\le
\norm{
\B_{b_{N,\ell}}(u_{\ll N}(t),z_N(t))
}_2
\norm{a_*z_N(t)}_2
\\
&\lesssim_s
N^{-1}
\norm{u(t)}_{H^s}
\norm{z_N(t)}_2^2.
\end{aligned}
\]
Recalling that $z_N=P_NJ^\sigma u$, Littlewood--Paley orthogonality
gives
\[
\begin{aligned}
\sum_{N\ge2}
N^{-1}\norm{z_N(t)}_2^2
&=
\sum_{N\ge2}
N^{-1}
\norm{P_NJ^\sigma u(t)}_2^2
\\
&\lesssim
\norm{u(t)}_{H^{\sigma-1/2}}^2.
\end{aligned}
\]
Consequently, uniformly for $\ell\in\{0,1,2\}$ and uniformly with
respect to the centre of the spatial weight,
\begin{equation}\label{eq:C-size}
\begin{aligned}
\sup_{0\le t\le T}
\sum_{N\ge2}
|\mathcal C_{N,\ell}(t)|
&\lesssim_s
\norm{u}_{L_T^\infty H^s}
\norm{u}_{L_T^\infty H^{\sigma-1/2}}^2
\\
&\lesssim_s
E_s(T)^3.
\end{aligned}
\end{equation}
Here the last inequality follows from $\sigma-\frac12\le s$.
Thus, in the small-data regime $E_s(T)\ll1$, the cubic corrections
are perturbative relative to the quadratic localized energies.

\begin{proposition}\label{prop:crossterm}
	Let $N\ge2$, let $\nu$ be one of the microlocal charts, and let
	$\ell\in\{0,1,2\}$. Then, uniformly with respect to the centre of the
	spatial weight,
	\begin{equation}\label{eq:cross-final}
		\int_0^T|\mathcal X_{N,\ell}(t)|\,\dd t
		\le
		C_RN^{-1}E_s(T)\mathcal P_{\nu,N,\ell+1}(T)
		+
		C_RT N^{-1/2}E_s(T)
		\norm{z_N}_{L_T^\infty L^2}^2.
	\end{equation}
\end{proposition}

\begin{proof}
	We first consider a $y$ chart. The corresponding argument for an $x$
	chart is given in Step 5. Set
	\[
	a_*
	=
	a\left(\frac{\mu(y-y_0)}{R}\right),
	\qquad
	a_*'
	=
	a'\left(\frac{\mu(y-y_0)}{R}\right),
	\qquad
	a_*''
	=
	a''\left(\frac{\mu(y-y_0)}{R}\right),
	\]
	where
	\[
	\mu\simeq N^{1/2},
	\qquad
	\Theta_\ell=\Theta_{\nu,N,\ell},
	\]
	and define
	\[
	G_\ell
	=
	\B_{b_{N,\ell}}(u_{\ll N},z_N).
	\]
	Here the primes in $a_*'$ and $a_*''$ refer to derivatives of the
	unscaled profile $a$, evaluated at $\mu(y-y_0)/R$.
	
	\smallskip
	\noindent
	\emph{Step 1: decomposition of the weighted commutator.}
	By \eqref{eq:exact-y-comm},
	\[
	[i\omega(D),a_*]
	=
	\frac{2\mu}{R}a_*'D_xD_y
	-
	i\frac{\mu^2}{R^2}a_*''D_x
	=
	\mathcal T_1'+\mathcal T_2'.
	\]
	Accordingly, $\mathcal X_{N,\ell}$ is the sum of the pairings of
	$G_\ell$ with $\mathcal T_1'z_N$ and $\mathcal T_2'z_N$. The operator
	$\mathcal T_1'$ contains the principal local-smoothing contribution.
	On the parabolic block, the second term has the additional relative
	factor
	\[
	\frac{\mu}{R|\eta|}
	\simeq
	R^{-1},
	\]
	and will be treated separately in Step 4.
	
	Define
	\[
	V
	=
	V_{\nu,N,\ell+1}
	=
	\Op\left(
	|v_y|^{1/2}\Theta_{\ell+1}
	\right).
	\]
	By \cref{def:chart-ladder}, $\Theta_{\ell+1}$ is identically one on
	the high-frequency input and output supports of $b_{N,\ell}$.
	Moreover, $v_y$ has a fixed sign, denoted by $\tau$, on these
	supports. Since
	\[
	2D_xD_y=v_y(D),
	\]
	the principal part of $\mathcal T_1'z_N$, when paired with $G_\ell$,
	can be written as
	\[
	\mathcal T_1'z_N
	=
	\tau\frac{\mu}{R}a_*'V^2z_N
	+
	\mathcal E_{\ell,N}z_N.
	\]
	Here $\mathcal E_{\ell,N}$ denotes the support error generated by
	$1-\Theta_{\ell+1}^2(D)$.
	
	This decomposition does not require $a_*'z_N$ to be frequency
	localized. Indeed, after moving
	$1-\Theta_{\ell+1}^2(D)$ by adjoint onto the output of
	$\B_{b_{N,\ell}}$, its principal contribution vanishes because
	$\Theta_{\ell+1}=1$ on that output support. The remaining expression
	is a commutator with $a_*'$. It contains one derivative of the
	high-frequency cutoff and one derivative of the scaled spatial
	weight, and therefore belongs to the strip-localized class estimated
	below.
	
	Since $|v_y|\simeq N^{3/2}$ on the parabolic block, introduce the
	bilinear symbol
	\begin{equation}\label{eq:bpp-def}
		b_{\ell}^{V}(\theta,\zeta)
		=
		\frac{
			|v_y(\theta+\zeta)|^{1/2}
			\Theta_{\ell+1}(\theta+\zeta)
		}{
			|v_y(\zeta)|^{1/2}
		}
		b_{N,\ell}(\theta,\zeta).
	\end{equation}
	In \eqref{eq:bpp-def}, we have used
	$\Theta_{\ell+1}(\zeta)=1$ on $\supp b_{N,\ell}$. The symbol
	$b_{\ell}^{V}$ has size $O(N^{-1})$ and satisfies the same
	high-variable derivative estimates as $b_{N,\ell}$. Moreover,
	\begin{equation}\label{eq:V-through}
		V\B_{b_{N,\ell}}(u_{\ll N},z_N)
		=
		\B_{b_{\ell}^{V}}(u_{\ll N},Vz_N).
	\end{equation}
	
	We shall repeatedly use the following strip-localization estimate.
	Choose $\psi_N$ equal to one on the relevant high-frequency block and
	supported strictly inside the fixed periodization cell. If $\tau_k$
	denotes translation by
	\[
	\left(
	\frac{2\pi k_1}{N},
	\frac{2\pi k_2}{N^{1/2}}
	\right),
	\qquad
	k=(k_1,k_2)\in\mathbb Z^2,
	\]
	then
	\begin{equation}\label{eq:translated-strip}
		\begin{aligned}
			&\norm{
				w_*\tau_k\bigl(\psi_N(D)Vz_N\bigr)
			}_{L^2([0,T]\times\mathbb R^2)}
			\\
			&\qquad\le
			C_M\langle k\rangle^2
			\sup_{y_0'\in\mathbb R}
			\norm{
				\psi_0\left(\frac{\mu(\cdot-y_0')}{R}\right)Vz_N
			}_{L^2([0,T]\times\mathbb R^2)},
		\end{aligned}
	\end{equation}
	where
	\[
	w_*
	=
	\psi_0\left(\frac{\mu(y-y_0)}{R}\right).
	\]
	Indeed, the kernel of $\psi_N(D)$ decays on the spatial scales
	$(N^{-1},N^{-1/2})$. These scales are no larger than the width
	$R/\mu\simeq RN^{-1/2}$ of the $y$ strips. The covering by strips is
	independent of time, and the translation $\tau_k$ only changes the
	strip centre. The polynomial factor in $k$ absorbs the finitely
	overlapping translated strips and the rapidly decreasing kernel
	tails.
	
	\smallskip
	\noindent
	\emph{Step 2: symmetrization and explicit commutator contributions.}
	Since $a_*'=w_*^2$ and $V$ is self-adjoint, we have
	\begin{align}\label{eq:sym-split}
		\ip{G_\ell}{a_*'V^2z_N}
		={}&
		\ip{w_*VG_\ell}{w_*Vz_N}
		+
		\ip{[V,w_*]G_\ell}{w_*Vz_N}
		\notag\\
		&+
		\ip{G_\ell}{w_*[w_*,V]Vz_N}.
	\end{align}
	Set
	\[
	\upsilon_N(\zeta)
	=
	|v_y(\zeta)|^{1/2}
	\Theta_{\ell+1}(\zeta),
	\qquad
	V=\Op(\upsilon_N).
	\]
	For $m=1,2,3$, define the scaled profile derivatives
	\[
	w_{*,m}(y)
	=
	w^{(m)}
	\left(
	\frac{\mu(y-y_0)}{R}
	\right)
	\]
	and the corresponding band operators
	\[
	A_{m,N}
	=
	M_{w_{*,m}}
	\Op\left(\partial_\eta^m\upsilon_N\right).
	\]
	Because $w_*$ depends only on $y$, the Kohn--Nirenberg composition
	formula, expanded through three anisotropic derivatives, yields
	\begin{equation}\label{eq:three-commutators}
		[V,w_*]
		=
		\sum_{m=1}^{3}
		c_m
		\left(\frac{\mu}{R}\right)^m
		A_{m,N}
		+
		\mathcal E_{4,N},
	\end{equation}
	where $c_1,c_2,c_3$ are universal constants and
	$\mathcal E_{4,N}$ is the fourth-order symbolic remainder. Since
	$[w_*,V]=-[V,w_*]$, substitution of
	\eqref{eq:three-commutators} into \eqref{eq:sym-split} gives the two
	explicit commutator contributions
	\begin{equation}\label{eq:commutator-catalogue}
		\begin{aligned}
			&\tau c_m
			\frac{\mu}{R}
			\left(\frac{\mu}{R}\right)^m
			\ip{
				M_{w_{*,m}}
				\Op(\partial_\eta^m\upsilon_N)G_\ell
			}{
				w_*Vz_N
			},
			\\
			&-\tau c_m
			\frac{\mu}{R}
			\left(\frac{\mu}{R}\right)^m
			\ip{
				G_\ell
			}{
				w_*M_{w_{*,m}}
				\Op(\partial_\eta^m\upsilon_N)Vz_N
			},
			\qquad
			m=1,2,3.
		\end{aligned}
	\end{equation}
	
	On the parabolic block,
	\begin{equation}\label{eq:commutator-orders}
		\left|
		\partial_\eta^m\upsilon_N(\zeta)
		\right|
		\lesssim
		N^{3/4-m/2},
		\qquad
		m=0,1,2,3.
	\end{equation}
	Consequently,
	\[
	\left(\frac{\mu}{R}\right)^m
	\norm{
		\Op(\partial_\eta^m\upsilon_N)
	}_{L^2\to L^2}
	\lesssim
	R^{-m}
	N^{m/2}
	N^{3/4-m/2}
	=
	R^{-m}N^{3/4}.
	\]
	Thus the commutator expansion does not reduce the frequency order of
	$V$. Each contribution retains order $N^{3/4}$ but gains a factor
	$R^{-m}$. It must therefore be estimated in the same strip-localized
	norms as the principal term.
	
	To make this reduction explicit, define, on the support of
	$b_{N,\ell}$,
	\[
	b_{\ell,m}^{\mathrm{out}}(\theta,\zeta)
	=
	\frac{
		\partial_\eta^m\upsilon_N(\theta+\zeta)
	}{
		\upsilon_N(\zeta)
	}
	b_{N,\ell}(\theta,\zeta),
	\qquad
	m=1,2,3.
	\]
	The quotient is well defined on the relevant support because
	$\upsilon_N(\zeta)\simeq N^{3/4}$. It may be extended smoothly away
	from that support. By \eqref{eq:commutator-orders} and the
	high-variable estimates for $b_{N,\ell}$,
	\[
	\left|
	\partial_\xi^p\partial_\eta^q
	b_{\ell,m}^{\mathrm{out}}(\theta,\zeta)
	\right|
	\lesssim_{p,q}
	N^{-1-m/2-p-q/2}.
	\]
	Moreover,
	\[
	\Op(\partial_\eta^m\upsilon_N)
	G_\ell
	=
	\B_{b_{\ell,m}^{\mathrm{out}}}
	(u_{\ll N},Vz_N).
	\]
	Multiplication by $(\mu/R)^m$, together with
	$\mu\simeq N^{1/2}$, therefore gives
	\[
	\left(\frac{\mu}{R}\right)^m
	N^{-1-m/2}
	\lesssim
	R^{-m}N^{-1}.
	\]
	
	For the second pairing in \eqref{eq:commutator-catalogue}, we move
	$\Op(\partial_\eta^m\upsilon_N)$ by adjoint onto the output of
	$G_\ell$. The resulting principal bilinear symbol is again
	$b_{\ell,m}^{\mathrm{out}}$. Commuting this operator through the
	product $w_*w_{*,m}$ differentiates the spatial profiles and produces
	one additional factor $\mu/R$ together with one additional
	high-frequency derivative. Such terms therefore have an extra factor
	$R^{-1}$ and satisfy a stronger estimate than the principal
	contribution of order $m$.
	
	The derivatives $w_{*,m}$ are Schwartz profiles and can be bounded by
	rapidly convergent sums of translated copies of $w_*$. Applying the
	high-variable Fourier expansion and
	\eqref{eq:translated-strip} to the two high-frequency factors gives
	two copies of
	\[
	\left(\frac{R}{\mu}\right)^{1/2}
	\mathcal P_{\nu,N,\ell+1}(T)^{1/2}.
	\]
	It follows that, for each $m=1,2,3$,
	\begin{equation}\label{eq:commutator-nonprincipal}
		\begin{aligned}
			&\int_0^T
			\Bigg[
			\left|
			\tau c_m
			\frac{\mu}{R}
			\left(\frac{\mu}{R}\right)^m
			\ip{
				M_{w_{*,m}}
				\Op(\partial_\eta^m\upsilon_N)G_\ell
			}{
				w_*Vz_N
			}
			\right|
			\\
			&\qquad\qquad+
			\left|
			\tau c_m
			\frac{\mu}{R}
			\left(\frac{\mu}{R}\right)^m
			\ip{
				G_\ell
			}{
				w_*M_{w_{*,m}}
				\Op(\partial_\eta^m\upsilon_N)Vz_N
			}
			\right|
			\Bigg]\,\dd t
			\\
			&\qquad\le
			CR^{-m}N^{-1}E_s(T)
			\mathcal P_{\nu,N,\ell+1}(T).
		\end{aligned}
	\end{equation}
	
	The fourth-order remainder in \eqref{eq:three-commutators} produces
	the two explicit pairings
	\[
	\tau\frac{\mu}{R}
	\ip{\mathcal E_{4,N}G_\ell}{w_*Vz_N}
	\]
	and
	\[
	-\tau\frac{\mu}{R}
	\ip{G_\ell}{w_*\mathcal E_{4,N}Vz_N}.
	\]
	The operator $\mathcal E_{4,N}$ has the same band localization and
	the same reciprocal kernel scales $(N^{-1},N^{-1/2})$ as the preceding
	terms. The strip-routing argument applied to these two expressions
	gives
	\begin{equation}\label{eq:commutator-terminal}
		\begin{aligned}
			\frac{\mu}{R}
			\int_0^T
			\Big(
			&
			\left|
			\ip{\mathcal E_{4,N}G_\ell}{w_*Vz_N}
			\right|
			\\
			&+
			\left|
			\ip{G_\ell}{w_*\mathcal E_{4,N}Vz_N}
			\right|
			\Big)\,\dd t
			\le
			CR^{-4}N^{-1}E_s(T)
			\mathcal P_{\nu,N,\ell+1}(T).
		\end{aligned}
	\end{equation}
	For $R$ sufficiently large, this contribution is subordinate to the
	principal estimate.
	
	The support error $\mathcal E_{\ell,N}$ from Step 1 produces the
	pairing
	\[
	\ip{G_\ell}{\mathcal E_{\ell,N}z_N}.
	\]
	As explained above, this error contains one derivative of
	$1-\Theta_{\ell+1}^2$ and one derivative of the scaled weight $a_*'$.
	The same argument used for $m=1$ in
	\eqref{eq:commutator-nonprincipal} gives
	\[
	\int_0^T
	\left|
	\ip{G_\ell}{\mathcal E_{\ell,N}z_N}
	\right|\,\dd t
	\le
	CR^{-1}N^{-1}E_s(T)
	\mathcal P_{\nu,N,\ell+1}(T).
	\]
	Hence the complete commutator expansion, including the fourth-order
	remainder and the support error, contributes only to the first term
	on the right-hand side of \eqref{eq:cross-final}. In particular, it
	does not produce an additional $N^{-1/2}$ remainder.
	
	\smallskip
	\noindent
	\emph{Step 3: estimate of the principal term.}
	By \cref{lem:high-expansion},
	\[
	w_*VG_\ell
	=
	N^{-1}
	\sum_{k\in\mathbb Z^2}
	w_*\bigl(c_k(D)u_{\ll N}\bigr)
	\tau_k\bigl(\psi_N(D)Vz_N\bigr),
	\]
	where, for every $M>0$,
	\[
	\sup_\theta|c_k(\theta)|
	\le
	C_M\langle k\rangle^{-M}.
	\]
	Plancherel's theorem and the Sobolev embedding
	$H^s(\mathbb R^2)\hookrightarrow L^\infty(\mathbb R^2)$ imply
	\[
	\norm{
		c_k(D)u_{\ll N}
	}_{L^\infty}
	\le
	C_M\langle k\rangle^{-M}E_s(T).
	\]
	Using \eqref{eq:translated-strip}, choosing $M$ sufficiently large,
	and summing over $k\in\mathbb Z^2$, we obtain
	\[
	\int_0^T
	\left|
	\ip{w_*VG_\ell}{w_*Vz_N}
	\right|\,\dd t
	\le
	CN^{-1}E_s(T)
	\frac{R}{\mu}
	\mathcal P_{\nu,N,\ell+1}(T).
	\]
	Restoring the prefactor $\mu/R$ from $\mathcal T_1'$ yields
	\[
	\frac{\mu}{R}
	\int_0^T
	\left|
	\ip{w_*VG_\ell}{w_*Vz_N}
	\right|\,\dd t
	\le
	CN^{-1}E_s(T)
	\mathcal P_{\nu,N,\ell+1}(T).
	\]
	By Step 2, all commutator and support-error contributions satisfy the
	same bound, with at least one additional power of $R^{-1}$.
	
	\smallskip
	\noindent
	\emph{Step 4: estimate of the subprincipal term.}
	Retain the notation
	\[
	\upsilon_N(\zeta)
	=
	|v_y(\zeta)|^{1/2}
	\Theta_{\ell+1}(\zeta),
	\qquad
	V=\Op(\upsilon_N).
	\]
	On the relevant support,
	\[
	\upsilon_N(\zeta)\simeq N^{3/4}.
	\]
	After inserting $\Theta_{\ell+1}$ on the high input and output
	frequencies, write
	\[
	D_x\Theta_{\ell+1}(D)=WV,
	\]
	where
	\[
	W
	=
	\Op\left(
	\frac{
		\xi\Theta_{\ell+1}
	}{
		|v_y|^{1/2}
	}
	\right).
	\]
	A chart sign may be inserted in this identity, but it is irrelevant
	for the estimates. Since $|\xi|\simeq N$ and
	$|v_y|\simeq N^{3/2}$, the symbol of $W$ has size $N^{1/4}$.
	
	To place both high-frequency factors in the local-smoothing norm,
	write the high-frequency input of $G_\ell$ as
	\[
	z_N
	=
	\upsilon_N(D)^{-1}Vz_N
	\]
	on the support of $b_{N,\ell}$, and move $W$ from the weighted factor
	to the output of the bilinear operator. The resulting principal
	bilinear symbol is
	\[
	d_\ell(\theta,\zeta)
	=
	\frac{
		(\xi+\alpha)
		\Theta_{\ell+1}(\theta+\zeta)
	}{
		|v_y(\theta+\zeta)|^{1/2}
		|v_y(\zeta)|^{1/2}
	}
	b_{N,\ell}(\theta,\zeta).
	\]
	The numerator is $O(N)$, each denominator is comparable to
	$N^{3/4}$, and $b_{N,\ell}=O(N^{-1})$. Therefore,
	\[
	|d_\ell(\theta,\zeta)|
	\lesssim
	N\cdot N^{-3/4}\cdot N^{-3/4}\cdot N^{-1}
	=
	N^{-3/2}.
	\]
	The high-variable symbol estimates also give
	\[
	\left|
	\partial_\xi^p\partial_\eta^q
	d_\ell(\theta,\zeta)
	\right|
	\lesssim_{p,q}
	N^{-3/2-p-q/2}.
	\]
	The high-variable Fourier expansion, followed by
	\eqref{eq:translated-strip}, yields
	\[
	\int_0^T
	\left|
	\ip{
		\B_{d_\ell}(u_{\ll N},Vz_N)
	}{
		a_*''Vz_N
	}
	\right|\,\dd t
	\lesssim_R
	N^{-3/2}E_s(T)
	\frac{R}{\mu}
	\mathcal P_{\nu,N,\ell+1}(T).
	\]
	Restoring the prefactor $\mu^2/R^2$ from $\mathcal T_2'$ and using
	$\mu\simeq N^{1/2}$, we find
	\[
	\frac{\mu^2}{R^2}
	N^{-3/2}
	\frac{R}{\mu}
	=
	\frac{\mu}{R}N^{-3/2}
	\lesssim
	R^{-1}N^{-1}.
	\]
	Thus the principal part of the subprincipal contribution is bounded
	by
	\[
	CR^{-1}N^{-1}E_s(T)
	\mathcal P_{\nu,N,\ell+1}(T).
	\]
	
	It remains to estimate the commutators generated when
	$\upsilon_N(D)^{-1}$ and $W$ are passed through $a_*''$, together
	with the corresponding support errors. Every term for which both
	high-frequency factors remain strip localized is treated as in
	Step 2 and gains at least one additional power of $R^{-1}$.
	
	After removing these strip-localized contributions, the terminal
	symbolic remainder contains at least two derivatives in the
	high-frequency variables. Before these derivatives are taken, its
	size is bounded by
	\[
	\underbrace{N^{-1}}_{b_{N,\ell}}
	\,
	\underbrace{N}_{D_x}
	\,
	\underbrace{N^{-3/2}}_{\text{two inverse smoothing factors}}
	\,
	\underbrace{N}_{\mu^2/R^2}
	=
	O_R(N^{-1/2}).
	\]
	Each derivative in a high-frequency variable gains at least
	$N^{-1/2}$. Indeed, $\partial_\xi$ acts at the scale
	$|\xi|\simeq N$ and gains $N^{-1}$, while $\partial_\eta$ acts at the
	scale $|\eta|\simeq N^{1/2}$ and gains $N^{-1/2}$. Since the terminal
	symbol $r_{\ell,N}$ contains at least two such derivatives, it
	satisfies the sharper bound
	\begin{equation}\label{eq:terminal-symbol}
		\left|
		\partial_\xi^p\partial_\eta^q
		r_{\ell,N}(\theta,\zeta)
		\right|
		\lesssim_{p,q,R}
		N^{-3/2-p-q/2}.
	\end{equation}
	For the conclusion of the proposition, it is enough to retain the
	weaker consequence
	\[
	\left|
	\partial_\xi^p\partial_\eta^q
	r_{\ell,N}(\theta,\zeta)
	\right|
	\lesssim_{p,q,R}
	N^{-1/2-p-q/2}.
	\]
	Using the sharper estimate \eqref{eq:terminal-symbol} would improve
	the second term in \eqref{eq:cross-final}, but would not change the
	final regularity threshold, which is determined by the coupled
	a priori estimates.
	
	Applying the high-variable Fourier expansion without smoothing
	weights gives
	\[
	\int_0^T
	\left|
	\ip{
		\B_{r_{\ell,N}}(u_{\ll N},z_N)
	}{
		z_N
	}
	\right|\,\dd t
	\le
	C_RT N^{-1/2}E_s(T)
	\norm{z_N}_{L_T^\infty L^2}^2.
	\]
	Combining the strip-localized and terminal contributions, we obtain
	\[
	\begin{aligned}
		\int_0^T
		\left|
		\ip{G_\ell}{\mathcal T_2'z_N}
		\right|\,\dd t
		\le{}&
		CR^{-1}N^{-1}E_s(T)
		\mathcal P_{\nu,N,\ell+1}(T)
		\\
		&+
		C_RT N^{-1/2}E_s(T)
		\norm{z_N}_{L_T^\infty L^2}^2.
	\end{aligned}
	\]
	This terminal symbolic remainder is the sole source of the second
	term on the right-hand side of \eqref{eq:cross-final}.
	
	\smallskip
	\noindent
	\emph{Step 5: the $x$ charts.}
	On the transition support of an $x$ chart,
	$|v_x|\simeq N$, and $v_x$ has a fixed sign. Define
	\[
	V_x
	=
	\Op\left(
	|v_x|^{1/2}
	\Theta_{x\tau,N,\ell+1}
	\right).
	\]
	The exact principal contribution is
	\[
	\frac{\lambda}{R}a_*'V_x^2,
	\qquad
	\lambda\simeq N,
	\]
	and the analogue of the symbol $b_{\ell}^{V}$ again has size
	$O(N^{-1})$. The arguments of Steps 1-3 therefore give
	\[
	\int_0^T|\mathcal X_{N,\ell}(t)|\,\dd t
	\le
	C_RN^{-1}E_s(T)
	\mathcal P_{\nu,N,\ell+1}(T)
	+
	C_RT N^{-1/2}E_s(T)
	\norm{z_N}_{L_T^\infty L^2}^2.
	\]
	The order-zero polynomial remainder and the Hilbert-transform
	remainder are controlled by
	\cref{prop:positive,lem:hilbert-weight}. Their translated-strip tails
	are estimated by \eqref{eq:translated-strip}. This proves
	\eqref{eq:cross-final}.
\end{proof}

\subsection{The quartic remainder}

\begin{proposition}\label{prop:quartic}
	For $s>1$ and $1<\sigma\le s$,
	\begin{equation}\label{eq:quartic}
		\sum_N\int_0^T|\mathcal R_N(t)|\dd t
		\lesssim_s
		T\norm{u}_{L_T^\infty H^s}^4.
	\end{equation}
	The estimate is uniform with respect to the centre of the spatial
	weight and, for the enlarged-cutoff construction, uniform in the
	level $\ell$.
\end{proposition}

\begin{proof}
	We estimate separately the three terms in the exact decomposition
	\[
	\mathcal R_N
	=
	\mathcal R_{N,0}
	+\mathcal R_{N,1}
	+\mathcal R_{N,2}
\]
	given above.  Since $a$ is bounded,
\[
	\norm{a_*f}_2
	\le
	\norm a_{L^\infty}\norm f_2,
	\]
	uniformly in the centre and in the frequency scale of the weight.
	
	\smallskip
	\emph{The low-frequency contribution.}
	Let $\chi_{\ll N}$ be the symbol of $P_{\ll N}$.  Since
\[
	\widehat{P_{\ll N}\Ncal(u)}(\theta)
	=
	\frac{i\alpha}{2}\chi_{\ll N}(\theta)
	\widehat{u^2}(\theta),
	\qquad
	\theta=(\alpha,\beta),
\]
	we have
\[
	\B_{b_N}(P_{\ll N}\Ncal(u),z_N)
	=
	\frac12\B_{d_N}(u^2,z_N),
	\]
	where
	\[
	d_N(\theta,\zeta)
	=
	i\alpha\chi_{\ll N}(\theta)b_N(\theta,\zeta).
\]
	The symbol $d_N$ is of the form covered by
	\eqref{eq:B-low-derivative}.  Therefore
	\[
	\norm{\B_{b_N}(P_{\ll N}\Ncal(u),z_N)}_2
	\lesssim_s
	\norm{u^2}_{H^s}\norm{z_N}_2,
	\]
	and hence
\[
	|\mathcal R_{N,0}|
	\lesssim_s
	\norm{u^2}_{H^s}\norm{z_N}_2^2.
	\]
	Summing in $N$ and using Littlewood--Paley orthogonality,
\[
	\sum_N|\mathcal R_{N,0}|
	\lesssim_s
	\norm{u^2}_{H^s}
	\sum_N\norm{P_NJ^\sigma u}_2^2
	\lesssim_s
	\norm{u^2}_{H^s}\norm u_{H^\sigma}^2.
\]
	Since $s>1$, $H^s(\R^2)$ is an algebra, and since $\sigma\le s$,
	\[
	\sum_N|\mathcal R_{N,0}(t)|
	\lesssim_s
	\norm{u(t)}_{H^s}^4.
\]
	
	\smallskip
	\emph{The two high-frequency contributions.}
	By \eqref{eq:B-basic},
	\[
	\norm{\B_{b_N}(u_{\ll N},F_N)}_2
	\lesssim_s
	N^{-1}\norm{u_{\ll N}}_{H^s}\norm{F_N}_2,
\]
	and
	\[
	\norm{\B_{b_N}(u_{\ll N},z_N)}_2
	\lesssim_s
	N^{-1}\norm{u_{\ll N}}_{H^s}\norm{z_N}_2.
	\]
	Since $P_{\ll N}$ is uniformly bounded on $H^s$, it follows that
\[
	|\mathcal R_{N,1}|+|\mathcal R_{N,2}|
	\lesssim_s
	N^{-1}\norm u_{H^s}\norm{F_N}_2\norm{z_N}_2.
\]
	Cauchy--Schwarz in $N$ gives
	\begin{align*}
		\sum_N\bigl(
		|\mathcal R_{N,1}|+|\mathcal R_{N,2}|
		\bigr)
		\lesssim_s{}&
		\norm u_{H^s}
		\left(\sum_NN^{-2}\norm{F_N}_2^2\right)^{1/2}\\
		&\times
		\left(\sum_N\norm{z_N}_2^2\right)^{1/2}.
	\end{align*}
	The second square function is bounded by
	\[
	\left(\sum_N\norm{z_N}_2^2\right)^{1/2}
	\lesssim
	\norm u_{H^\sigma}.
\]
	For the first one, use
	\[
	F_N
	=
	\frac12P_NJ^\sigma\partial_x(u^2).
	\]
	If $\varphi_N$ denotes the symbol of $P_N$, define
	$\widetilde P_N$ by
	\[
	\widehat{\widetilde P_Nf}(\xi,\eta)
	=
	\frac{i\xi}{N}\varphi_N(\xi,\eta)\widehat f(\xi,\eta).
\]
	The symbols of $\widetilde P_N$ are uniformly of order zero, and
\[
	N^{-1}F_N
	=
	\frac12\widetilde P_NJ^\sigma(u^2).
\]
	Littlewood--Paley theory therefore gives
\[
	\left(\sum_NN^{-2}\norm{F_N}_2^2\right)^{1/2}
	\lesssim
	\norm{J^\sigma(u^2)}_2
	=
	\norm{u^2}_{H^\sigma}.
\]
	Consequently,
\[
	\sum_N\bigl(
	|\mathcal R_{N,1}|+|\mathcal R_{N,2}|
	\bigr)
	\lesssim_s
	\norm u_{H^s}\norm{u^2}_{H^\sigma}\norm u_{H^\sigma}.
\]
	Since $\sigma>1$, $H^\sigma(\R^2)$ is an algebra, and
	$\sigma\le s$.  Hence
\[
	\sum_N\bigl(
	|\mathcal R_{N,1}(t)|+|\mathcal R_{N,2}(t)|
	\bigr)
	\lesssim_s
	\norm{u(t)}_{H^s}^4.
\]
	
	Combining the low- and high-frequency estimates, we obtain, for every
	$t\in[0,T]$,
\[
	\sum_N|\mathcal R_N(t)|
	\lesssim_s
	\norm{u(t)}_{H^s}^4.
\]
	Integration in time yields
\[
	\sum_N\int_0^T|\mathcal R_N(t)|\dd t
	\lesssim_s
	\int_0^T\norm{u(t)}_{H^s}^4\dd t
	\le
	T\norm{u}_{L_T^\infty H^s}^4,
\]
	which proves \eqref{eq:quartic}.  The proof uses only the uniform
	high-variable symbol bounds and is therefore unchanged when $b_N$ is
	replaced by any $b_{N,\ell}$, $\ell=0,1,2$.
\end{proof}
\subsection{The nonlinear smoothing estimate}

\begin{proposition}\label{prop:smoothing}
Let $s>1$, $1<\sigma<s$, and let $u$ be a smooth solution of
\eqref{IVP} on $[0,T]$, $T\le1$.  Set
\[
 E_s(T)=\norm{u}_{L_T^\infty H^s},
 \qquad
 A(T)=\norm{\nabla u}_{L_T^2L^\infty}.
\]
Then 
\begin{equation}\label{eq:smoothing-final}
 S_\sigma(u;T)^2
 \lesssim_{s,\sigma}
 E_s(T)^2\bigl(1+T+T^{1/2}A(T)\bigr)
 +E_s(T)^3+TE_s(T)^4+TE_s(T)^5,
\end{equation}
where $S_\sigma(u;T)^2$ is defined in \eqref{eq:S-def} and the constants depend on $s-\sigma$ and are stable under smooth frequency truncation.
\end{proposition}

\begin{proof}
Notice that all identities are first applied to
smooth frequency-truncated solutions; the truncation is removed at the
end.

\smallskip
\emph{(I) The localized weighted identity.}
Fix a chart $\nu$, a level $\ell\in\{0,1,2\}$, and the centre of the
corresponding spatial weight.  Put
\[
 z_{\nu,N,\ell}=\Theta_{\nu,N,\ell}(D)z_N,
 \qquad
 A_*=\tau a_*,
 \qquad
 \tau=\sgn v_\nu.
\]
Since
\[
 (\partial_t-i\omega(D))z_{\nu,N,\ell}
 =-\Theta_{\nu,N,\ell}(D)F_N,
\]
\cref{lem:weighted-energy} gives the concrete identity
\begin{align}\label{eq:smoothing-weighted-identity}
 &\frac{\dd}{\dd t}
 \ip{A_*z_{\nu,N,\ell}}{z_{\nu,N,\ell}}
 +\ip{i[\omega(D),A_*]z_{\nu,N,\ell}}
 {z_{\nu,N,\ell}}\notag\\
 &\qquad
 =-2\operatorname{Re}
 \ip{A_*\Theta_{\nu,N,\ell}(D)F_N}
 {z_{\nu,N,\ell}}.
\end{align}
Integrating in time, applying \cref{prop:positive}, taking the
supremum over the spatial centre, and absorbing
\eqref{eq:remainder-absorb} yield
\begin{align}\label{eq:positive-from-weighted}
 \mathcal P_{\nu,N,\ell}(T)
 \lesssim_R{}&(1+T)\norm{z_N}_{L_T^\infty L^2}^2\notag\\
 &+\int_0^T
 \left|
 \operatorname{Re}
 \ip{A_*\Theta_{\nu,N,\ell}(D)F_N}
 {z_{\nu,N,\ell}}
 \right|\dd t.
\end{align}
The term $T\norm{z_N}_{L_T^\infty L^2}^2$ contains the order-zero
calculus remainders.  The translated tails of the Hilbert-transform
remainder are absorbed only after the supremum is taken.  This is
legitimate because \eqref{eq:positive-finite} makes that supremum
finite before absorption.

\smallskip
\emph{(II) Paraproduct decomposition and the transport term.}
Apply Bony's decomposition to the nonlinear pairing on the right of
\eqref{eq:smoothing-weighted-identity}.  The low--high term is
$u_{\ll N}\partial_xz_N$; after inserting the chart multiplier it is
the sum of
\[
 u_{\ll N}\partial_xz_{\nu,N,\ell}
 \quad\text{and}\quad
 [\Theta_{\nu,N,\ell}(D),u_{\ll N}]\partial_xz_N.
\]
Integrating the first expression by parts produces
\[
 \int_0^T\norm{u_x(t)}_\infty\norm{z_N(t)}_2^2\dd t
\]
and, only for an $x$ weight, the weighted byproduct
\[
 \int_0^T\!\int_{\R^2}
 |u_{\ll N}|\,|\partial_xa_*|\,
 |z_{\nu,N,\ell}|^2\dd x\dd y\dd t.
\]
Passing the chart multiplier through $a_*$ creates kernels at spatial
scale $N^{-1}$, whereas the strip width is $R/N$; decomposing those
kernels into translated strips gives rapidly decreasing tails.  On an
$x$ chart $|v_x|\ge c_0N$, so the ratio of this byproduct to the
positive form is at most
\begin{equation}\label{eq:transport-absorption-ratio}
 \frac{C_R\norm{u(t)}_\infty}{c_0N}.
\end{equation}
Let
\begin{equation}\label{eq:N0-smoothing}
 N_0=2C_RE_s(T).
\end{equation}
After enlarging $C_R$ to include $c_0^{-1}$, the term
\eqref{eq:transport-absorption-ratio} is absorbed for $N\ge N_0$.
For the remaining dyadic frequencies,
\begin{align}\label{eq:transport-low-frequencies}
 &\sum_{2\le N<N_0}N
 \int_0^T\norm{u(t)}_\infty\norm{z_N(t)}_2^2\dd t\notag\\
 &\qquad\lesssim_s
 TN_0E_s(T)\sum_N\norm{z_N}_{L_T^\infty L^2}^2
 \lesssim_s TE_s(T)^4.
\end{align}
Here, and at every later occurrence of a square sum of the quantities
$\norm{z_N}_{L_T^\infty L^2}$, we use that
\begin{equation}\label{eq:sup-in-time-square-sum}
 \sum_{N\ge2}\norm{z_N}_{L_T^\infty L^2}^2
 \le\Bigl(\sum_{N\ge2}N^{-2(s-\sigma)}\Bigr)
 \sup_N\sup_{0\le t\le T}N^{2s}\norm{P_Nu(t)}_2^2
 \le C_{s-\sigma}E_s(T)^2 .
\end{equation}
This is not Littlewood--Paley orthogonality, because the maximizing
time depends on $N$.  It is the same device used in the summation of
\cref{sec:refined}, and it is one more point at which the strict
inequality $\sigma<s$ is required.

\smallskip
\emph{(III) Balanced terms and easy chart commutators.}
By \cref{lem:paraproduct}, the balanced, high--high, and fractional
commutator terms are bounded after square summation by
\begin{equation}\label{eq:smoothing-balanced-sum}
 \sum_N\int_0^T
 \norm{\nabla u(t)}_\infty\norm{z_N(t)}_2^2\dd t
 \lesssim T^{1/2}A(T)E_s(T)^2.
\end{equation}
The component \eqref{eq:mNx} of the chart commutator and the part of
\eqref{eq:mNy} away from \eqref{eq:band} have smooth low-frequency
symbols and satisfy the same bound by the high-variable expansion.

\smallskip
\emph{(IV) The tangential transition sector.}
At each localization level, split the band commutator by
\eqref{eq:sector-partition}.  The tangential symbol
$m_{N,\Gamma}\psi_{\mathrm{tan}}$ is controlled by
\cref{lem:tangential}.  Hence
\begin{align}\label{eq:log-absorb}
 &\sum_N(1+\log N)
 \int_0^T\norm{u_x(t)}_\infty
 \norm{P_NJ^\sigma u(t)}_2^2\dd t\notag\\
 &\qquad\le
 C_{s-\sigma}T^{1/2}A(T)E_s(T)^2,
\end{align}
because
$(1+\log N)N^{-2(s-\sigma)}\le C_{s-\sigma}$.
The low-shell term containing $\norm u_\infty$ instead contributes
$CTE_s(T)^3$ after summation.

\smallskip
\emph{(V) The normal transition sector below $N_0$.}
For $N<N_0$, use the crude size
$|m_{N,\Gamma}\psi_{\mathrm{nor}}|\lesssim N$ and
\cref{lem:high-expansion}.  This gives
\begin{equation}\label{eq:lowN}
 \sum_{2\le N<N_0}\int_0^T
 \left|\ip{\B_{m_{N,\Gamma}\psi_{\mathrm{nor}}}
 (u_{\ll N},z_N)}{a_*z_N}\right|\dd t
 \lesssim_s TE_s(T)^4.
\end{equation}

\smallskip
\emph{(VI) The normal transition sector for $N\ge N_0$.}
For $N\ge N_0$, the normal component of the symmetrized chart
commutator \eqref{eq:chart-commutator} equals $2\tau$ times
$\mathcal T_{N,\ell,2}$, with $m_{N,\ell}$ and $b_{N,\ell}$ obtained
from $\Theta_{\nu,N,\ell}$ as in \cref{def:positive-ladder}; the
factor $2$ comes from the right-hand side of
\eqref{eq:smoothing-weighted-identity} and the sign $\tau$ from
$A_*=\tau a_*$.  Both are harmless and are suppressed below.  Integrating the exact normal-form identity
\eqref{eq:C-identity} gives
\begin{equation}\label{eq:normal-form-inserted}
 \int_0^T\mathcal T_{N,\ell,2}(t)\dd t
 =\mathcal C_{N,\ell}(0)-\mathcal C_{N,\ell}(T)
 +\int_0^T\bigl(\mathcal X_{N,\ell}(t)
 +\mathcal R_{N,\ell}(t)\bigr)\dd t.
\end{equation}
Thus the normal form cancels exactly the term not controlled by
\eqref{eq:smoothing-balanced-sum}; it does not modify any of the easy
paraproduct terms.

For later reference, define the complete nonrecursive contribution
at level $\ell$ by
\begin{align}\label{eq:defined-base-remainder}
 \mathfrak B_{\nu,N,\ell}(T)
 :={}&C_R(1+T)\norm{z_N}_{L_T^\infty L^2}^2\notag\\
 &+C_R\int_0^T
 \Bigl((1+\log N)\norm{u_x(t)}_\infty
       +\norm{\nabla u(t)}_\infty
       +\norm{u(t)}_\infty\Bigr)
 \norm{z_N(t)}_2^2\dd t\notag\\
 &+C_R\bigl(|\mathcal C_{N,\ell}(0)|
       +|\mathcal C_{N,\ell}(T)|\bigr)
 +C_R\int_0^T|\mathcal R_{N,\ell}(t)|\dd t.
\end{align}
Equations \eqref{eq:smoothing-balanced-sum}, \eqref{eq:log-absorb},
and \eqref{eq:C-size}, together with \cref{prop:quartic}, imply,
uniformly in
$\ell\in\{0,1,2\}$,
\begin{align}\label{eq:base-remainder-sum}
 \sum_{N\ge2}\mathfrak B_{\nu,N,\ell}(T)
 \lesssim_{s,\sigma}{}&
 E_s(T)^2\bigl(1+T+T^{1/2}A(T)\bigr)\notag\\
 &+E_s(T)^3+TE_s(T)^4.
\end{align}

Using \cref{prop:crossterm} in
\eqref{eq:normal-form-inserted}, and then inserting the result into
\eqref{eq:positive-from-weighted}, yields the fully specified
recursion
\begin{align}\label{eq:ladder-recursion}
 \mathcal P_{\nu,N,\ell}(T)
 \le{}&\mathfrak B_{\nu,N,\ell}(T)
 +q_N\mathcal P_{\nu,N,\ell+1}(T)\notag\\
 &+C_RT N^{-1/2}E_s(T)
 \norm{z_N}_{L_T^\infty L^2}^2,
 \qquad
 q_N=C_RN^{-1}E_s(T)\le\frac12.
\end{align}
The resonance, symbol, boundary-correction, and quartic estimates are
uniform at all levels because
\eqref{eq:c0-resonance} keeps the four enlarged bands inside the
range of \cref{prop:res-lower}.

\smallskip
\emph{(VII) Closing the localization ladder.}
Back-substitute \eqref{eq:ladder-recursion} for
$\ell=0,1,2$ and use \eqref{eq:P3-crude}.  Since $q_N\le1/2$,
\begin{align}\label{eq:ladder-close}
 \mathcal P_{\nu,N,0}(T)
 \lesssim{}&\sum_{\ell=0}^2q_N^\ell
 \mathfrak B_{\nu,N,\ell}(T)
 +C_RT N^{-1/2}E_s(T)
 \norm{z_N}_{L_T^\infty L^2}^2\notag\\
 &+q_N^3N^2T\norm{z_N}_{L_T^\infty L^2}^2.
\end{align}
The middle term sums to $CTE_s(T)^3$.  Since
\[
 q_N^3N^2T\norm{z_N}_{L_T^\infty L^2}^2
 \lesssim_R
 TN^{-1}E_s(T)^3\norm{z_N}_{L_T^\infty L^2}^2,
\]
the last term sums to $CTE_s(T)^5$.  Combining
\eqref{eq:transport-low-frequencies}, \eqref{eq:lowN},
\eqref{eq:base-remainder-sum}, and \eqref{eq:ladder-close}, then
summing over the finitely many charts, gives the desired bound for
the level-zero positive quantities.  Finally,
\cref{lem:local-mean,prop:positive} converts those quantities into
$S_\sigma(u;T)^2$.  This proves \eqref{eq:smoothing-final}.
\end{proof}

\section{A priori estimates}\label{sec:apriori}

Fix $s>19/16$.  We choose $\eps_0>0$ such that
\begin{equation}\label{eq:eps-choice}
 \frac{19}{16}+3\eps_0<s
\end{equation}
and set
\begin{equation}\label{eq:sigma-choice}
 \sigma=s-\eps_0.
\end{equation}
For $0<T\le1$ we use the notation
\begin{equation}\label{eq:norms}
 E_s(T)=\norm{u}_{L_T^\infty H^s},
 \qquad
 A(T)=\norm{\nabla u}_{L_T^2L^\infty},
 \qquad
 M(T)=M_T(u),
 \qquad
 S(T)=S_\sigma(u;T).
\end{equation}

The following product estimate is the link between the maximal
function bounds and the microlocal smoothing estimates.

\begin{proposition}\label{prop:product}
Let $1/2<\gamma<1$, let $1<\sigma\le s$, and assume that
\begin{equation}\label{eq:gamma-condition}
 \gamma+\frac12<\sigma.
\end{equation}
Then
\begin{equation}\label{eq:product}
 \norm{J^\gamma(uu_x)}_{L_T^2L^2}
 \lesssim_{s,\sigma}
 M_T(u)\,S_\sigma(u;T)+E_s(T)A(T).
\end{equation}
\end{proposition}

\begin{proof}
We use the paraproduct decomposition of \cref{lem:paraproduct}.  The
balanced interactions and the commutator terms satisfy
\begin{equation}\label{eq:product-balanced}
 \norm{\mathcal E_\gamma(u)}_{L_T^2L^2}
 \lesssim E_s(T)A(T),
\end{equation}
where $\mathcal E_\gamma(u)$ denotes their sum.  It remains to
estimate the low--high part
\begin{equation}\label{eq:product-low-high}
 \Pi_\gamma(u)
 =\sum_{N\ge2}u_{<N/8}\,P_NJ^\gamma\partial_xu.
\end{equation}

Let $\{\chi_{\nu,N}\}_\nu$ be the microlocal partition introduced in
\cref{sec:smoothing}.  On the support of $\chi_{\nu,N}$ one has
\begin{equation}\label{eq:velocity-lower-product}
 |v_\nu(\zeta)|\gtrsim N,
 \qquad |\zeta|\simeq N.
\end{equation}
Consequently, by \eqref{eq:gamma-condition},
\begin{equation}\label{eq:multiplier-compare}
 \langle\zeta\rangle^\gamma|\xi|
 \lesssim
 \langle\zeta\rangle^\sigma|v_\nu(\zeta)|^{1/2}.
\end{equation}
For each $\nu$ define
\[
 r_{\nu,N}(\zeta)
 =\frac{\langle\zeta\rangle^\gamma i\xi\,
 \chi_{\nu,N}(\zeta)}
 {\langle\zeta\rangle^\sigma|v_\nu(\zeta)|^{1/2}}.
\]
The inverse Fourier transform of $r_{\nu,N}$ is uniformly bounded in
$L^1(\R^2)$.  We observe that the $x$ charts are
not confined to a single parabolic block.  On a $y$ chart, the symbol
derivatives satisfy
\[
 |\partial_\xi^a\partial_\eta^b r_{\nu,N}(\zeta)|
 \lesssim_{a,b}N^{-a-b/2}
\]
after using \eqref{eq:gamma-condition}; rescaling
$(\xi,\eta)=(N\widetilde\xi,N^{1/2}\widetilde\eta)$ therefore gives
the asserted kernel bound.  On an $x$ chart,
\[
 |v_x|\simeq\max\{|\xi|,\eta^2\},
 \qquad
 |v_x|\gtrsim |\eta|N^{1/2}.
\]
In particular,
\[
 |\partial_\eta |v_x|^{-1/2}|
 \lesssim |\eta||v_x|^{-3/2}
 \lesssim N^{-1/2}|v_x|^{-1/2},
 \qquad
 |\partial_\xi |v_x|^{-1/2}|
 \lesssim N^{-1}|v_x|^{-1/2},
\]
and the same scaled bounds follow for higher derivatives.  To account
for the whole $x$ chart, split further into
$|\eta|\lesssim N^{1/2}$ and dyadic blocks $|\eta|\simeq\mu$ with
$\mu\ge N^{1/2}$.  The first block is handled by the parabolic
rescaling above.  On the latter blocks, rescaling by $(N,\mu)$ gives
an $L^1$ kernel bound multiplied by
\[
 N^{\gamma+1/2-\sigma}\frac{N^{1/2}}\mu
 \le \frac{N^{1/2}}\mu.
\]
These factors are geometrically summable over dyadic $\mu$.  This
proves the uniform $L^1$ kernel bound on every chart.  Convolution with
these kernels is therefore bounded on both mixed norms used below.

If
\[
 G_{\nu,N}
 =J^\sigma|v_\nu(D)|^{1/2}\chi_{\nu,N}(D)P_Nu,
\]
then the square partition $\sum_\nu\chi_{\nu,N}^2=1$ gives the exact
reconstruction
\begin{equation}\label{eq:product-chart-reconstruction}
 P_NJ^\gamma\partial_xu
 =\sum_\nu r_{\nu,N}(D)G_{\nu,N}.
\end{equation}
Moreover,
\begin{equation}\label{eq:product-chart-reduction}
 \norm{r_{\nu,N}(D)G_{\nu,N}}
 _{L_x^\infty L_{y,T}^2}
 \lesssim
 \norm{G_{\nu,N}}_{L_x^\infty L_{y,T}^2}
\end{equation}
for an $x$ region, with the analogous estimate after interchanging
$x$ and $y$ for a $y$ region.

For an $x$ region, mixed H\"older gives
\begin{align}
 &\norm{u_{<N/8}\,
 r_{\nu,N}(D)G_{\nu,N}}_{L^2_{t,x,y}}
 \notag\\
 &\qquad\le
 \norm{u_{<N/8}}_{L_x^2L_{y,T}^\infty}
 \norm{r_{\nu,N}(D)G_{\nu,N}}
 _{L_x^\infty L_{y,T}^2}.
 \label{eq:product-holder-x}
\end{align}
For a $y$ region we use
\begin{align}
 &\norm{u_{<N/8}\,
 r_{\nu,N}(D)G_{\nu,N}}_{L^2_{t,x,y}}
 \notag\\
 &\qquad\le
 \norm{u_{<N/8}}_{L_y^2L_{x,T}^\infty}
 \norm{r_{\nu,N}(D)G_{\nu,N}}
 _{L_y^\infty L_{x,T}^2}.
 \label{eq:product-holder-y}
\end{align}
The low-frequency projectors are bounded in the two maximal norms,
since their convolution kernels have uniformly bounded $L^1$ norm.
After summing in $N$ and in the finitely many microlocal regions,
Littlewood--Paley almost orthogonality, \eqref{eq:product-chart-reduction},
and \eqref{eq:S-def} yield
\begin{equation}\label{eq:product-low-high-bound}
 \norm{\Pi_\gamma(u)}_{L_T^2L^2}
 \lesssim M(T)S(T).
\end{equation}
Combining \eqref{eq:product-balanced} and
\eqref{eq:product-low-high-bound} proves \eqref{eq:product}.
\end{proof}

We now combine the energy estimate, the maximal function estimate,
the refined Strichartz estimate, and \cref{prop:smoothing}.

\begin{proposition}\label{prop:coupled}
Let $u$ be a smooth solution on $[0,T]$, $0<T\le1$.  Then
\begin{align}
 E_s(T)
 &\lesssim \norm{u_0}_{H^s}
 \exp\bigl(C T^{1/2}A(T)\bigr),
 \label{eq:coupled-E}\\
 M(T)
 &\lesssim_s \norm{u_0}_{H^s}
 +T^{1/2}\bigl(M(T)S(T)+E_s(T)A(T)\bigr),
 \label{eq:coupled-M}\\
 A(T)
 &\lesssim_s E_s(T)+M(T)S(T)+E_s(T)A(T),
 \label{eq:coupled-A}\\
 S(T)^2
 &\lesssim_s E_s(T)^2\bigl(1+T+T^{1/2}A(T)\bigr)
 +E_s(T)^3+TE_s(T)^4+TE_s(T)^5.
 \label{eq:coupled-S}
\end{align}
\end{proposition}

\begin{proof}
Estimate \eqref{eq:coupled-E} is the energy inequality of
\cref{lem:energy}, since
\[
 \int_0^T\norm{\nabla u(t)}_{L^\infty}\dd t
 \le T^{1/2}A(T).
\]

We next prove \eqref{eq:coupled-M}.  The retarded maximal estimate
applied to
\[
 u(t)=U(t)u_0-
 \int_0^tU(t-t')\bigl(uu_x\bigr)(t')\dd t'
\]
gives
\begin{equation}\label{eq:maximal-application}
 M(T)
 \lesssim \norm{u_0}_{H^s}
 +\norm{J^{1/2+\eps_0}(uu_x)}_{L_T^1L^2}.
\end{equation}
By Cauchy--Schwarz in time,
\begin{equation}\label{eq:maximal-time-cs}
 \norm{J^{1/2+\eps_0}(uu_x)}_{L_T^1L^2}
 \le T^{1/2}
 \norm{J^{1/2+\eps_0}(uu_x)}_{L_T^2L^2}.
\end{equation}
We apply \cref{prop:product} with
\[
 \gamma=\frac12+\eps_0,
 \qquad \sigma=s-\eps_0.
\]
The required inequality is
\[
 \gamma+\frac12=1+\eps_0<s-\eps_0,
\]
which follows from \eqref{eq:eps-choice}.  Substitution in
\eqref{eq:maximal-application}-\eqref{eq:maximal-time-cs} proves
\eqref{eq:coupled-M}.

We now use the refined Strichartz estimate.  Apply
\cref{prop:refined} to the equation
\begin{equation}\label{eq:refined-application-equation}
 \partial_tu+\mathcal Lu=-uu_x,
\end{equation}
where $\mathcal L$ is the linear BO--ZK operator, and take
$\eps=\eps_0$ in \eqref{eq:refined}.  We obtain
\begin{align}
 A(T)
 \lesssim{}&
 \norm{J^{19/16+\eps_0}u}_{L_T^\infty L^2}
 +\norm{J^{11/16+\eps_0}(uu_x)}_{L_T^2L^2}
 \notag\\
 &+\norm{u}_{L_T^\infty L^2}
 +\norm{uu_x}_{L_T^2L^2}.
 \label{eq:refined-applied}
\end{align}
The first and third terms are bounded by $E_s(T)$ because
$19/16+\eps_0<s$.  For the last term,
\begin{equation}\label{eq:forcing-low-order}
 \norm{uu_x}_{L_T^2L^2}
 \le \norm{u}_{L_T^\infty L^2}
 \norm{u_x}_{L_T^2L^\infty}
 \le E_s(T)A(T).
\end{equation}
Finally, apply \cref{prop:product} with
\begin{equation}\label{eq:refined-product-choice}
 \gamma=\frac{11}{16}+\eps_0,
 \qquad \sigma=s-\eps_0.
\end{equation}
The condition of that proposition is
\[
 \gamma+\frac12
 =\frac{19}{16}+\eps_0<s-\eps_0,
\]
and this again follows from \eqref{eq:eps-choice}.  Therefore
\begin{equation}\label{eq:refined-forcing-product}
 \norm{J^{11/16+\eps_0}(uu_x)}_{L_T^2L^2}
 \lesssim M(T)S(T)+E_s(T)A(T).
\end{equation}
Combining \eqref{eq:refined-applied},
\eqref{eq:forcing-low-order}, and
\eqref{eq:refined-forcing-product} proves \eqref{eq:coupled-A}.
This also shows explicitly where the exponents $19/16$ and $11/16$
from the refined Strichartz estimate enter the nonlinear argument.

Lastly, \eqref{eq:coupled-S} follows from \cref{prop:smoothing} with
$\sigma=s-\eps_0$.  The conditions $1<\sigma<s$ follow from
\eqref{eq:eps-choice}.
\end{proof}

\begin{proposition}\label{prop:small-data}
There exist constants $\delta_s>0$ and $C_s>0$ such that, if
\[
 \norm{u_0}_{H^s}\le\delta_s,
\]
then every smooth solution of \eqref{IVP} on $[0,T]$, $0<T\le1$,
satisfies
\begin{equation}\label{eq:small-bound}
 E_s(T)+M(T)+A(T)+S(T)
 \le C_s\norm{u_0}_{H^s}.
\end{equation}
Consequently, a smooth solution constructed by regularization extends
to the whole interval $[0,1]$.
\end{proposition}

\begin{proof}
Let $d_0=\norm{u_0}_{H^s}$ and put
\[
 \mathcal F(T)=E_s(T)+M(T)+A(T)+S(T).
\]
For smooth solutions, $\mathcal F$ is continuous and nondecreasing.
Moreover, $A(T)+S(T)\to0$ as $T\downarrow0$, while Bernstein's
inequality, or the homogeneous part of \eqref{eq:retarded-max}, gives
\[
 \limsup_{T\downarrow0}M(T)
 \lesssim_s\norm{u_0}_{H^s}.
\]
Consequently
$\limsup_{T\downarrow0}\mathcal F(T)\le C_sd_0$; after decreasing
$\delta_s$, the bootstrap set $\{T:\mathcal F(T)<1\}$ is nonempty.
Assume that $\mathcal F(T)\le1$.  By \eqref{eq:coupled-E},
\begin{equation}\label{eq:bootstrap-energy-small}
 E_s(T)\le C_sd_0.
\end{equation}
Taking $\delta_s$ smaller if necessary, \eqref{eq:coupled-S} and
\eqref{eq:bootstrap-energy-small} imply
\begin{equation}\label{eq:bootstrap-smoothing-small}
 S(T)\le C_sd_0.
\end{equation}
Indeed, under the bootstrap assumption all factors multiplying
$E_s(T)^2$ in \eqref{eq:coupled-S} are bounded, and the remaining
terms are of order at least $E_s(T)^3$.

Using \eqref{eq:bootstrap-energy-small} and
\eqref{eq:bootstrap-smoothing-small} in
\eqref{eq:coupled-M}-\eqref{eq:coupled-A}, we obtain
\begin{align*}
 M(T)&\le C_sd_0+C_sd_0\bigl(M(T)+A(T)\bigr),\\
 A(T)&\le C_sd_0+C_sd_0\bigl(M(T)+A(T)\bigr).
\end{align*}
Choose $\delta_s$ so that the last terms can be absorbed.  It follows
that
\[
 \mathcal F(T)\le C_sd_0
\]
whenever $\mathcal F(T)\le1$.  Thus the bootstrap set is also closed
in the interval of existence, and the usual open-and-closed
continuity argument proves
\eqref{eq:small-bound} on every interval on which the smooth solution
exists.  This uniform bound prevents blowup of the regularized Sobolev
norms, so the standard continuation criterion extends the solution to
$[0,1]$.
\end{proof}

\section{Proof of the main theorem}\label{sec:proof-main}

\subsection{Scaling and lifespan}

We observe from \eqref{eq:scaling-intro} that the BO--ZK scaling is
\begin{equation}\label{eq:scaling}
 u_\lambda(x,y,t)
 =\lambda u(\lambda x,\lambda^{1/2}y,\lambda^2t).
\end{equation}
For $0<\lambda\le1$,
\begin{equation}\label{eq:scaling-Hs}
 \norm{u_\lambda(0)}_{H^s}
 \le\lambda^{1/4}\norm{u_0}_{H^s}.
\end{equation}
Indeed, after changing variables in Fourier space,
\[
 \norm{u_\lambda(0)}_{H^s}^2
 =\lambda^{1/2}\int
 (1+\lambda^2\xi^2+\lambda\eta^2)^s
 |\widehat{u_0}(\xi,\eta)|^2\dd\xi\dd\eta.
\]
Choose $\lambda$ so that $\lambda^{1/4}\norm{u_0}_{H^s}\le\delta_s$.  The small-data result on the unit interval rescales to a lifespan $T=\lambda^2$, proving \eqref{eq:lifespan}.

\subsection{Existence}

Let $u_{0,n}=P_{\le n}u_0$.  Standard regularized energy estimates
give a smooth solution $u_n$ on a maximal interval; this construction
is also contained in the Sobolev theories of
\cite{CunhaPastorLow,Nascimento2020}.  After scaling to the small-data regime, \cref{prop:small-data} is uniform in $n$ and extends every $u_n$ to $[0,1]$.

For $w_{n,m}=u_n-u_m$,
\[
 (w_{n,m})_t+\mathcal H_x(w_{n,m})_{xx}+(w_{n,m})_{xyy}
 +u_n(w_{n,m})_x+w_{n,m}(u_m)_x=0.
\]
Taking the $L^2$ inner product with $w_{n,m}$ gives
\begin{equation}\label{eq:L2-difference}
 \norm{w_{n,m}(t)}_2
 \le \norm{u_{0,n}-u_{0,m}}_2
 \exp\left(C\int_0^t
 (\norm{(u_n)_x}_\infty+\norm{(u_m)_x}_\infty)\dd\tau\right).
\end{equation}
Thus $(u_n)$ is Cauchy in $C([0,1];L^2)$.  The uniform $L^\infty H^s$ bound, interpolation, and local compactness yield a limit
\[
 u\in C([0,1];H^{s'}(\R^2))
 \quad\text{for every }s'<s,
\]
which solves \eqref{IVP} in distributions.  After passing to a
subsequence, the uniform bounds give weak-* convergence in
$L_T^\infty H^s$ and in each mixed-norm space defining $A$, $M$, and
$S$.  These mixed spaces are normed by the corresponding K\"othe
preduals: $L_T^2L^1_{x,y}$ for $A$, $L_x^2L^1_{y,T}$ and
$L_y^2L^1_{x,T}$ for $M$, and $L_x^1L^2_{y,T}$ and
$L_y^1L^2_{x,T}$ for $S$.  Testing against these preduals and taking
the supremum yields weak-* lower semicontinuity of all the asserted
auxiliary norms.  The frequency-envelope argument below proves the
convergence and time continuity to $C([0,1];H^s)$.

\subsection{Frequency envelopes and strong
\texorpdfstring{$H^s$}{Hs} continuity}

We record the dyadic estimate used both for strong convergence of the
smooth approximations and for the Bona--Smith argument.

\begin{lemma}\label{lem:dyadic-energy}
Let $u$ be a smooth solution and put $u_N=P_Nu$.  Then, for every
dyadic $N\ge1$,
\begin{equation}\label{eq:dyadic-energy}
 \frac{\dd}{\dd t}\norm{u_N}_2
 \le C\norm{\nabla u}_\infty
 \left(
 \sum_{N/8\le M\le8N}\norm{u_M}_2
 +\sum_{M\ge8N}\frac NM\norm{u_M}_2
 \right).
\end{equation}
The usual interpretation is made when $N=1$.
\end{lemma}

\begin{proof}
The linear terms are skew-adjoint.  Decompose $uu_x$ by Bony's
paraproduct.  In the low--high interaction, after allowing finitely
many neighboring high frequencies, write
\[
 P_N(u_{<N/8}\partial_xu_N)
 =u_{<N/8}\partial_xu_N
 +[P_N,u_{<N/8}]\partial_xu_N.
\]
The first term is integrated by parts and the commutator kernel gives
\[
 \big|\ip{P_N(u_{<N/8}\partial_xu_N)}{u_N}\big|
 \lesssim \norm{\nabla u}_\infty
 \norm{u_N}_2\sum_{M\simeq N}\norm{u_M}_2.
\]
The balanced interactions satisfy the same estimate by Bernstein and
Coifman--Meyer theory.

For a high--high interaction with $M\simeq M'\ge8N$, use
$uu_x=\frac12\partial_x(u^2)$ to place the derivative on the output:
\begin{align*}
 \norm{P_N\partial_x(u_Mu_{M'})}_2
 &\lesssim N\norm{u_M}_\infty\norm{u_{M'}}_2\\
 &\lesssim \frac NM\norm{\nabla u}_\infty\norm{u_{M'}}_2.
\end{align*}
Pairing with $u_N$, summing the finite overlaps, and dividing by
$\norm{u_N}_2$ proves \eqref{eq:dyadic-energy}; the zero case follows
by regularization.
\end{proof}

Let $e_N(t)=N^s\norm{u_N(t)}_2$.  Multiplication of
\eqref{eq:dyadic-energy} by $N^s$ gives
\begin{equation}\label{eq:dyadic-weighted-energy}
 \frac{\dd}{\dd t}e_N(t)
 \le C\norm{\nabla u}_\infty
 \left(
 \sum_{M\simeq N}\left(\frac NM\right)^s e_M(t)
 +\sum_{M\ge8N}\left(\frac NM\right)^{s+1}e_M(t)
 \right).
\end{equation}

Choose $0<\kappa<\min\{1,s+1\}$.  An admissible frequency envelope
is a positive sequence $(c_N)$ such that
\begin{equation}\label{eq:envelope-slow}
 e_N(0)\le c_N,
 \qquad
 c_M\le c_N\max\left\{\left(\frac MN\right)^\kappa,
 \left(\frac NM\right)^\kappa\right\}.
\end{equation}
For example, one may take
\begin{equation}\label{eq:envelope-definition}
 c_N=\left(\sum_M
 2^{-2\kappa|\log_2N-\log_2M|}e_M(0)^2\right)^{1/2}.
\end{equation}
Then $\sum_Nc_N^2\simeq\norm{u_0}_{H^s}^2$.  The slow variation in
\eqref{eq:envelope-slow} is used quantitatively in
\eqref{eq:dyadic-weighted-energy}: both kernels satisfy
\begin{equation}\label{eq:envelope-kernel-sum}
 \sum_{M\simeq N}\left(\frac NM\right)^s c_M
 +\sum_{M\ge8N}\left(\frac NM\right)^{s+1}c_M
 \le C_{s,\kappa}c_N.
\end{equation}
Consequently, with $X(t)=\sup_N e_N(t)/c_N$, integration of
\eqref{eq:dyadic-weighted-energy} and
\eqref{eq:envelope-kernel-sum} yields
\begin{equation}\label{eq:envelope-propagation}
 \sup_N\frac{N^s\norm{P_Nu(t)}_2}{c_N}
 \le
 \exp\left(C\int_0^t\norm{\nabla u(\tau)}_\infty\dd\tau\right).
\end{equation}

Apply \eqref{eq:envelope-propagation} to the approximate solutions
$u_n$, using the envelope of $u_0$; it also dominates
$P_{\le n}u_0$.  Since $(c_N)\in\ell^2$,
\begin{equation}\label{eq:uniform-tail-approximations}
 \lim_{K\to\infty}\sup_n\sup_{t\le1}
 \norm{P_{>K}u_n(t)}_{H^s}=0.
\end{equation}
Every fixed low-frequency truncation converges in
$C([0,1];H^s)$ by the $L^2$ convergence and Bernstein's inequality.
Together with \eqref{eq:uniform-tail-approximations}, this proves
$u_n\to u$ in $C([0,1];H^s)$ and hence
$u\in C([0,1];H^s)$.

For later use with a family of initial data, it is convenient to
record a tail version that does not require a common pointwise
$\ell^2$ envelope.  Let $K_0=2^J$ and define
\begin{equation}\label{eq:tail-weight}
 d_{J,N}=\min\left\{1,\left(\frac N{K_0}\right)^\kappa\right\}.
\end{equation}
The ratio of two such weights is bounded by the right-hand side of
\eqref{eq:envelope-slow}.  Multiplying
\eqref{eq:dyadic-weighted-energy} by $d_{J,N}$, applying the
$\ell^2$ Schur test, and using Gronwall gives
\begin{equation}\label{eq:weighted-tail-propagation}
 \left(\sum_Nd_{J,N}^2N^{2s}\norm{P_Nu(t)}_2^2\right)^{1/2}
 \le C
 \exp\left(C\int_0^t\norm{\nabla u(\tau)}_\infty\dd\tau\right)
 \left(\sum_Nd_{J,N}^2N^{2s}\norm{P_Nu_0}_2^2\right)^{1/2}.
\end{equation}
Because $d_{J,N}=1$ for $N\ge K_0$, the left-hand side controls the
$H^s$ tail above $K_0$.  If a set of initial data is compact in $H^s$,
the right-hand side of \eqref{eq:weighted-tail-propagation} tends to
zero uniformly on that set as $J\to\infty$.  This is the uniform tail
estimate used in the continuous-dependence argument below.

\subsection{Uniqueness}

Let $u,v$ be two solutions in the class \eqref{eq:solutionclass}, and set $w=u-v$.  Then
\[
 w_t+\mathcal H_xw_{xx}+w_{xyy}+uw_x+wv_x=0.
\]
To justify the energy identity at this regularity, apply a Friedrichs
multiplier $P_{\le L_0}$ to the difference equation, pair with
$P_{\le L_0}w$, and use the skew-adjointness of the linear multiplier.
The transport commutators are bounded by
$C(\norm{u_x}_\infty+\norm{v_x}_\infty)\norm w_2^2$, uniformly in
$L_0$.  Passing to the limit $L_0\to\infty$ by dominated convergence
gives
\begin{equation}\label{eq:uniqueness}
 \frac{\dd}{\dd t}\norm w_2^2
 \lesssim(\norm{u_x}_\infty+\norm{v_x}_\infty)\norm w_2^2.
\end{equation}
Since $u_x,v_x\in L^1_TL^\infty$ by Cauchy--Schwarz in time, Gronwall proves uniqueness and Lipschitz dependence in $L^2$.

\subsection{Continuous dependence}

Let $u_{0,n}\to u_0$ in $H^s$ and let $u_n,u$ be the corresponding
solutions on a common interval determined by a fixed $H^s$ ball.  The
$L^2$ difference estimate gives convergence in $C_TL^2$.  The set
$\{u_0\}\cup\{u_{0,n}:n\ge1\}$ is compact in $H^s$; therefore
\eqref{eq:weighted-tail-propagation} makes the high-frequency tails
of $u_n$ and $u$ uniformly small.  For every fixed $K_0$, Bernstein's
inequality and the $L^2$ convergence give
\[
 \norm{P_{\le K_0}(u_n-u)}_{C([0,T];H^s)}
 \lesssim K_0^s\norm{u_n-u}_{C([0,T];L^2)}\longrightarrow0.
\]
Letting first $n\to\infty$ and then $K_0\to\infty$ proves
\[
 \norm{u_n-u}_{C([0,T];H^s)}\longrightarrow0.
\]
This is the Bona--Smith argument \cite{BonaSmith,CunhaPastorLow};
it establishes continuity, but not the higher differentiability of
the flow map ruled out in \cite{EsfahaniPastor}.  This completes the
proof of \cref{thm:main}.

\section{Concluding remarks}\label{sec:concluding}

The proof combines three features of the BO--ZK phase.  Product-frequency
kernel estimates quantify the loss of curvature in the parabolic region
and yield the mixed maximal estimate with frequency power \(N^{1/2}\).
The longitudinal and transverse group velocities do not vanish
simultaneously, so signed positive commutators recover local smoothing
on four complementary charts.  Finally, in the chart-transition region,
the normal low--high interaction is nonresonant by one full
high-frequency power; a cubic modification of the localized energy
converts this gain into a summable quartic remainder.

The threshold can be read directly from the two parameters in the
refined Strichartz argument. For an admissible exponent \(p>8/3\) and
a short-time scale \(N^{-\vartheta}\), the nonlinear closure requires
\[
 s>\max\left\{a_\vartheta(p),
 b_\vartheta(p)+\frac12\right\},
\]
where \(a_\vartheta(p)-b_\vartheta(p)=\vartheta\).  Since
\(a_\vartheta(p)\) is increasing and
\(b_\vartheta(p)+1/2\) is decreasing in \(\vartheta\), the two
constraints balance uniquely at \(\vartheta=1/2\).  At this scale their
common value is \(2-13/(6p)\), and therefore
\begin{equation}\label{eq:minmax}
 \inf_{\substack{p>8/3\\0\le\vartheta\le1}}
 \max\left\{a_\vartheta(p),
 b_\vartheta(p)+\frac12\right\}
 =\frac{19}{16}.
\end{equation}
The infimum is approached as \(p\downarrow8/3\); it is not attained
because the admissible range is strict. 

Equation \eqref{eq:minmax} evaluates the present scheme and is not a
sharpness statement.  The formal scaling index is \(s_c=-1/4\), while
no norm-inflation or failure-of-continuity result is known that provides
a comparable lower bound for isotropic well-posedness.  Possible routes
below \(19/16\) include a genuinely anisotropic resummation of the
product blocks, in the spirit of the maximal-function analysis for ZK
equations \cite{RibaudVento3D,LinaresRamos}; a short-time
Fourier-restriction framework in isotropic Sobolev spaces, coupled to a
modified energy \cite{IonescuKenigTataru,MolinetPilod,RibaudVento}; and
transversality estimates based on nonlinear Loomis--Whitney inequalities
\cite{BennettCarberyWright,Kinoshita2D,HerrKinoshita,KinoshitaSchippa}.
The last approach would have to accommodate the \(C^{1,1}\)
characteristic surface across \(\xi=0\).  Each route must still address
the low-longitudinal-frequency resonance that rules out a direct Picard
iteration.  Thus the remaining gap to scaling reflects the present
analytic method rather than a conjectured threshold.

\appendix

\section{Computational consistency checks}
\label{sec:computational}

A public computational notebook accompanies the manuscript and provides
reproducible symbolic and numerical consistency checks organized in the
same order as the analytical development. Exact symbolic calculations
check the phase derivatives, the Hessian determinant, the resonance
identity and its derivatives, the refined Strichartz exponents, and the
scaling arithmetic. Finite-dimensional Fourier calculations and
dealiased pseudospectral experiments examine selected energy,
commutator, kernel, local-smoothing, product, and frequency-envelope
relations over finite sampled ranges.

The kernel experiment tests the sampled dispersive profiles associated
with \cref{prop:kernel} and separately checks the scaling of the
nonstationary tail majorant used in its proof. In addition, a static
concentration experiment based on an annular Fourier profile measures
the ratio between the mixed maximal norm and the $L^2$ norm for a
dyadic family of rescaled functions. The resulting log-log slope is
$1/2$, reproducing the scaling mechanism described in
\cref{rem:sharpness} and illustrating the unavoidable $N^{1/2}$
frequency power in \cref{thm:maximal}.

These computations are intended solely as reproducibility,
falsification, and consistency checks. They are performed at finite
resolution, on finite computational domains, and over finite sets of
parameters. Consequently, they do not establish any of the analytic
estimates used in the article. Every mathematical statement required
for the proof is established independently in the preceding sections,
and no argument depends on the numerical output.

\section*{Data availability}

No datasets were generated or analyzed for this study. The
manuscript-synchronized Python notebook used for the symbolic and
numerical consistency checks will be publicly available after request.

\section*{Funding}

This research did not receive any specific grant from funding agencies
in the public, commercial, or not-for-profit sectors.

\section*{Declaration of competing interest}

The author declares that he has no known competing financial interests
or personal relationships that could have appeared to influence the
work reported in this paper.

\section*{Declaration of Generative AI and AI-assisted technologies in the writing process}

During the preparation of this work, the author used Claude
(Anthropic) and ChatGPT (OpenAI) to assist in organizing the exposition
of the frequency-adapted microlocal analysis, improving language and
clarity, and preparing, refactoring, and diagnostically reviewing code
for symbolic and numerical consistency checks.  An exchange with these
systems suggested organizing the local-smoothing argument through
frequency-adapted microlocal charts.  The author subsequently
formulated and proved the chart construction, maximal-function
estimate, normal-form correction, and all related estimates, and
independently verified every definition, statement, calculation,
reference, and numerical value.  After using these tools, the author
reviewed and edited the manuscript as needed and takes full
responsibility for the content of the article.  Generative AI systems
were not used as substitutes for mathematical proof or verification.


\begin{thebibliography}{99}

\bibitem{Benjamin}
T. B. Benjamin,
\emph{Internal waves of permanent form in fluids of great depth},
J. Fluid Mech. \textbf{29} (1967), 559--592.

\bibitem{BennettCarberyWright}
J. Bennett, A. Carbery, and J. Wright,
\emph{A non-linear generalisation of the Loomis--Whitney inequality and applications},
Math. Res. Lett. \textbf{12} (2005), 443--457.
\href{https://doi.org/10.4310/MRL.2005.v12.n4.a1}{doi:10.4310/MRL.2005.v12.n4.a1}.

\bibitem{BonaSmith}
J.~L. Bona and R. Smith,
\emph{The initial-value problem for the Korteweg--de Vries equation},
Philos. Trans. Roy. Soc. London Ser. A \textbf{278} (1975), 555--601.
\href{https://doi.org/10.1098/rsta.1975.0035}{doi:10.1098/rsta.1975.0035}.

\bibitem{CalderonVaillancourt}
A.~P. Calder\'on and R. Vaillancourt,
\emph{On the boundedness of pseudo-differential operators},
J. Math. Soc. Japan \textbf{23} (1971), no.~2, 374--378.
\href{https://doi.org/10.2969/jmsj/02320374}{doi:10.2969/jmsj/02320374}.

\bibitem{CunhaPastorWeighted}
A. Cunha and A. Pastor,
\emph{The IVP for the Benjamin--Ono--Zakharov--Kuznetsov equation in weighted Sobolev spaces},
J. Math. Anal. Appl. \textbf{417} (2014), 660--693.
\href{https://doi.org/10.1016/j.jmaa.2014.03.056}{doi:10.1016/j.jmaa.2014.03.056}.

\bibitem{CunhaPastorLow}
A. Cunha and A. Pastor,
\emph{The IVP for the Benjamin--Ono--Zakharov--Kuznetsov equation in low regularity Sobolev spaces},
J. Differential Equations \textbf{261} (2016), 2041--2067.
\href{https://doi.org/10.1016/j.jde.2016.04.022}{doi:10.1016/j.jde.2016.04.022}.

\bibitem{EsfahaniPastor}
A. Esfahani and A. Pastor,
\emph{Ill-posedness results for the (generalized) Benjamin--Ono--Zakharov--Kuznetsov equation},
Proc. Amer. Math. Soc. \textbf{139} (2011), 943--956.
\href{https://doi.org/10.1090/S0002-9939-2010-10532-4}{doi:10.1090/S0002-9939-2010-10532-4}.

\bibitem{EsfahaniPastorInstability}
A. Esfahani and A. Pastor,
\emph{Instability of solitary wave solutions for the generalized BO--ZK equation},
J. Differential Equations \textbf{247} (2009), 3181--3201.
\href{https://doi.org/10.1016/j.jde.2009.09.014}{doi:10.1016/j.jde.2009.09.014}.

\bibitem{EsfahaniPastorUCP}
A. Esfahani and A. Pastor,
\emph{On the unique continuation property for Kadomtsev--Petviashvili-I and Benjamin--Ono--Zakharov--Kuznetsov equations},
Bull. Lond. Math. Soc. \textbf{43} (2011), 1130--1140.
\href{https://doi.org/10.1112/blms/bdr048}{doi:10.1112/blms/bdr048}.

\bibitem{EsfahaniPastorSharp}
A. Esfahani and A. Pastor,
\emph{Sharp constant of an anisotropic Gagliardo--Nirenberg-type inequality and applications},
Bull. Braz. Math. Soc. (N.S.) \textbf{48} (2017), 171--185.
\href{https://doi.org/10.1007/s00574-016-0017-5}{doi:10.1007/s00574-016-0017-5}.

\bibitem{EsfahaniPastorBona}
A. Esfahani, A. Pastor, and J.~L. Bona,
\emph{Stability and decay properties of solitary-wave solutions for the generalized BO--ZK equation},
Adv. Differential Equations \textbf{20} (2015), 801--834.
\href{https://doi.org/10.57262/ade/1435064514}{doi:10.57262/ade/1435064514}.

\bibitem{GrafakosOh}
L. Grafakos and S. Oh,
\emph{The Kato--Ponce inequality},
Comm. Partial Differential Equations \textbf{39} (2014), 1128--1157.
\href{https://doi.org/10.1080/03605302.2013.822885}{doi:10.1080/03605302.2013.822885}.

\bibitem{Hormander}
L. H\"ormander,
\emph{The Analysis of Linear Partial Differential Operators I},
Springer-Verlag, Berlin, 1983.

\bibitem{HormanderIII}
L. H\"ormander,
\emph{The Analysis of Linear Partial Differential Operators III:
Pseudo-Differential Operators},
Grundlehren der mathematischen Wissenschaften, vol.~274,
Springer-Verlag, Berlin, 1985.
\href{https://doi.org/10.1007/978-3-540-49938-1}{doi:10.1007/978-3-540-49938-1}.

\bibitem{HerrKinoshita}
S. Herr and S. Kinoshita,
\emph{Subcritical well-posedness results for the Zakharov--Kuznetsov equation in dimension three and higher},
Ann. Inst. Fourier (Grenoble) \textbf{73} (2023), no.~3, 1203--1267.
\href{https://doi.org/10.5802/aif.3547}{doi:10.5802/aif.3547}.

\bibitem{IonescuKenigTataru}
A.~D. Ionescu, C.~E. Kenig, and D. Tataru,
\emph{Global well-posedness of the KP-I initial-value problem in the energy space},
Invent. Math. \textbf{173} (2008), 265--304.
\href{https://doi.org/10.1007/s00222-008-0115-0}{doi:10.1007/s00222-008-0115-0}.

\bibitem{KatoPonce}
T. Kato and G. Ponce,
\emph{Commutator estimates and the Euler and Navier--Stokes equations},
Comm. Pure Appl. Math. \textbf{41} (1988), 891--907.

\bibitem{KenigKPI}
C.~E. Kenig,
\emph{On the local and global well-posedness theory for the KP-I equation},
Ann. Inst. H. Poincar\'e Anal. Non Lin\'eaire \textbf{21} (2004),
no.~6, 827--838.
\href{https://doi.org/10.1016/j.anihpc.2003.12.002}{doi:10.1016/j.anihpc.2003.12.002}.

\bibitem{KenigKoenig}
C.~E. Kenig and K.~D. Koenig,
\emph{On the local well-posedness of the Benjamin--Ono and modified Benjamin--Ono equations},
Math. Res. Lett. \textbf{10} (2003), 879--895.
\href{https://doi.org/10.4310/MRL.2003.v10.n6.a13}{doi:10.4310/MRL.2003.v10.n6.a13}.

\bibitem{KenigZiesler}
C.~E. Kenig and S.~N. Ziesler,
\emph{Maximal function estimates with applications to a modified Kadomstev--Petviashvili equation},
Commun. Pure Appl. Anal. \textbf{4} (2005), 45--91.
\href{https://doi.org/10.3934/cpaa.2005.4.45}{doi:10.3934/cpaa.2005.4.45}.

\bibitem{Kinoshita2D}
S. Kinoshita,
\emph{Global well-posedness for the Cauchy problem of the Zakharov--Kuznetsov equation in 2D},
Ann. Inst. H. Poincar\'e C Anal. Non Lin\'eaire \textbf{38} (2021), no.~2, 451--505.
\href{https://doi.org/10.1016/j.anihpc.2020.08.003}{doi:10.1016/j.anihpc.2020.08.003}.

\bibitem{KinoshitaSchippa}
S. Kinoshita and R. Schippa,
\emph{Loomis--Whitney-type inequalities and low regularity well-posedness of the periodic Zakharov--Kuznetsov equation},
J. Funct. Anal. \textbf{280} (2021), no.~6, Paper No.~108904, 53~pp.
\href{https://doi.org/10.1016/j.jfa.2020.108904}{doi:10.1016/j.jfa.2020.108904}.

\bibitem{KochTzvetkov}
H. Koch and N. Tzvetkov,
\emph{Local well-posedness of the Benjamin--Ono equation in $H^s(\R)$},
Int. Math. Res. Not. \textbf{2003} (2003), no.~26, 1449--1464.
\href{https://doi.org/10.1155/S1073792803211260}{doi:10.1155/S1073792803211260}.

\bibitem{KohnNirenberg}
J.~J. Kohn and L. Nirenberg,
\emph{An algebra of pseudo-differential operators},
Comm. Pure Appl. Math. \textbf{18} (1965), 269--305.
\href{https://doi.org/10.1002/cpa.3160180121}{doi:10.1002/cpa.3160180121}.

\bibitem{Latorre}
J.~C. Latorre, A.~A. Minzoni, N.~F. Smyth, and C.~A. Vargas,
\emph{Evolution of Benjamin--Ono solitons in the presence of weak Zakharov--Kuznetsov lateral dispersion},
Chaos \textbf{16} (2006), 043103.
\href{https://doi.org/10.1063/1.2355555}{doi:10.1063/1.2355555}.

\bibitem{LinaresPilodSaut}
F. Linares, D. Pilod, and J.-C. Saut,
\emph{The Cauchy problem for the fractional Kadomtsev--Petviashvili equations},
SIAM J. Math. Anal. \textbf{50} (2018), no.~3, 3172--3209.
\href{https://doi.org/10.1137/17M1145379}{doi:10.1137/17M1145379}.

\bibitem{LinaresRamos}
F. Linares and J.~P.~G. Ramos,
\emph{Maximal function estimates and local well-posedness for the generalized Zakharov--Kuznetsov equation},
SIAM J. Math. Anal. \textbf{53} (2021), 914--936.
\href{https://doi.org/10.1137/20M1344524}{doi:10.1137/20M1344524}.

\bibitem{MolinetPilod}
L. Molinet and D. Pilod,
\emph{Bilinear Strichartz estimates for the Zakharov--Kuznetsov equation and applications},
Ann. Inst. H. Poincar\'e Anal. Non Lin\'eaire \textbf{32} (2015), no.~2, 347--371.
\href{https://doi.org/10.1016/j.anihpc.2013.12.003}{doi:10.1016/j.anihpc.2013.12.003}.

\bibitem{MolinetSautTzvetkov}
L. Molinet, J.-C. Saut, and N. Tzvetkov,
\emph{Ill-posedness issues for the Benjamin--Ono equation and related equations},
SIAM J. Math. Anal. \textbf{33} (2001), no.~4, 982--988.
\href{https://doi.org/10.1137/S0036141001385307}{doi:10.1137/S0036141001385307}.

\bibitem{NascimentoControl2026}
A.~C. Nascimento,
\emph{Stabilization and controllability for the Benjamin--Ono--Zakharov--Kuznetsov equation on $\mathbb T^2$},
preprint, 2026.

\bibitem{NascimentoDGBOZK2026}
A.~C. Nascimento,
\emph{Stabilization of dispersion-generalized Benjamin--Ono--Zakharov--Kuznetsov},
preprint, 2026.

\bibitem{Nascimento2020}
A.~C. Nascimento,
\emph{On special regularity properties of solutions of the Benjamin--Ono--Zakharov--Kuznetsov (BO--ZK) equation},
Commun. Pure Appl. Anal. \textbf{19} (2020), 4285--4325.
\href{https://doi.org/10.3934/cpaa.2020194}{doi:10.3934/cpaa.2020194}.

\bibitem{Ono}
H. Ono,
\emph{Algebraic solitary waves in stratified fluids},
J. Phys. Soc. Japan \textbf{39} (1975), 1082--1091.

\bibitem{RibaudVento}
F. Ribaud and S. Vento,
\emph{Local and global well-posedness results for the Benjamin--Ono--Zakharov--Kuznetsov equation},
Discrete Contin. Dyn. Syst. \textbf{37} (2017), 449--483.
\href{https://doi.org/10.3934/dcds.2017019}{doi:10.3934/dcds.2017019}.

\bibitem{RibaudVento3D}
F. Ribaud and S. Vento,
\emph{Well-posedness results for the three-dimensional Zakharov--Kuznetsov equation},
SIAM J. Math. Anal. \textbf{44} (2012), no.~4, 2289--2304.
\href{https://doi.org/10.1137/110850566}{doi:10.1137/110850566}.

\bibitem{Stein}
E.~M. Stein,
\emph{Harmonic Analysis: Real-Variable Methods, Orthogonality, and Oscillatory Integrals},
Princeton University Press, Princeton, NJ, 1993.

\bibitem{TaylorPDO}
M.~E. Taylor,
\emph{Pseudodifferential Operators and Nonlinear PDE},
Progress in Mathematics, vol.~100, Birkh\"auser, Boston, 1991.
\href{https://doi.org/10.1007/978-1-4612-0431-2}{doi:10.1007/978-1-4612-0431-2}.

\end{thebibliography}
\end{document}